\documentclass[12pt,a4paper]{amsart}

\usepackage[english]{babel}
\usepackage[utf8]{inputenc}
\usepackage{csquotes}
\usepackage[T1]{fontenc}
\usepackage{latexsym}
\usepackage[leqno]{amsmath}
\usepackage{amssymb,amsthm,amsfonts}
\usepackage{mathtools}
\usepackage{relsize}
\usepackage{framed, color}
\usepackage{dsfont}
\usepackage{todonotes}
\usepackage{cancel}
\usepackage[normalem]{ulem}
\usepackage{xcolor}

\usepackage{gauss} 
\usepackage{geometry}
\usepackage{stmaryrd}
\usepackage{nicefrac}

\usepackage{enumerate}
\usepackage[shortlabels]{enumitem}
\setlist[enumerate]{label=(\alph*),font=\normalshape}
\setlist[itemize]{font=\normalshape}

\makeatletter
\let\originalitem\item
\renewcommand{\item}[1][]{%
	\if\relax\detokenize{#1}\relax%
		\originalitem%
	\else%
		\originalitem[#1]%
		\phantomsection
		\def\@currentlabel{#1}
	\fi%
}
\makeatletter

\usepackage{float}

\usepackage{hyperref}

\usepackage[noabbrev,capitalize]{cleveref}

\usepackage{refcheck}

\makeatletter
\newcommand{\refcheckize}[1]{%
  \expandafter\let\csname @@\string#1\endcsname#1%
  \expandafter\DeclareRobustCommand\csname relax\string#1\endcsname[1]{%
    \csname @@\string#1\endcsname{##1}\wrtusdrf{##1}}%
  \expandafter\let\expandafter#1\csname relax\string#1\endcsname
}
\makeatother
\refcheckize{\cref}
\refcheckize{\Cref}

\let\originalleft\left
\let\originalright\right
\renewcommand{\left}{\mathopen{}\mathclose\bgroup\originalleft}
\renewcommand{\right}{\aftergroup\egroup\originalright}

\definecolor{darkgreen}{RGB}{0,70,60}
\definecolor{darkpink}{RGB}{80,8,60}

\definecolor{kitgreen}{RGB}{0,150,130}
\definecolor{kitblue}{RGB}{70,100,170}
\definecolor{kitmaygreen}{RGB}{140,182,60}
\definecolor{kityellow}{RGB}{252,229,0}
\definecolor{kitorange}{RGB}{223,155,27}
\definecolor{kitbrown}{RGB}{167,130,46}
\definecolor{kitred}{RGB}{162,34,35}
\definecolor{kitpurple}{RGB}{163,16,124}
\definecolor{kitcyanblue}{RGB}{35,161,224}

\allowdisplaybreaks

\theoremstyle{plain}
\newtheorem{theorem}{Theorem}[section]
\newtheorem{definition}[theorem]{Definition}
\newtheorem{proposition}[theorem]{Proposition} 
\newtheorem{lemma}[theorem]{Lemma} 

\theoremstyle{definition}
\newtheorem{remark}[theorem]{Remark}

\newcommand{\C}{\mathbb{C}} 
\newcommand{\R}{\mathbb{R}} 
 
\newcommand{\Z}{\mathbb{Z}} 
\newcommand{\Zodd}{\mathbb{Z}_\mathrm{odd}}
\newcommand{\Nodd}{\mathbb{N}_\mathrm{odd}}

\newcommand{\N}{\mathbb{N}} 
\newcommand{\T}{\mathbb{T}} 
\newcommand{\F}{\mathcal{F}}
\newcommand{\der}[1][]{\,#1\mathrm{d}}

\DeclareMathOperator*{\vspan}{span}
\DeclareMathOperator{\dist}{dist}
\DeclareMathOperator*{\supp}{supp}

\DeclareMathOperator*{\sign}{sign}
\DeclareMathOperator*{\esssup}{ess~sup}
\DeclareMathOperator*{\essinf}{ess~inf}

\DeclareMathOperator{\tr}{tr}

\renewcommand{\Re}{\operatorname{Re}}

\newcommand{\calB}{\mathcal{B}}

\newcommand{\bfD}{\mathbf{D}}
\newcommand{\bfE}{\mathbf{E}}
\newcommand{\bfB}{\mathbf{B}}
\newcommand{\bfH}{\mathbf{H}}
\newcommand{\bfx}{\mathbf{x}}
\newcommand{\bfN}{\mathbf{N}}
\newcommand{\bfP}{\mathbf{P}}

\newcommand{\landauO}{\mathcal{O}}
\newcommand{\landauo}{\hbox{o}}

\newcommand{\ee}{\mathrm{e}}
\newcommand{\ii}{\mathrm{i}}
\let\eps\varepsilon

\let\wto\rightharpoonup
\let\embeds\hookrightarrow
\DeclareMathAlphabet{\othermathbb}{U}{bbold}{m}{n}
\newcommand{\bbone}{\othermathbb{1}}

\newcommand{\vb}{\vec{b}\@ifnextchar{_}{}{\@ifnextchar{^}{\,}{\hspace{0.1em}}}}
\newcommand{\vv}{\vec{v}\@ifnextchar{_}{}{\@ifnextchar{^}{\,}{\hspace{0.1em}}}}
\newcommand{\vx}{\vec{x}\@ifnextchar{_}{}{\@ifnextchar{^}{\,}{\hspace{0.1em}}}}
\newcommand{\vy}{\vec{y}\@ifnextchar{_}{}{\@ifnextchar{^}{\,}{\hspace{0.1em}}}}

\newcommand{\set}[1]{\left\{ #1 \right\}}
\newcommand{\abs}[1]{\left\lvert #1 \right\rvert}
\newcommand{\Abs}[1]{\bigl\lvert #1 \bigr\rvert}
\newcommand{\nnorm}[1]{\left\vert\kern-0.25ex\left\vert\kern-0.25ex\left\vert #1 \right\vert\kern-0.25ex\right\vert\kern-0.25ex\right\vert}
\newcommand{\seminorm}[1]{\left[ #1 \right]}

\newcommand{\impvar}{\,\cdot\,}

\newcommand{\pdv}[3][]{%
	\if\relax\detokenize{#1}\relax%
	\frac{\partial #2}{\partial #3}%
	\else%
	\frac{\partial^{#1} #2}{\partial {#3}^{#1}}%
	\fi%
}
\newcommand{\dv}[3][]{%
	\if\relax\detokenize{#1}\relax%
	\frac{\mathrm{d} #2}{\mathrm{d} #3}%
	\else%
	\frac{\mathrm{d}^{#1} #2}{\mathrm{d} {#3}^{#1}}%
	\fi%
}

\makeatletter
\newlength{\negph@wd}
\newcommand{\negphantom}[1]{%
	\ifmmode
		\mathpalette\negph@math{#1}%
	\else
		\negph@do{#1}%
	\fi
}
\newcommand{\negph@math}[2]{\negph@do{$\m@th#1#2$}}
\newcommand{\negph@do}[1]{%
	\settowidth{\negph@wd}{#1}%
	\hspace*{-\negph@wd}%
}
\makeatother

\newcommand{\clonelabel}[2]{\@bsphack
	\expandafter\ifx\csname r@#2\endcsname\relax
	\else\protected@write\@auxout{}{\string\newlabel{#1}%
		{\csname r@#2\endcsname}}%
	\fi
	\expandafter\ifx\csname r@#2@cref\endcsname\relax
	\else\protected@write\@auxout{}{\string\newlabel{#1@cref}%
		{\csname r@#2@cref\endcsname}}%
	\fi
	\@esphack}
\newcommand\labelandcopy[2]{\expandafter\gdef\csname labeled:#1\endcsname{#2}\label{#1}#2}
\newcommand\paste[1]{\csname labeled:#1\endcsname\tag{\ref{#1}}}

\newcommand{\derf}[2][]{%
	\if\relax\detokenize{#1}\relax%
		\frac{\mathrm{d} #2}{\sqrt{2 \pi}}%
	\else%
		\frac{\mathrm{d} #2}{(2 \pi)^{\nicefrac{#1}{2}}}%
	\fi%
}

\newcommand{\pspace}{X}
\newcommand{\espace}{Y_N}
\newcommand{\espaceF}{Y_N^K}
\newcommand{\espacemult}{Y_{N,k_0}}
\newcommand{\espacetilde}{\widetilde Y_N}
\newcommand{\espaceins}{Y_{\Nins}}
\newcommand{\espaceret}{Y_{\Nav}}

\newcommand{\efct}{E}
\newcommand{\efctI}{E_I}
\newcommand{\efctB}{E_B}
\newcommand{\efctN}{E_N}
\newcommand{\emin}{E^\star}
\newcommand{\earg}{u^\star}
\newcommand{\efctF}{E \vert_{\espaceF}}
\newcommand{\eminF}{E^{K,\star}}
\newcommand{\eargF}{u^{K,\star}}

\newcommand{\Lrad}[2][2]{L^{#1}_{\mathrm{rad}}\if\relax\expandafter\detokenize{#2}\relax\else(#2)\fi}
\newcommand{\Lanti}[2][2]{L^{#1}_{\mathrm{anti}}\if\relax\expandafter\detokenize{#2}\relax\else(#2)\fi}
\newcommand{\Lradanti}[2][2]{L^{#1}_{\mathrm{rad},\mathrm{anti}}\if\relax\expandafter\detokenize{#2}\relax\else(#2)\fi}
\newcommand{\Hrad}[2][1]{H^{#1}_{\mathrm{rad}}\if\relax\expandafter\detokenize{#2}\relax\else(#2)\fi}
\newcommand{\Hradanti}[2][1]{H^{#1}_{\mathrm{rad},\mathrm{anti}}\if\relax\expandafter\detokenize{#2}\relax\else(#2)\fi}

\newcommand{\singularK}{\mathfrak{F}}
\newcommand{\regularK}{\mathfrak{R}}
\newcommand{\multK}{\mathfrak{K}}

\newcommand{\tildePhi}{\widetilde{\smash{\phi}\vphantom{ty}}}
\newcommand{\fracDT}[1]{\abs{\partial_t}^{#1}}
\newcommand{\halfDT}{\fracDT{\nicefrac12}}
\newcommand{\DT}{\fracDT{}}

\newcommand{\norm}[1]{\left\lVert \if\relax\expandafter\detokenize{#1}\relax\impvar\else #1\fi \right\rVert}
\newcommand{\Norm}[1]{\bigl\lVert \if\relax\expandafter\detokenize{#1}\relax\impvar\else #1\fi \bigr\rVert}

\newcommand{\Nins}{N_\mathrm{ins}}
\newcommand{\Nav}{N_\mathrm{av}}
\newcommand{\normN}[1]{\norm{#1}_N}
\newcommand{\normNins}[1]{\norm{#1}_{\Nins}}
\newcommand{\normNav}[1]{\norm{#1}_{\Nav}}
\newcommand{\normNtilde}[1]{\norm{#1}_N^\sim}
\newcommand{\qN}[1]{Q_N \if\relax\expandafter\detokenize{#1}\relax\else\left(#1\right)\fi}
\newcommand{\qNins}[1]{Q_{\Nins} \if\relax\expandafter\detokenize{#1}\relax\else\left(#1\right)\fi}
\newcommand{\qNav}[1]{Q_{\Nav} \if\relax\expandafter\detokenize{#1}\relax\else\left(#1\right)\fi}
\newcommand{\qNtilde}[1]{\widetilde Q_N \if\relax\expandafter\detokenize{#1}\relax\else\left(#1\right)\fi}

\norefnames
\nocitenames
\setoffmsgs

\usepackage[backend=biber, style = numeric]{biblatex} 
\bibliography{bibliography}

\makeatletter
\@namedef{subjclassname@2020}{\textup{2020} Mathematics Subject Classification}
\makeatother

\begin{document}

	\title[Existence of traveling breather solutions]{Existence of traveling breather solutions to cubic nonlinear Maxwell equations in waveguide geometries}
	
	\author{Sebastian Ohrem}
	\address{Institute for Analysis, Karlsruhe Institute of Technology (KIT), D-76128 Karlsruhe, Germany}\email{sebastian.ohrem@kit.edu}
	
	\author{Wolfgang Reichel}
	\address{Institute for Analysis, Karlsruhe Institute of Technology (KIT), D-76128 Karlsruhe, Germany}\email{wolfgang.reichel@kit.edu}

	\date{\today} 
	
	\subjclass[2020]{Primary: 35Q61, 49J10; Secondary: 35C07, 78A50}

    \keywords{Maxwell equations, nonlinear material law, polychromatic breather solutions, variational method}
	
	\begin{abstract}
		We consider the full set of Maxwell equations in a slab or cylindrical waveguide with a cubically nonlinear material law for the polarization of the electric field. The nonlinear polarization may be instantaneous or retarded, and we assume it to be confined inside the core of the waveguide. We prove existence of infinitely many spatially localized, real-valued and time-periodic solutions (breathers) propagating inside the waveguide by applying a variational minimization method to the resulting scalar quasilinear elliptic-hyperbolic equation for the profile of the breathers. The temporal period of the breathers has to be carefully chosen depending on the linear properties of the waveguide. As an example, our results apply if a two-layered linear axisymmetric waveguide is enhanced by a third core region with low refractive index where also the nonlinearity is located. In this case we can also connect our existence result with a bifurcation result. We illustrate our results with numerical simulations. Our solutions are polychromatic functions in general, but for some special models of retarded nonlinear material laws, also monochromatic solutions can exist. In this case the numerical simulations raise an interesting open question: are the breather solutions with minimal energy monochromatic or polychromatic?  	
	\end{abstract}
    
	\maketitle

\section{Introduction and exemplary results}

Our results show the existence of spatially localized, real-valued and time-periodic solutions (called breathers) to the full set of Maxwell's equations. We consider two types of waveguide geometries: the slab waveguide and the axially symmetric waveguide. Our breathers travel inside the waveguide and are periodic in the direction of travel. In the axially symmetric waveguide they decay to zero in all directions orthogonal to the waveguide, whereas in the slab waveguide they are independent of one direction orthogonal to the waveguide and decay to zero in the remaining direction. The nonlinear properties of the material are confined to the waveguide and are built according to a Kerr-law which may be instantaneous or retarded (temporally averaged). Before we summarize the main literature contributions we first introduce the physical problem. Towards the end of this introduction we comment on the physical consequences of our main theorems.

\medskip

As our underlying physical model we consider Maxwell's equations
\begin{align} \label{eq:maxwell}
    \begin{aligned}
	&\nabla \cdot \bfD = 0, 
	&&\nabla \times \bfE = - \bfB_t, 
	\\
	&\nabla \cdot \bfB = 0, 
	&&\nabla \times \bfH = \bfD_t,
	\end{aligned}
\end{align}
in the absence of charges and currents. Constitutive relations between the electric field $\bfE$ and electric displacement $\bfD$ as well as the magnetic field $\bfH$ and the magnetic induction $\bfB$ are formulated by the following material laws 
\begin{align} \label{eq:material}
	\bfD = \eps_0 \bfE + \bfP(\bfE),
	\qquad
	\bfB = \mu_0 \bfH,
\end{align}
where $\eps_0>0$ is the vacuum permittivity, $\mu_0>0$ the vacuum permeability and $c_0=\nicefrac{1}{\sqrt{\eps_0\mu_0}}$ the vacuum speed of light. The relation $\bfB = \mu_0\bfH$ reflects that the we assume no interaction of the magnetic field with the material. The interaction of the electric field with the material, however, is described by the polarization field $\bfP(\bfE)$ which we assume to take the form 
\begin{align} \label{eq:polarization}
	\bfP(\bfE) = \eps_0 \chi_1(\bfx) \bfE
	+ \eps_0 \chi_3(\bfx) \bfN(\bfE)
\end{align}
with $\bfx = (x, y, z)$ being the spatial variable, cf. \cite{agrawal, babin_figotin, maloney_newell}. Moreover, we assume that the cubic nonlinearity $\bfN(\bfE)$ is isotropic, of Kerr-type, and retarded (temporally averaged) of the form 
\begin{align}\label{eq:nonlinear_polarization}
	\bfN(\bfE)(\bfx, t) = \int_0^\infty \tilde \kappa(\tau) \abs{\bfE(\bfx, t-\tau)}^2 \der \tau \,\bfE(\bfx, t) 	
\end{align}
which includes the case of an instantaneous nonlinearity 
\begin{align}\label{eq:nonlin:instantaneous}
	\bfN(\bfE) = \abs{\bfE}^2 \bfE
\end{align}
if we allow $\tilde\kappa=\delta_0$ to be the delta-distribution supported at time $0$. A physical discussion of these material laws is given in \cite{butcher_cotter, Fabrizio_Morro, maloney_newell} where also higher-order dependencies and anisotropy are discussed. Since we are looking for time-periodic fields $\bfE(\bfx,t+T)=\bfE(\bfx,t)$ with period $T>0$, the nonlinearity may be re-written as
\begin{equation}
\label{eq:nonlin:retarded}
\bfN(\bfE)(\bfx, t) = \frac{1}{T}\int_0^T \kappa(\tau) \abs{\bfE(\bfx, t-\tau)}^2 \der \tau \,\bfE(\bfx, t) =\left(\kappa\ast \abs{\bfE(\bfx,\cdot)}^2\right)(t) \bfE(\bfx,t)
\end{equation}
with the $T$-periodic function $\kappa(\tau)\coloneqq T \sum_{k\in\Z} \tilde\kappa(\tau+kT)$ and where we understand $\tilde\kappa\mid_{(-\infty,0)}\equiv 0$. Moreover we have used the convolution notation $(\kappa\ast v)(t) = \tfrac1T \int_0^T \kappa(\tau) v(t-\tau) \der \tau$ for the weighted temporal average of a measurable function $v$ (which still includes the instantaneous case where $\kappa=\delta_0^\text{per}$). From these equations we obtain the following second-order quasilinear equation for the electric field $\bfE$:
\begin{align}\label{eq:second_order_maxwell}
	\nabla \times \nabla \times \bfE + \eps_0 \mu_0 \bigl( \left( 1 + \chi_1(\bfx)\right) \bfE + \chi_3(\bfx) \bfN(\bfE) \bigr)_{tt}
	= 0.
\end{align}
We will show as part of our results how to recover the full set of Maxwell's equations from \eqref{eq:second_order_maxwell}. Under suitable assumptions on the convolution kernel $\kappa$, cf. \eqref{eq:ass:kappa}, we will show that \eqref{eq:second_order_maxwell} has a variational structure. Examples are given in Section~\ref{sec:examples}.

\medskip

We are interested in breather solutions of \eqref{eq:second_order_maxwell} which are moving with speed $c\in (0,c_0)$. Our results depend on the choice of the coefficients $\chi_1$, $\chi_3$, the retardation function $\kappa$, the propagation speed $c$, and the desired period $T > 0$. We denote the frequency associated to $T$ by $\omega \coloneqq \frac{2 \pi}{T}$. 

\medskip

In the literature there are several treatments of the existence of breather solutions of \eqref{eq:second_order_maxwell}. The first sequence of papers deals with monochromatic breathers, i.e., breathers of the type $\bfE(\bfx,t)=E(\bfx)\cos(k_0\omega t+t_0)$. Such breathers are not compatible with the instantaneous nonlinearity but with the retarded nonlinearity, e.g., in the case $\kappa(t)=1$ which may occur when $\tilde\kappa(t)=\sum_{n\in \N_0} \alpha_k\bbone_{[nT, (n+1)T)}$ with $\alpha_n \geq 0$, $\sum_{n=0}^\infty \alpha_n=T^{-1}$. Monochromatic breathers have the advantage that \eqref{eq:second_order_maxwell} reduces to the stationary elliptic problem 
\begin{equation} \label{curl_curl}
	\nabla \times \nabla \times E - \eps_0 \mu_0 k_0^2 \omega^2 \bigl( \left( 1 + \chi_1(\bfx)\right) E + \frac{\chi_3(\bfx)}{2} |E|^2 \bigr) E 
	= 0.
\end{equation}
Instead of a cubic nonlinearity $\frac{\chi_3(\bfx)}{2} |E|^2 E$, also saturated nonlinearities $g(\bfx,|E|^2)E$ with a bounded function $g$ naturally appear. The cases of saturated nonlinearities were first elaborated by Stuart et al. \cite{McLeod92,Stuart04,Stuart91,StuartZhou03,StuartZhou96,StuartZhou10,StuartZhou05,StuartZhou01} in the case of traveling breathers in an axisymmetric waveguide. Using divergence free, TE- or TM-polarized ansatz fields, \eqref{curl_curl} was reduced to a one-dimensional nonlinear elliptic problem which can, e.g., be solved variationally. In the follow-up result \cite{Mederski_Reichel} the assumption of strict axisymmetry is dropped and more general two-dimensional waveguide profiles are considered, also allowing pure power nonlinearities. The case of standing monochromatic breathers also originates from Stuart's work and leads to the elliptic nonlinear curl-curl problem \eqref{curl_curl} in the vector-valued case. First works \cite{azzolini_et_al, bartsch_et_al, benci_fortunato} considered axisymmetric divergence free ansatz functions, which allowed to reduce $\nabla\times\nabla\times$ to $-\Delta$. Using Helmholtz decomposition and suitable profile decompositions for Palais-Smale sequences, this restriction has been overcome by Mederski et al. \cite{MederskiENZ,MederskiSchino22,MederskiSchinoSzulkin20}, see also the survey \cite{BartschMederskiSurvey} and references therein, with the isotropic cubic Kerr-nonlinearity still being left as an open problem. A different approach using limiting absorption principles \cite{mandel_lap} or dual variational approaches was carried out by Mandel~\cite{mandel_dual}, cf. also \cite{mandel_dual_nonlocal} where a spatially nonlocal variant of the stationary curl-curl problem was solved. Still within the area of monochromatic breathers, Dohnal et al. considered in \cite{dohnal2} breathers at interfaces between (lossy) metals and dielectrics including retardation and in \cite{dohnal1} they rigorously approximated breathers in photonic crystals when the frequency parameter is near a band edge.

In the second, much smaller sequence of papers, truly polychromatic breathers are considered for instantaneous nonlinearities. The first approach which we are aware of, is \cite{PelSimWein} where spatially localized
traveling wave solutions of the 1+1-dimensional version of the quasi-linear Maxwell problem \eqref{eq:second_order_maxwell} were investigated. The authors treat the case where the linear coefficient $\chi_1$ is a periodic arrangement of delta potentials. Using local bifurcation methods the authors solve a related system
which is homotopically linked to the Maxwell problem written as an infinite coupled system arising from a multiple scale ansatz. It is analytically not clear whether the bifurcation branch ever reaches the original Maxwell system but numerical results support the existence of spatially localized traveling waves. A fully rigorous treatment for the existence of breathers on finite large time scales was given in \cite{DohnalSchnaubeltTietz} for a set-up of Kerr-nonlinear dielectrics occupying two different halfspaces. Two further rigorous treatments of exact polychromatic breather solutions occurred in \cite{bruell_idzik_reichel} and \cite{kohler_reichel} where either the linear or the nonlinear coefficients take the form of delta-distributions  and the existence of travelling breathers was accomplished by using bifurcation theory and variational methods, respectively. We are not aware of any treatment of polychromatic breathers in the presence of retarded nonlinearities.

\subsection{Examples of our results} \label{sec:examples}
We first describe our results on the level of examples. General results will be given in \cref{sec:main_results}. Breather solutions are rare phenomena, and hence the fact that our examples contain rather specific assumptions on the material coefficients and do not leave much leeway for perturbations should not be surprising. The main difference to the previous results may be summarized as follows: while we allow both instantaneous and retarded material laws, our traveling breather solutions are generally polychromatic and hence not limited to monochromatic ansatz functions. Moreover, our solutions satisfy the full set of Maxwell's equations exactly, the material coefficients $\chi_1, \chi_3$ are bounded, and our solutions can be numerically approximated with little effort. 

In the following, the speed of light is assumed to be $1=\nicefrac{1}{\sqrt{\epsilon_0\mu_0}}$.
Breather solutions will be time-periodic with period $T$ and are propagating along the $z$-axis with speed $c\in (0,1)$.  We consider two geometries for breathers: the cylindrical geometry where $\chi_1(\bfx)=\tilde\chi_1(r), \chi_3(\bfx)=\tilde\chi_3(r)$ only depend on $r=\sqrt{x^2+y^2}$, and the slab geometry where $\chi_1(\bfx)=\tilde\chi_1(x), \chi_3(\bfx)=\tilde\chi_3(x)$ only depend on $x$. In the cylindrical geometry we consider electric fields of the form
$$
\bfE(\bfx,t) = W(r, t - \tfrac{1}{c}z) \cdot (-\tfrac{y}{r},\tfrac{x}{r},0)^\top
$$
and in the slab geometry the electric field takes the form
$$
\bfE(\bfx, t) = \bigr(0, W(x, t - \tfrac{1}{c} z),0\bigl)^\top
$$
where in both settings $W$ is a real-valued profile which is localized in the first variable ($r$-direction in the cylindrical case and $x$-direction in the slab case) and $T$-periodic in the second variable. In both geometries the electric field is a divergence-free TE-mode which means that $\bfE$ is orthogonal to the direction of propagation.

\begin{definition} \label{def:weak_maxwell} The fields $\bfD, \bfE, \bfB, \bfH\in L^1_\mathrm{loc}(\R^3\times\R;\R^3)$ weakly solve Maxwell's equations provided
\begin{align*}
		&\int_{\R^4} \bfD \cdot \nabla \phi \der (\bfx, t) = 0,
		&& \int_{\R^4} \bfE \cdot \nabla \times \Phi \der (\bfx, t) = \int_{\R^4} \bfB \cdot \partial_t \Phi \der (\bfx, t),
		\\
		&\int_{\R^4} \bfB \cdot \nabla \phi \der (\bfx, t) = 0,
		&& \int_{\R^4} \bfH \cdot \nabla \times \Phi \der (\bfx, t) = - \int_{\R^4} \bfD \cdot \partial_t \Phi \der (\bfx, t)
	\end{align*}
	holds for all $\phi \in C_{\mathrm c}^\infty(\R^4; \R)$ and $\Phi \in C_{\mathrm c}^\infty(\R^4; \R^3)$. 
\end{definition}

The following theorem can be read as an explicit recipe for the construction of materials which support breathers. For the kernel $\tilde\kappa$ we generally assume \eqref{eq:ass:kappa}. Explicit examples include, e.g., $\tilde\kappa(t)=\bbone_{[0,\infty)}\left(T^4 + 4 t^4\right)^{-1} t$ or $\tilde\kappa(t)= \sum_{n \in \N_0} \alpha_n \bbone_{[n T, (n+1)T]}(t)$ where $\alpha_n\geq 0$ with $\sum_{n \in \N_0} \alpha_n = T^{-1}$, cf. \cref{rem:examples:kappa} for details. The material coefficients $\tilde\chi_1$, $\tilde\chi_3$ are assumed fixed and positive and take the form 
	\begin{align*}
		\tilde\chi_1(x) = \begin{cases}
			d, &\abs{x} < R, \\
			\tilde\chi_1^\ast(\abs{x} - R), & \abs{x} > R,
		\end{cases} \qquad
		\tilde\chi_3(x) = \begin{cases}
			- \gamma, & \abs{x} < R, \\
			0, & \abs{x} > R
		\end{cases}
	\end{align*}
where either $\tilde\chi_1^\ast=\tilde\chi_1^\mathrm{per}:\R\to\R$ is a $P$-periodic function defined on one periodicity cell by
	\begin{align*}
		\tilde\chi_1^\mathrm{per}(x) = \begin{cases}
			a &\abs{x} < \tfrac12 \theta P, \\
			b, &\tfrac12 \theta P < \abs{x} < \tfrac12 P
		\end{cases}
	\end{align*}
or $\tilde\chi_1^\ast=\tilde\chi_1^\mathrm{step}:\R\to\R$ is a step function defined by
    \begin{align*}
		\tilde\chi_1^\mathrm{step}(x) = \begin{cases}
			a, & \abs{x} < \rho, \\
			b, & \abs{x} > \rho
		\end{cases}
	\end{align*}
	with $a,b,d,P,R,\gamma,\rho>0, \theta\in (0,1)$.

	We are also using a sign-dependent distance function for a point $p\in\R$ and a set $M\subseteq \R$:
	$$
	\dist^+(p, M) = \inf\{ d^+(p,m): m\in M\} \mbox{ with } d^+(p,m)=\left\{ \begin{array}{ll} 
	|p-m| & \mbox{ if } m\geq p, \vspace{\jot} \\ 
	 \infty & \mbox{ if } m<p. 
	 \end{array} \right.
	$$

\begin{theorem} \label{thm:main_example} Suppose that the nonlinearity $N$ is given by either \eqref{eq:nonlin:instantaneous} or by \eqref{eq:nonlin:retarded} where $\kappa$ satisfies \eqref{eq:ass:kappa}. Then there exists a (nonzero) $T$-periodic real-valued weak solution of the Maxwell problem \eqref{eq:maxwell}, \eqref{eq:material}, \eqref{eq:polarization} in the sense of \cref{def:weak_maxwell} both for the slab and the cylindrical case\footnote{In the cylindrical case, write $r$ instead of $x$, and restrict $\tilde\chi_1, \tilde\chi_3$ to the half-line $[0,\infty)$.}, and for the following two choices of the polarization coefficient $\tilde\chi_1^\ast$:
\begin{itemize}
\item[(i)]
	If $\tilde\chi_1^\ast=\tilde\chi_1^\mathrm{per}$ then we assume that the propagation speed $c\in (0,1)$ is chosen such that $0<d<c^{-2}-1<\min\left\{a, b, \frac{a + d}{2}\right\}$ and  
	\begin{align}\label{eq:thm:periodic_example:coeff_ass}
		 \frac{\sqrt{a+1-c^{-2}}\cdot\theta}{\sqrt{b+1-c^{-2}}\cdot (1-\theta)} = \frac{m}{n}\in\frac{\Nodd}{\Nodd}
	\end{align}
	and define 
	$$
	T \coloneqq \frac{4\sqrt{a+1-c^{-2}}\theta P}{m} = \frac{4\sqrt{b+1-c^{-2}}(1-\theta) P}{n}. 
	$$
\item[(ii)]
	If $\tilde\chi_1^\ast=\tilde\chi_1^\mathrm{step}$ then we assume that the propagation speed $c\in (0,1)$ is chosen such that $0<\min\{b,d\}\leq\max\{b,d\} < c^{-2}-1<a$. Moreover, there are $m, n \in \N$ coprime with 
	\begin{align}\label{eq:thm:step_example:coeff_ass}
	 0 < \xi <\arctan\sqrt{\tfrac{a+1-c^{-2}}{-d-1+c^{-2}}} \text{ where } \xi \coloneqq \dist^+\left(\arctan\sqrt{\tfrac{a+1-c^{-2}}{-b-1+c^{-2}}}, \frac{m \pi}{2 n} + \frac{\pi}{n} \Z\right)
	\end{align}
	and 
	$$
    T \coloneqq 4\sqrt{a+1-c^{-2}}\rho\frac{n}{m}.
	$$
\end{itemize}	
	Additionally, the solution has at all times finite and uniformly bounded electromagnetic energy per unit square in $y,z$ (slab case), or per unit segment in $z$ (cylindrical case). 
\end{theorem}

\subsection{Discussion of the examples} \label{sec:discussion_examples}
Let us explain the reason behind the particular choices of the coefficients in a physical context. The parameters $a, b$ are properties of a linear waveguide (without any nonlinear effect) whose profile is given either by the purely periodic profile $\tilde\chi_1^\mathrm{per}$ or the pure step profile $\tilde\chi_1^\mathrm{step}$. Then the conditions on $a, b, c$ have the nature of a nonresonance condition, i.e., there are no guided waves $\bfE(\bfx,t) = \tilde W(r) \ee^{ik\omega(t-\nicefrac{z}{c})}\cdot (-\tfrac{y}{r},\tfrac{x}{r},0)^\top$ with time period $T=\tfrac{2\pi}{\omega}$ propagating with speed $c$ along the linear waveguide. Mathematically, this is expressed by a property of the operator $(1+\chi_1(\bfx))^{-1} \cdot \nabla\times\nabla\times$ appearing in \eqref{eq:second_order_maxwell}: namely all multiples $k^2\omega^2$ with $k\in\Zodd$ are required to stay away from the spectrum of this weighted operator when restricted to suitable TE-modes propagating with speed $c$ along the waveguide. This requirement is quite restrictive and its fulfillment can be guaranteed if $\omega=\tfrac{2\pi}{T}$ is chosen in the particular way and the parameters $a, b, c$ satisfy either \eqref{eq:thm:periodic_example:coeff_ass} or \eqref{eq:thm:step_example:coeff_ass}. 

The remaining conditions on $d$ may be described as follows: by inserting a new material of width $2R$ at the center of the waveguide the purely periodic or pure step waveguide is perturbed. On the linear level the new material has a low refractive index $d$ and on the level of the nonlinear refractive index it contributes a defocusing effect. The quantitative strength of the nonlinear effect plays no role in the sense that $\gamma>0$ may be arbitrary small. The value $d$ always satisfies a two-sided condition: on one hand $0<d<c^{-2}-1$ and on the other hand 
$$
c^{-2}-1 < \frac{a+d}{2} \quad \mbox{ or } \quad 0<\xi<\arctan\sqrt{\tfrac{a+1-c^{-2}}{-d-1+c^{-2}}}.
$$
We note that these conditions are always satisfied if $d$ is below but sufficiently close to $c^{-2}-1$. On the linear level, the presence of the new (linear) material at the core of the waveguide still does not allow for guided waves of time period $T$ and wave speed $c$. However, at a different value $d_\ast<d$ such a linear guided mode exists. Moreover, for all values $\tilde d\in (d_\ast,d)$ a solution of the nonlinear equation \eqref{eq:second_order_maxwell} exists, which bifurcates from $0$ as $\tilde d\to d_\ast$. In other words, the solution of Theorem~\ref{thm:main_example} is part of a bifurcation phenomenon with $d$ as a bifurcation parameter. In a nutshell: the nonlinear equation allows for guided modes in the waveguide at parameter values for which there are no linear guided modes. We comment on this phenomenon in Section~\ref{sec:further_estimates}.

\subsection{Outline of paper}
In \cref{sec:main_results} we state the general form of our results (\cref{thm:radial:main} and \cref{thm:slab:main}) of which Theorem~\ref{thm:main_example} is a special case. For particular choices of the parameters compatible with \cref{thm:main_example}, illustrations of approximate breathers can be found at the end of \cref{sec:main_results}. Our main results are stated both for the cylindrical geometry and the slab geometry. For the proofs we discuss in detail only the cylindrical geometry, as the slab geometry can be treated similarly with less difficulties. Sections~\ref{sec:domain_restriction}--\ref{sec:main_proof} contain the proof of our main results. In \cref{sec:domain_restriction} we show how the problem \eqref{eq:second_order_maxwell} on $\R^3\times\R$ can be reduced to a problem on the bounded domain $[0,R]\times[0,T]$. We then treat this reduced problem using a simple variational minimization method. In \cref{sec:approximation} we study a regularization of the bounded domain problem and in this way obtain an improved regularity result for the solutions of both the regularized and the original problem. \cref{sec:main_proof} closes the proof of the main results. Adaptations for the slab geometry are discussed in \cref{sec:slab_modifications}. Moreover, in \cref{sec:further_estimates} we show the further regularity result that $\|\bfE\|_{L^\infty(\supp \chi_3; L^2([0,T]))}$ is finite and we explain what this has to do with the dielectric character of the waveguide. Finally, in the same section, we comment on the bifurcation phenomenon w.r.t. the parameter $d$.

The appendices contain important technical tools. In \cref{sec:fractional} we prove some auxiliary results on the fractional Laplacian as well as a version of the famous Kenig-Ponce-Vega inequality on the torus. In \cref{sec:functional_av} we show a basic convexity result for our variational approach, lower bounds on integrated versions of the nonlinearity, and two trace inequalities. Then, in \cref{sec:calc_examples} we verify that the examples given in \cref{thm:main_example} satisfy the conditions of the general existence results.
Lastly, \cref{sec:numerics} details the numerical methods used to obtain approximations to the breather solutions that appear in the following section in the images in \cref{fig:periodic:breathers,fig:step:breathers} as illustrations of \cref{thm:main_example}.

Let us finish this introduction by pointing out some observations and open questions, cf. \cref{sec:main_results} for details. In all our results we allow the breathers to be a polychromatic superposition of Fourier modes of arbitrary multiples of the basic frequency $\omega$. In case of an instantaneous nonlinearity, necessarily infinitely many Fourier modes are non-zero. For time-averaged nonlinearities there is the possibility of monochromatic breathers and indeed (under suitable assumptions on $\kappa$) such monochromatic breathers exists. As our numerical simulations suggest, they appear to be more smooth than their polychromatic counterparts, and moreover, for the slab geometry, it seems that only monochromatic ground states exist. This is not the case for the cylindrical geometry. These findings based on numerical observations are analytically still open, but they do shed new light onto the a-priori choice of a monochromatic ansatz by  Stuart et al.~\cite{McLeod92,Stuart04,Stuart91,StuartZhou03,StuartZhou96,StuartZhou10,StuartZhou05,StuartZhou01} and later by others \cite{azzolini_et_al, BartschMederskiSurvey, bartsch_et_al, benci_fortunato, MederskiENZ,MederskiSchino22,MederskiSchinoSzulkin20}.

\section{Main results and numerical illustration} \label{sec:main_results}

After having given examples we now state our main results in more general form. We divide this into two subsections: one for the cylindrical geometry and one for the slab geometry. We define $\T \coloneqq \R \slash_{T \Z}$ as the torus of length $T$ which is our time domain equipped with the measure $\!\der t = \!\der[\frac{1}{T}]\lambda$ where $\!\der\lambda$ is the Lebesgue measure on $[0, T]$.

\subsection{Cylindrical geometry}
First we consider a cylindrical material, where $\chi_1(\bfx) = \tilde\chi_1(r)$ and $\chi_3(\bfx) = \tilde\chi_3(r)$ with $r \coloneqq \sqrt{x^2 + y^2}$. For $\bfE$ we consider a wave which is radial in the $(x, y)$-directions, travels with speed $c > 0$ in $z$-direction, and which has the form 
\begin{align}\label{eq:radial:ansatz}
	\bfE(\bfx, t) = w_t(r, t - \tfrac{1}{c}z) \cdot (-\tfrac{y}{r},\tfrac{x}{r},0)^\top
\end{align}
with a real-valued profile function $w_t(r, t)$. Inserting the ansatz \eqref{eq:radial:ansatz} into \eqref{eq:second_order_maxwell} and integrating once w.r.t. $t$ yields
\begin{align}\label{eq:radial:problem}
	- w_{rr} - \tfrac{1}{r} w_r + \tfrac{1}{r^2} w + \bigl(\tilde\chi_1(r)+1-c^{-2}\bigr) w_{tt} + \tilde\chi_3(r) N(w_t)_t = 0,
	\qquad r \in [0, \infty), t \in \R
\end{align}
with 
\begin{align}\label{eq:def:N:local}
	N(w_t) = \Nins(w_t) = w_t^3
\end{align}
or
\begin{align}\label{eq:def:N:averaged}
	N(w_t) = \Nav(w_t) = (\kappa\ast w_t^2) w_t
\end{align} 
corresponding to \eqref{eq:nonlin:instantaneous} and \eqref{eq:nonlin:retarded}, respectively. If the nonlinearity is given by \eqref{eq:def:N:averaged}, we require $\kappa$ to satisfy the following assumptions:
	\begin{align}\label{eq:ass:kappa}
		\begin{cases}
			\kappa \in C^\alpha(\T) \text{ for some } \alpha > 0, \\
			\kappa(t) = \kappa(-t) > 0 \text{ for } t \in \T, \\
			L^4(\T) \to \R, v \mapsto \int_\T (\kappa \ast v^2) v^2 \der t \text{ is convex}	
		\end{cases}
	\end{align}
where the convexity assumption is satisfied if, e.g., $\max \kappa\leq 2 \min \kappa$ or if the Fourier transform of $\kappa$ is non-negative, cf. \cref{lem:convexity_conditions} and \cref{rem:examples:kappa} for further concrete examples. In the following, $N$ will always denote either $\Nins$ or $\Nav$.  Under assumptions~\eqref{eq:ass:kappa} on $\kappa$, we will show that \eqref{eq:radial:problem} has a variational structure that is crucial in our study.

\medskip

In the context of radial symmetry it is important to see the relation between a radially symmetric function $f_\sharp: \R^2\setminus B_R(0)\to \R$, $R\geq 0$, and its radial profile function $f \colon [R, \infty) \to \R$ via the map $f_\sharp \colon \R^2\setminus B_R(0) \to \R, (x, y) \mapsto f(\sqrt{x^2+y^2})$. For $1\leq p<\infty$ this gives rise to the function spaces
\begin{align*}
	\Lrad[p]{[R, \infty)}
	\coloneqq \set{f \in L^1_\mathrm{loc}((R, \infty)) \colon f_\sharp \in L^p(\R^2\setminus B_R(0))}
\end{align*}
with norm
\begin{align*}
	\norm{f}_{L^p_\mathrm{rad}([R, \infty))} 
	\coloneqq \frac{1}{\sqrt[p]{2 \pi}}\Norm{f_\sharp}_{L^p(\R^2\setminus B_R(0))}
	= \norm{f}_{L^p([R, \infty), \der[r] r)}
\end{align*}
For functions depending on radius and time we define
\begin{align*}
	\Lrad[p]{[R, \infty) \times \T}
	\coloneqq \set{f \in L^1_\mathrm{loc}((R, \infty) \times \T) \colon f_\sharp \in L^p(\R^2\setminus B_R(0) \times \R)}.
\end{align*}
Other spaces of radially symmetric functions based on $\Lrad[2]{[R, \infty) \times \T}$, such as $\Hrad[k]{[R, \infty) \times \T}$, are defined analogously.

\medskip

For time-periodic functions $w\colon[0,\infty)\times\T \to \C$ we consider the temporal Fourier transform $\F$ and denote for $k\in \Z$ the $k$-th Fourier coefficient of $w$ by $\hat w_k = \F_k[w] = \int_\T w \overline{e_k} \der t$ where $e_k(t) \coloneqq \ee^{\ii k \omega t}$. For the linear part of the differential equation \eqref{eq:radial:problem}
\begin{align*}
	L w = - w_{rr} - \tfrac{1}{r} w_r + \tfrac{1}{r^2} w + \bigl(\tilde\chi_1+1-c^{-2}\bigr) w_{tt}
\end{align*}
we can apply the Fourier transform and obtain $\F_k [L w] = L_k \hat w_k$ with
\begin{align*}
	L_k \coloneqq - \partial_r^2 - \tfrac{1}{r} \partial_r + \tfrac{1}{r^2} - k^2 \omega^2 \bigl(\tilde\chi_1+1-c^{-2}\bigr).
\end{align*}

\medskip

We make the following assumptions on the nonlinearity $N$, the potentials $\tilde\chi_1, \tilde\chi_3$ and the operators $L_k$. Denote by $\Nodd \coloneqq 2\N-1 = \set{1, 3, 5, \dots}$. 
\begin{enumerate}
	\item[(A1)]\label{ass:radial:VandGamma}\label{ass:radial:first} 
	$\tilde\chi_1, \tilde\chi_3 \in L^\infty([0,\infty), \R)$ and $\supp(\tilde\chi_3) = [0, R]$ where $R > 0$.

	\item[(A2)]\label{ass:radial:N} 
	$N$ is given either by \eqref{eq:def:N:local}, or by \eqref{eq:def:N:averaged} where $\kappa$ satisfies \eqref{eq:ass:kappa}.
	
	\item[(A3)]\label{ass:radial:elliptic} 
	$\esssup_{[0, R]} \tilde\chi_1 \leq c^{-2}-1$, $\esssup_{[0, R]} \tilde\chi_3 < 0$.

	\item[(A4)]\label{ass:radial:fundamental_solutions} 
	There exists a solution $\phi_k \in \Hrad[2]{[R, \infty)} \setminus \set{0}$ of $L_k \phi_k = 0$ for each $k \in \Nodd$.

	\item[(A5)]\label{ass:radial:fundamental_solution:estimates}\label{ass:radial:secondlast}
	The following inequalities hold for $\phi_k$, $k \in \Nodd$:
	\begin{align*}
		\liminf_{k \to \infty} \frac{\abs{\phi_k(R)}}{\norm{\phi_k}_{\Lrad{[R, \infty)}}} > 0, 
		\qquad 
		\sup_{k} \frac{\abs{\phi_k'(R)}}{k \norm{\phi_k}_{\Lrad{[R, \infty)}}} < \infty.
	\end{align*}

	\item[(A6)]\label{ass:radial:nontrivial_sol}\label{ass:radial:last} 
	With $I_\alpha$ denoting the modified Bessel function of first kind, there exists $k_0 \in \Nodd$ such that $\phi_{k_0}(R) \neq 0$ and the following inequality holds:
	\begin{align*}
		\frac{\phi_{k_0}'(R)}{\phi_{k_0}(R)} > \frac{\lambda k_0 I_1'(\lambda k_0 R)}{I_1(\lambda k_0 R)}
		\quad\text{where}\quad
		\lambda \coloneqq \omega \bigl(c^{-2}-1-\essinf_{[0, R]}\tilde\chi_1\bigr)^{\nicefrac12}.
	\end{align*}
\end{enumerate}

We call $\phi_k$ a fundamental solution for $L_k$. Since $L_k = L_{-k}$ we define $\phi_{-k} \coloneqq \phi_k$ for all $k \in \Nodd$. The reason for considering $k\in \Nodd$ instead of $k \in \N_0$ is that $\ker(L_0) = \vspan \set{r, \tfrac{1}{r}}$ does not contain nonzero $\Lrad{[R, \infty)}$-functions. The restriction to $\Nodd$ amounts to considering $\nicefrac T2$-antiperiodic functions which is compatible with the cubic nonlinearity in \eqref{eq:radial:problem}.

Assumption~\ref{ass:radial:nontrivial_sol} is in place to ensure existence of nontrivial solutions to \eqref{eq:radial:problem}. Since $\frac{I_1'(z)}{I_1(z)}\to 1$ as $z \to \infty$ (see \cite{gradshteyn}), a sufficient condition for \ref{ass:radial:nontrivial_sol} to hold is 
\begin{enumerate}
	\item[(A6')]\label{ass:radial:nontrivial_sol:alt} 
	$\displaystyle \limsup_{k \to \infty} \frac{\phi_k'(R)}{k \phi_k(R)} > \omega \bigl(c^{-2}-1-\essinf_{[0, R]}\tilde\chi_1\bigr)^{\nicefrac12}$,
\end{enumerate}
which additionally ensures that \ref{ass:radial:nontrivial_sol} holds for infinitely many $k_0$.

Next we state our main theorem for the cylindrical geometry. 
\begin{theorem}\label{thm:radial:main}
	Assume \ref{ass:radial:first}--\ref{ass:radial:last} hold for given $N, \kappa, \tilde\chi_1, \tilde\chi_3$ and $T$. Then there exists a (nonzero) $T$-periodic real-valued weak solution of the Maxwell problem \eqref{eq:maxwell}, \eqref{eq:material}, \eqref{eq:polarization} in the sense of \cref{def:weak_maxwell}. Furthermore, localization orthogonal to the direction of propagation is expressed by the fact that at all times $t_0\in\R$ the electromagnetic energy per unit segment along the $z$-direction
	\begin{align*}
		\int_{\R \times \R \times [z_0, z_0 + 1]} \bigl(\bfD \cdot \bfE + \bfB \cdot \bfH\bigr) \der(x, y, z)
	\end{align*}
	is finite for all $z_0 \in \R$ and uniformly bounded.
\end{theorem}

\begin{remark} Let us explain why our assumptions \ref{ass:radial:elliptic}, \ref{ass:radial:nontrivial_sol} enforce $\tilde\chi_1$ to take values both below and above $c^{-2}-1$. Suppose for contradiction that $\tilde\chi_1\leq c^{-2}-1$ everywhere on $[0,\infty)$. If $w$ is a weak $T$-periodic solution to \eqref{eq:radial:problem}, we see that $w = 0$ must hold by multiplying \eqref{eq:radial:problem} with $w$ and integrating on $[0,\infty)\times\T$ with respect to the measure $r\der r \der t$. Hence, non-trivial solutions do not exist. In fact, the assumption~\ref{ass:radial:nontrivial_sol} conflicts with $\tilde \chi_1(r)\leq c^{-2}-1$ everywhere on $[0,\infty)$. Namely, in this case $\phi_k$ satisfies $(r\phi_k')' = (\frac{1}{r}+ rk^2\omega^2\bigl(c^{-2}-1-\tilde\chi_1(r)\bigr)\phi_k$. Multiplication with $\phi_k$ and integration from $R$ to $\infty$ yields $R\phi_k'(R)\phi_k(R) = -\int_R^\infty r|\phi_k'|^2+ (\frac{1}{r}+ rk^2\omega^2\bigl(c^{-2}-1-\tilde\chi_1(r)\bigr)|\phi_k|^2 \,dr\leq 0$. Thus, $\phi_k'(R)$ and $\phi_k(R)$ have opposite sign, contradicting \ref{ass:radial:nontrivial_sol} and the fact that $I_1, I_1'$ are positive on $(0,\infty)$.
\end{remark}

We end this subsection with a multiplicity result. For this, we first explain what kind multiplicity we consider. Given a solution $w$ of \eqref{eq:radial:problem}, any time-shift $(x, t) \mapsto w(x, t + \tau)$ for $\tau \in \T$ also solves \eqref{eq:radial:problem}. Moreover, if $N = \Nav$ with $\kappa \equiv 1$ one can shift the individual frequencies separately, i.e. $(x, t) \mapsto \sum_{k \in \Z} \hat w_k(x) e_k(t + \tau_k)$ solves \eqref{eq:radial:problem} for all $\tau_k \in \T$ with $\tau_k = \tau_{-k}$. By \emph{distinct} solutions we mean solutions that are not shifts of one another.

\begin{theorem}\label{thm:radial:multiplicity}
	Assume \ref{ass:radial:first}--\ref{ass:radial:secondlast} hold for given $N, \kappa, \tilde\chi_1, \tilde\chi_3, T$. If \ref{ass:radial:nontrivial_sol} holds for infinitely many $k_0 \in \Nodd$ (e.g. if \ref{ass:radial:nontrivial_sol:alt} is true) then there exist infinitely many distinct $T$-periodic real-valued weak solutions of the Maxwell problem \eqref{eq:maxwell}, \eqref{eq:material}, \eqref{eq:polarization} in the sense of \cref{def:weak_maxwell} with finite and uniformly bounded electromagnetic energy per unit segment along the $z$-direction.
\end{theorem}

\subsection{Slab geometry}
In our second setting, we consider slab materials that extend infinitely in the $(y,z)$-directions. Here $\chi_1(\bfx) = \tilde \chi_1(x)$, $\chi_3(\bfx) = \tilde \chi_3(x)$ and we look for traveling polarized waves moving at speed $c > 0$ in $y$-direction and being constant along the $z$-direction.
More precisely, we consider fields $\bfE$ given by the ansatz
\begin{align}\label{eq:slab:ansatz}
	\bfE(\bfx, t) = \bigr(0, 0, w_t(x, t - \tfrac{1}{c} y)\bigl)^\top.
\end{align}
Inserting into \eqref{eq:second_order_maxwell} and integrating once w.r.t. $t$ leads to the equation 
\begin{align}\label{eq:slab:problem}
	- w_{xx} + \left(\tilde \chi_1(x)+1-c^{-2}\right) w_{tt} + \tilde\chi_3(x) N(w_t)_t = 0
\end{align}
for the profile function $w_t(x, t)$. Similar to the radial setting we define the operators
\begin{align*}
	\widetilde L \coloneqq -\partial_x^2 + \left(\tilde \chi_1(x)+1-c^{-2}\right) \partial_t^2, 
	\qquad
	\widetilde L_k \coloneqq -\partial_x^2 - k^2 \omega^2 \left(\tilde \chi_1(x)+1-c^{-2}\right),
\end{align*}
so that $\F_k \widetilde L = \widetilde L_k \F_k$ holds for the temporal Fourier transform $\F$. We require the following assumptions on $\tilde\chi_1, \tilde\chi_3$ and $\widetilde L_k$: 

\begin{enumerate}
	\item[(\~A1)]\label{ass:slab:VandGamma}\label{ass:slab:first} 
	$\tilde\chi_1, \tilde\chi_3\in L^\infty(\R, \R)$ are even with $\supp(\tilde\chi_3) = [-R, R]$ where $R > 0$.

	\item[(\~A2)]\label{ass:slab:N}
	$N$ is given either by \eqref{eq:def:N:local}, or by \eqref{eq:def:N:averaged} where $\kappa$ satisfies \eqref{eq:ass:kappa}.
	
	\item[(\~A3)]\label{ass:slab:elliptic} 
	$\esssup_{[-R, R]} \tilde\chi_1 \leq c^{-2}-1$, $\esssup_{[-R, R]} \chi_3 < 0$.
	
	\item[(\~A4)]\label{ass:slab:fundamental_solutions} 
	There exists a solution $\tildePhi_k \in H^2([R, \infty)) \setminus \set{0}$ of $\widetilde L_k \tildePhi_k = 0$ for each $k \in \Nodd$.

	\item[(\~A5)]\label{ass:slab:fundamental_solutions:estimates}\label{ass:slab:secondlast}
	The following inequalities hold for $\tildePhi_k$, $k \in \Nodd$:
	\begin{align*}
		\liminf_{k \to \infty} \frac{\Abs{\tildePhi_k(R)}}{\Norm{\tildePhi_k}_{L^2([R, \infty))}} > 0, 
		\qquad 
		\sup_{k} \frac{\Abs{\tildePhi_k'(R)}}{k \Norm{\tildePhi_k}_{L^2([R, \infty))}} < \infty.
	\end{align*}

	\item[(\~A6)]\label{ass:slab:nontrivial_sol}\label{ass:slab:last}
	There exists $k_0 \in \Nodd$ such that $\tildePhi_{k_0}(R) \neq 0$ and the following inequality holds:
	\begin{align*}
		\frac{\tildePhi_{k_0}'(R)}{\tildePhi_{k_0}(R)} > \lambda k_0 \tanh\left( \lambda k_0 R \right)
		\quad\text{with}\quad
		\lambda \coloneqq \omega\bigl(c^{-2}-1-\essinf_{[-R, R]}\tilde\chi_1\bigr)^{\nicefrac12}.
	\end{align*}
\end{enumerate}
Again, a sufficient condition for \ref{ass:slab:nontrivial_sol} to hold is
\begin{enumerate}
	\item[(\~A6')] \label{ass:slab:nontrivial_sol:alt} 
	$\displaystyle \limsup_{k \to \infty} \frac{\tildePhi_k'(R)}{k \tildePhi_k(R)} > \omega \bigl(c^{-2}-1-\essinf_{[-R, R]}\tilde\chi_1\bigr)^{\nicefrac12}$.
\end{enumerate}

We can now formulate our main theorems for the slab geometry.
\begin{theorem}\label{thm:slab:main}
	Assume \ref{ass:slab:first}--\ref{ass:slab:last} hold for given $N, \tilde\chi_1, \tilde\chi_3$ and $T$. Then there exists a (nonzero) $T$-periodic real-valued weak solution of the Maxwell problem \eqref{eq:maxwell}, \eqref{eq:material}, \eqref{eq:polarization} in the sense of \cref{def:weak_maxwell}. Furthermore, localization in the $x$-direction is expressed by the fact that at all times $t_0\in\R$ the electromagnetic energy per unit square in the $y,z$-direction
	\begin{align*}
		\int_{\R \times [y_0,y_0+1]\times [z_0, z_0 + 1]} \bigl(\bfD \cdot \bfE + \bfB \cdot \bfH\bigr) \der(x, y, z)
	\end{align*}
	is finite for all $y_0, z_0 \in \R$ and uniformly bounded w.r.t. $t_0, z_0$.
\end{theorem}

\begin{theorem}\label{thm:slab:multiplicity}
	Assume \ref{ass:slab:first}--\ref{ass:slab:secondlast} hold for given $N, \tilde\chi_1, \tilde\chi_3, T$. If \ref{ass:radial:nontrivial_sol} holds for infinitely many $k_0 \in \Nodd$ (e.g. if \ref{ass:radial:nontrivial_sol:alt} is true) then there exist infinitely many distinct $T$-periodic real-valued weak solutions of the Maxwell problem \eqref{eq:maxwell}, \eqref{eq:material}, \eqref{eq:polarization} in the sense of \cref{def:weak_maxwell} with finite and uniformly bounded electromagnetic energy per unit square along the $y,z$-direction.
\end{theorem}

\subsection{Numerical illustrations, discussion, and some open questions}

In the following we apply the numerical scheme outlined in \cref{sec:numerics} and show results for the profile $w_t$ of the electric field, cf. \eqref{eq:radial:ansatz} or \eqref{eq:slab:ansatz}. The breathers we obtain analytically are ground states in the sense that they are minimizers of the energy functional $E$ discussed in \cref{sec:domain_restriction}. Here we show approximations to these ground states. We consider particular potentials $\tilde \chi_1$ and $\tilde \chi_1$ which are compatible with the parameter choices of Theorem~\ref{thm:main_example}. For the periodic case $\tilde \chi_1^\ast = \tilde \chi_1^{\mathrm{per}}$ we show in \cref{fig:periodic:breathers} four images which cover both choices of the nonlinearity (time-averaged and instantaneous) and both choices of the geometry (cylindrical and slab). For the step case $\tilde \chi_1^\ast = \tilde \chi_1^{\mathrm{step}}$ also four images covering both types of nonlinearities and both types of geometries are shown in \cref{fig:step:breathers}.

\begin{figure}
	\centering

	\includegraphics[width=.81\linewidth]{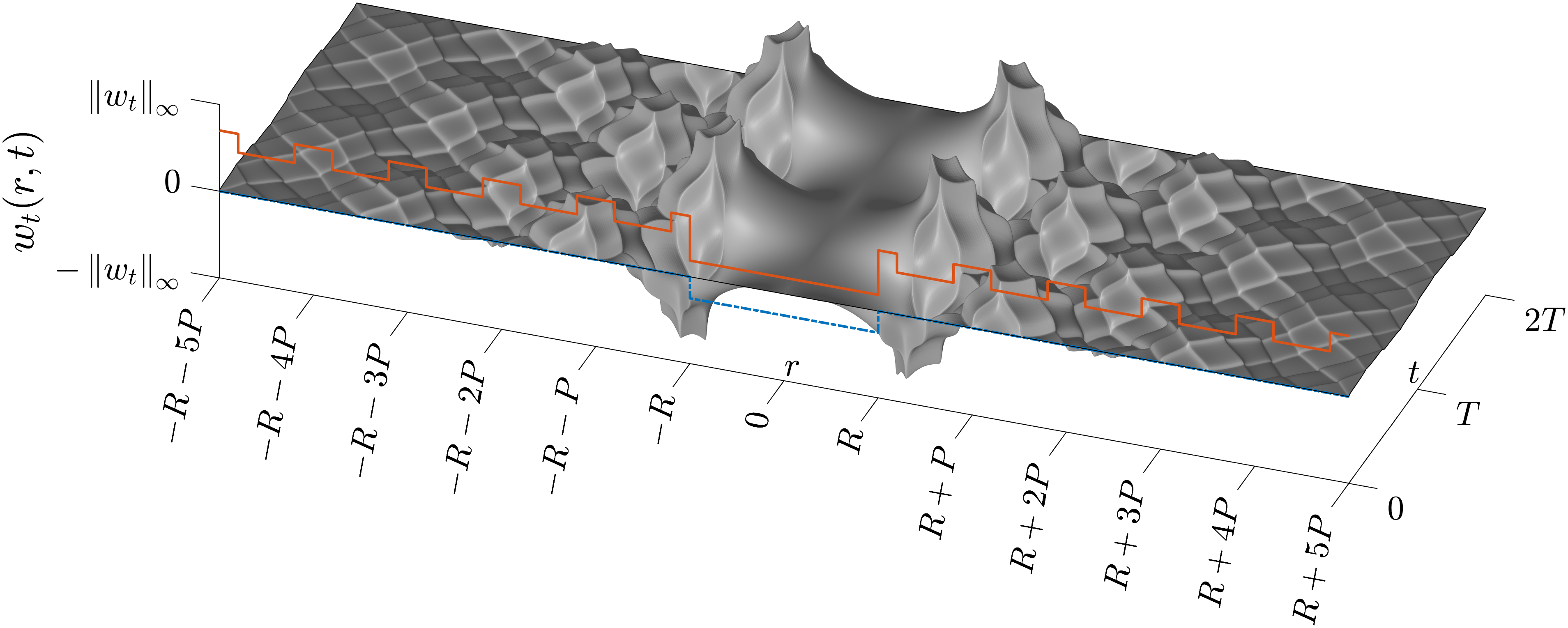}	
	\vspace{-4ex}

	\includegraphics[width=.81\linewidth]{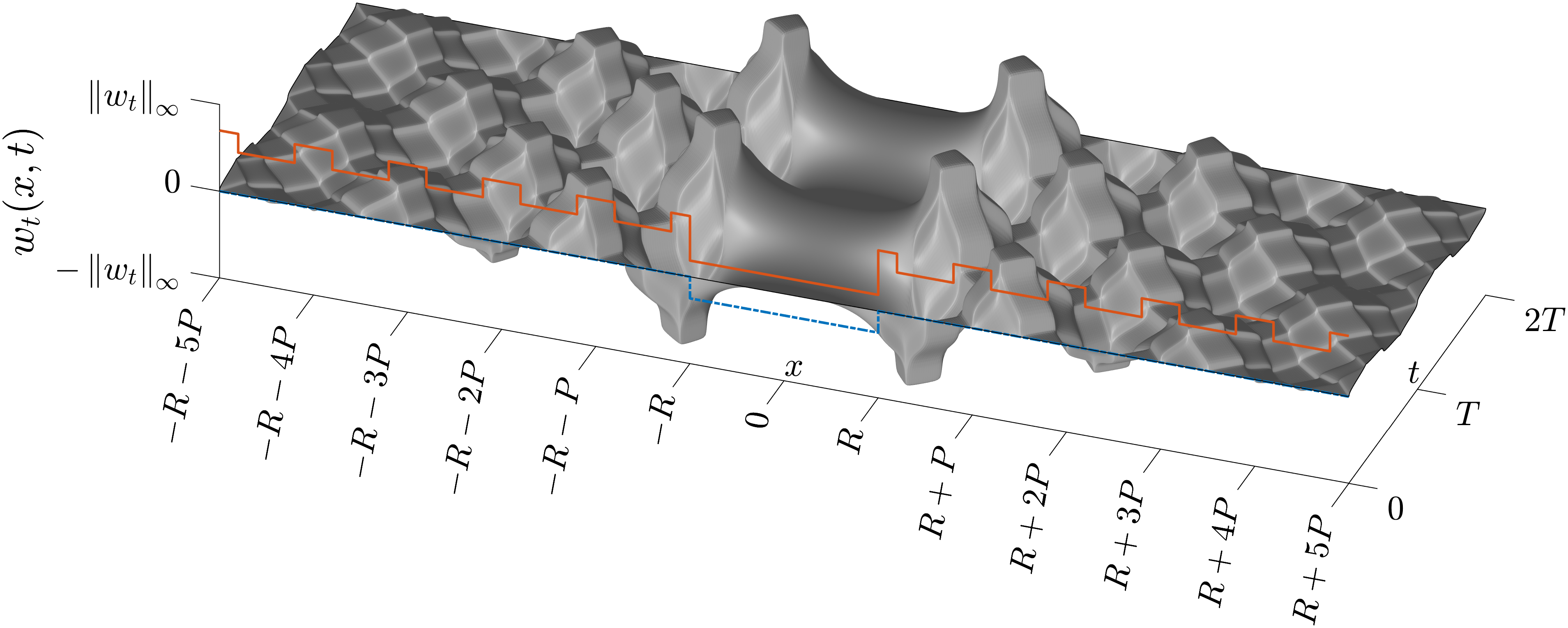}	
	\vspace{-4ex}
	
	\includegraphics[width=.81\linewidth]{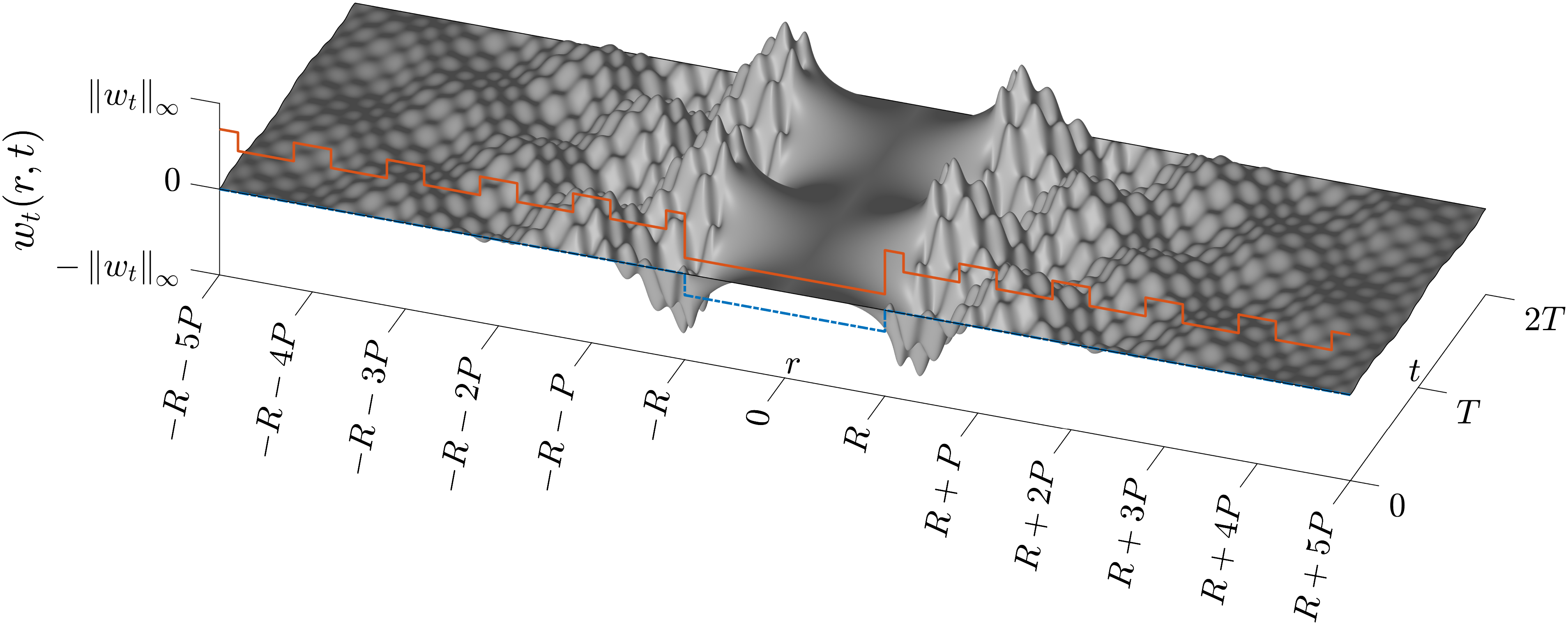}	
	\vspace{-5ex}

	\includegraphics[width=.81\linewidth]{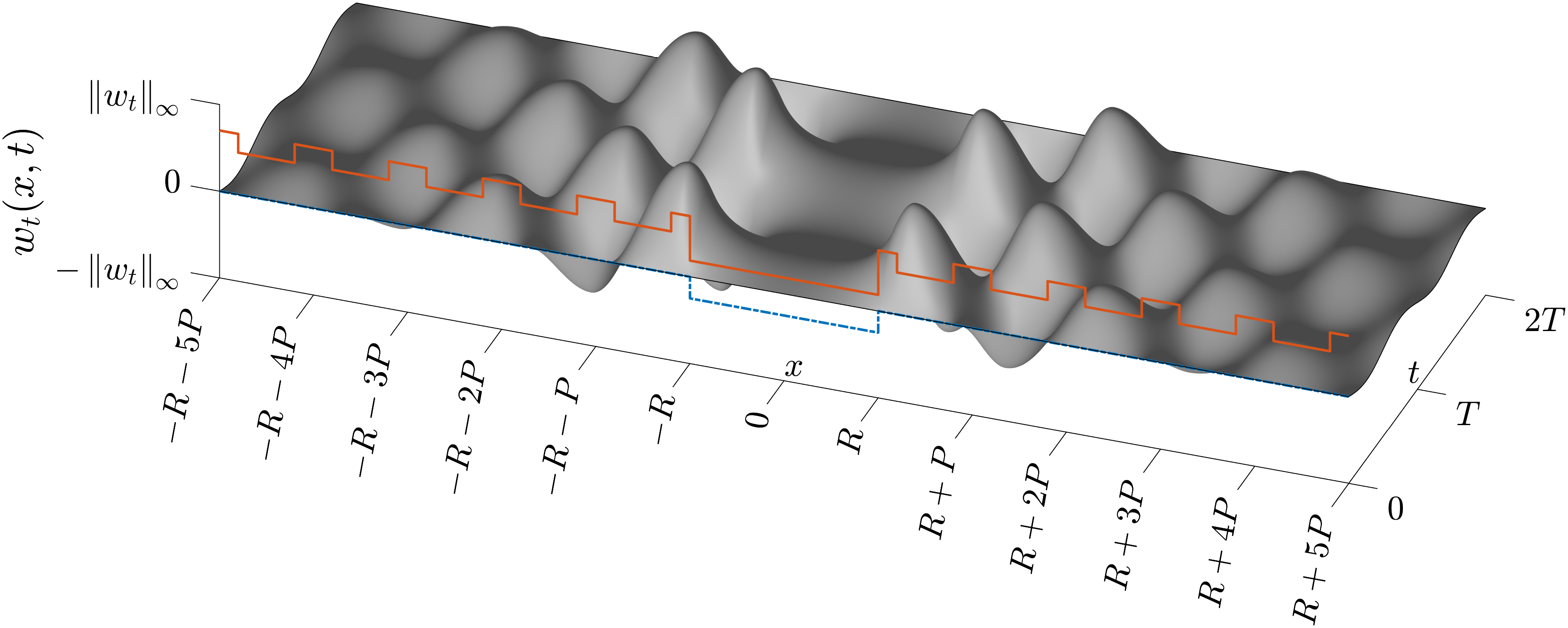}	
	\vspace{-1ex}

	\caption{Periodic potential outside $[-R,R]$: intensity (approximated) of electric field of breather solutions to \cref{thm:main_example} in reduced coordinates (cf. \eqref{eq:radial:ansatz} and \eqref{eq:slab:ansatz}) over $2$ time periods, with potentials $\tilde \chi_1$ (orange) and $\tilde \chi_3$ (blue).
	Parameters are $T = 4, \omega = \frac{\pi}{2}, c = \frac{2}{3}, a = \frac{45}{16}, b = \frac{35}{18}, d = \frac{3}{4}, R = P = 2, \theta = \frac{2}{5}, \gamma = m = n = 1$, $\kappa \equiv 1$. Top to bottom:
	$\Nins$ and cylindrical geometry; $\Nins$ and slab geometry; $\Nav$ and cylindrical geometry with $R = \frac{43}{20}$ instead; $\Nav$ and slab geometry.}
	\label{fig:periodic:breathers}
\end{figure}

\begin{figure}
	\centering
	
	\includegraphics[width=.71\linewidth]{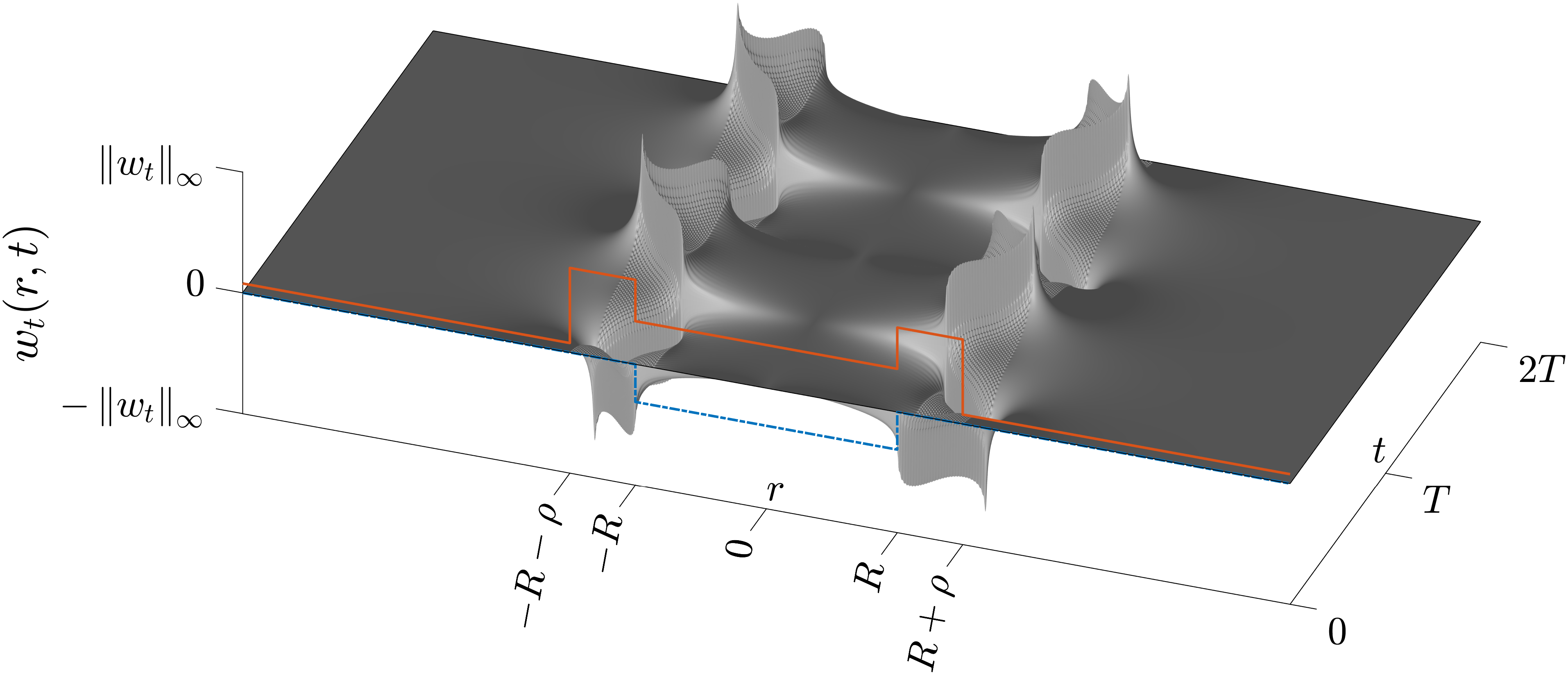}	
	\vspace{-2ex}

	\includegraphics[width=.71\linewidth]{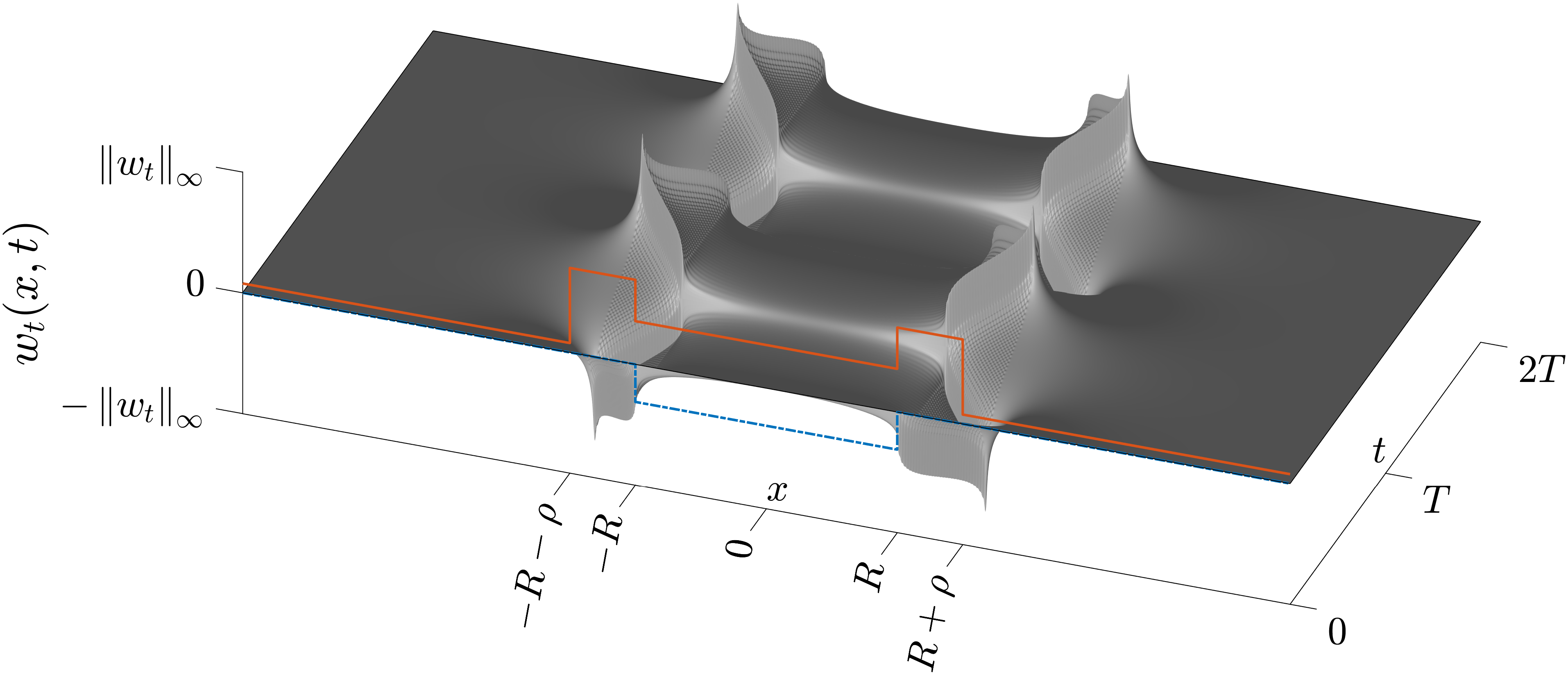}	
	\vspace{-2ex}

	\includegraphics[width=.71\linewidth]{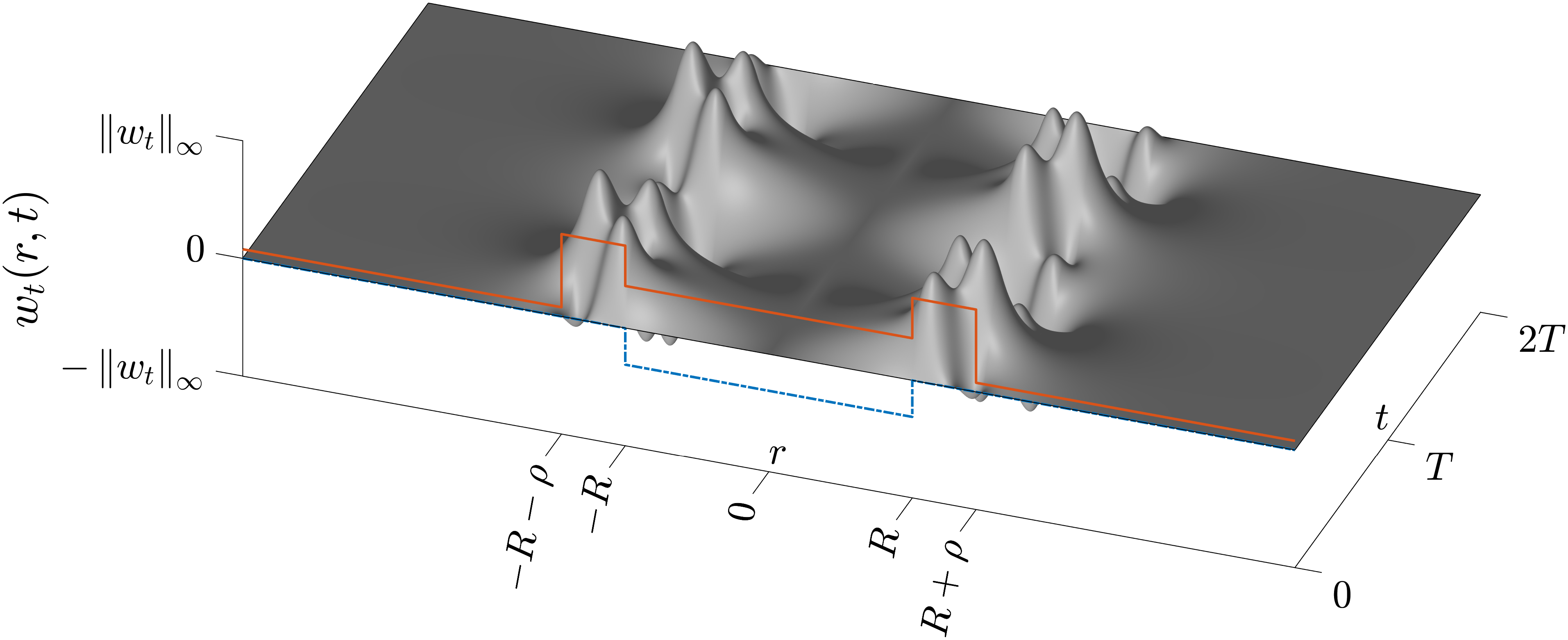}	
	\vspace{-1ex}
	
	\includegraphics[width=.71\linewidth]{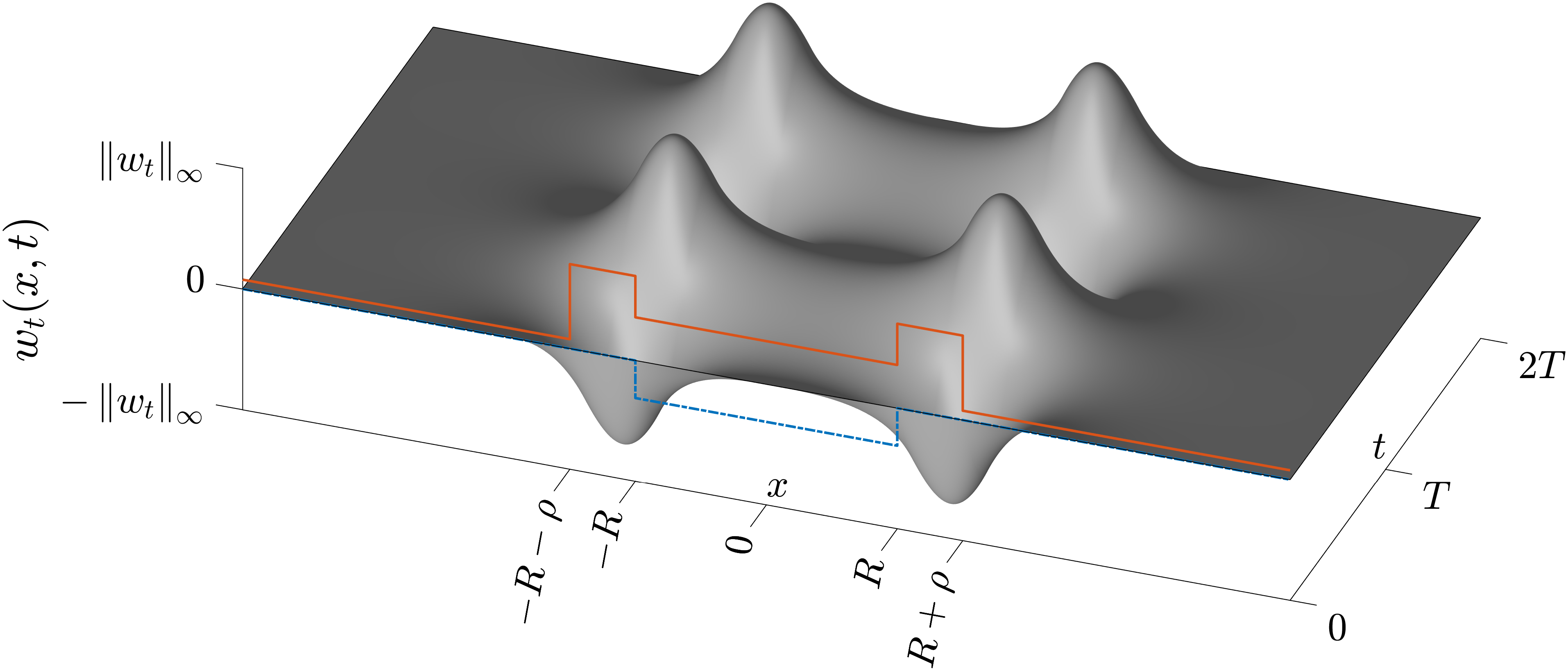}	
	\vspace{-1ex}

	\caption{Step potential outside $[-R,R]$: 
	intensity (approximated) of electric field of breather solutions to \cref{thm:main_example} in reduced coordinates (cf. \eqref{eq:radial:ansatz} and \eqref{eq:slab:ansatz}) over $2$ time periods, with potentials $\tilde \chi_1$ (orange) and $\tilde \chi_3$ (blue).
	Parameters are $T = 4, \omega = \frac{\pi}{2}, c = \frac{2}{3}, a = \frac{9}{4}, b = \frac{1}{4}, d = \frac{23}{20}, R = 2, \rho = \gamma = m = n = 1, \kappa \equiv 1$. Top to bottom: 
	$\Nins$ and cylindrical geometry; $\Nins$ and slab geometry; $\Nav$ and cylindrical geometry with $R = \frac{9}{4}$ instead; $\Nav$ and slab geometry.}
	\label{fig:step:breathers}
\end{figure}

The following observations can be made leading to open questions or conjectures:
\begin{itemize}
\item Although it is in general impossible to tell whether a computed solution is a global or just a local minimizer, the numerical minimization scheme in the instantaneous case always ends up in the same state (up to time shifts) independently of the initial state. One may therefore conjecture that ground states are unique up to shifts in time. Moreover, they seem to be even in time.
\item Ground states for time-averaged nonlinearities seem to be more smooth than for instantaneous nonlinearities. Can one show improved regularity of ground states for time-averaged nonlinearities?
\item For time-averaged nonlinearities one can consider monochromatic solutions with frequencies $k\omega$ provided $\hat\kappa_{2k}=0$ (see discussion below). In the cylindrical setting we found both monochromatic and polychromatic breathers (depending on the chosen parameters), whereas in the slab setting we only found monochromatic breathers. Can one prove that in the slab setting ground states are monochromatic? Under which parameter conditions in the cylindrical setting are ground states monochromatic/polychromatic?
\end{itemize}

A monochromatic breather has a profile $w$ of the form
\begin{align*}
	w(r, t) = \Re[v(r) e_k(t)] = \tfrac12 v(r) e_k(t) + \tfrac12 \overline{v(r)} e_{-k}(t)
\end{align*}
for some function $v$. It is compatible with the nonlinearity in the time-averaged case if $\hat \kappa_{2k} = 0$, since then the nonlinearity
\begin{align*}
	\Nav(w) 
	= \tfrac14 \Re[\hat \kappa_{2k}(v^3 e_{3k} + \abs{v}^2 v e_k) + 2 \hat \kappa_0 \abs{v}^2 v e_k]
	= \tfrac{\hat \kappa_0}{2} \Re[\abs{v}^2 v e_k]
\end{align*}
is also monochromatic along monochromatic functions. The bottom images in \cref{fig:periodic:breathers,fig:step:breathers} always depict monochromatic breathers (for the slab geometry, time-averaged nonlinearity with $\kappa \equiv 1$, and frequency index $k=1$). All other images show polychromatic breathers. Furthermore, for the time-averaged nonlinearity one can state that if there exists a nontrivial breather $w$ then there also exists a monochromatic breather with frequency index $k\in\Nodd$ provided $\hat \kappa_{2k} = 0$ and $\int_0^\infty L_k \hat w_k \cdot\hat w_k \,r\mathrm dr<0$.  

The instantaneous nonlinearity $N = \Nins$ however is not compatible with monochromatic breathers, hence all breathers for $N = \Nins$ are necessarily polychromatic, and they have infinitely many excited frequency indices $k$.

\section{Reduction to a bounded domain problem}\label{sec:domain_restriction}

From now on we assume that assumptions \ref{ass:radial:first}--\ref{ass:radial:last} are satisfied, and we set 
\begin{align}\label{eq:def:VandGamma}
	V(r) \coloneqq - (\tilde \chi_1(r) + 1 - c^{-2})
	\quad\text{and}\quad
	\Gamma(r) \coloneqq - \tilde \chi_3(r),
\end{align}
allowing us to write \eqref{eq:radial:problem} as
\begin{align}\label{eq:radial:VGproblem}
	- w_{rr} - \tfrac1r w_r + \tfrac{1}{r^2} w - V(r) w_{tt} - \Gamma(r) N(w_t)_t = 0,
	\quad r \in [0, \infty), t \in \T
\end{align}
where $V \geq 0, \Gamma \geq 0$ on $[0, R]$ due to \ref{ass:radial:elliptic}. We will show that \eqref{eq:radial:VGproblem} can be reduced to a variational problem \eqref{eq:radial:bounded_problem} below, where the conditions \eqref{eq:ass:kappa} on $\kappa$ are essential.

We consider functions $w$ which are $\nicefrac{T}{2}$--antiperiodic in time. 
This is compatible with the structure of \eqref{eq:radial:VGproblem}, in particular with the cubic nonlinearity, and we use the suffix ``anti'' to denote spaces consisting of functions which are $\nicefrac{T}{2}$--antiperiodic in time.

Using the fundamental solutions $\phi_k$   given by \ref{ass:radial:fundamental_solutions} we can further make the ansatz
\begin{align*}
	w(r, t) = \begin{cases}
		u(r, t), & 0 \leq r < R, \\
		\sum_{k \in \Zodd} \alpha_k \phi_k(r) e_k(t), & r > R.
	\end{cases}
\end{align*}
where $\alpha_k \in \C$ and $u \in \Hradanti[1]{[0, R]\times \T}$ are to be determined. Note that $\alpha_{-k} = \overline{\alpha_k}$ since $w$, $\phi_k$ being real-valued together with $\phi_{-k}=\phi_k$ imply $\alpha_{-k} \phi_{-k} = \overline{\alpha_k \phi_k}$.

We want to ensure that $w$ and $w_r$ taken from inside and outside match at $r = R$. This leads to the following conditions:
\begin{align}\label{eq:loc:deducing_boundary_data}
	u(R, t) = \sum_{k \in \Zodd} \alpha_k \phi_k(R) e_k(t), 
	\qquad
	u_r(R, t) = \sum_{k \in \Zodd} \alpha_k \phi_k'(R) e_k(t).
\end{align}
By assumption~\ref{ass:radial:fundamental_solution:estimates} we have $\phi_k(R) \neq 0$ for almost all $k \in \Zodd$. Let 
\begin{align*}
	\singularK \coloneqq \set{k \in \Zodd \colon \phi_k(R) = 0} \subseteq \Zodd	
\end{align*}
denote the finite exclusion set. The exceptional indices $k \in \singularK$ have to be treated differently than the regular indices $k \in \regularK \coloneqq \Zodd \setminus \singularK$.
Note also that due to assumption~\ref{ass:radial:fundamental_solution:estimates} there exist constants $c^\star, C^\star > 0$ such that
\begin{align}\label{eq:ass:radial:radial:estimates:reformulation}
	\abs{\phi_k(R)} \geq c^\star \norm{\phi_k}_{\Lrad{[R,\infty)}}, 
	\quad
	\abs{\phi_k'(R)} \leq C^\star \abs{k} \norm{\phi_k}_{\Lrad{[R,\infty)}}, 
	\quad
	\frac{\abs{\phi_k'(R)}}{\abs{\phi_k(R)}} \leq \frac{C^\star}{c^\star} \abs{k}
\end{align}
hold for all $k \in \regularK$.

Let us show the difference between $\singularK$ and $\regularK$:
for $k \in \singularK$ equations~\eqref{eq:loc:deducing_boundary_data} reduce to
\begin{align*}
	\hat u_k(R) = 0 \quad\text{and}\quad \alpha_k = \frac{\hat u_k'(R)}{\phi_k'(R)},
\end{align*}
whereas for $k \in \regularK$ we have 
\begin{align*}
	\alpha_k = \frac{\hat u_k(R)}{\phi_k(R)}
	\quad\text{and}\quad
	\hat u_k'(R) = \frac{\phi_k'(R)}{\phi_k(R)} \hat u_k(R).
\end{align*}

Thus we formally obtain the following boundary value problem for $u$:
\begin{align}\label{eq:radial:bounded_problem}
	\begin{cases}
		- u_{rr} - \tfrac{1}{r} u_r + \tfrac{1}{r^2} u - V(r) u_{tt} - \Gamma(r) N(u_t)_t = 0 \text{ in } [0, R] \times \T,
		\\ \hat u_k'(R) = \frac{\phi_k'(R)}{\phi_k(R)} \hat u_k(R) \text{ for } k \in \regularK,
		\\ \hat u_k(R) = 0 \text{ for } k \in \singularK.
	\end{cases}
\end{align}
The formal calculation will be justified in the proof of Theorem~\ref{thm:radial:main} when we establish the weak-solution property. Problem \eqref{eq:radial:bounded_problem} again is variational and solutions are critical points of the functional $\efct$ given by
\begin{equation}\label{eq:radial:energy}
	\begin{aligned} 
		& \efct(u) = \efctI(u) - \efctB(u) \quad \mbox{ where }
		\\ & \efctI(u) \coloneqq \int_{[0, R] \times \T} \left( \tfrac12 u_r^2 + \tfrac12 \left(\tfrac{1}{r} u\right)^2 + \tfrac12 V(r) u_t^2 + \tfrac14 \Gamma(r) N(u_t) u_t\right) \der[r] (r, t) 
		\\ & \efctB(u) \coloneqq \frac{R}{2} \sum_{k \in \regularK} \frac{\phi_k'(R)}{\phi_k(R)} \abs{\hat u_k(R)}^2
	\end{aligned}
\end{equation}
subject to the constraints $\hat u_k(R) = 0$ for $k \in \singularK$.
Indeed, for a (sufficiently regular) solution $u$ and a sufficiently smooth function $\varphi \colon [0, R] \times \T \to \R$ we have
\begin{align*}
	0 &= \int_{[0, R] \times \T} \left(- u_{rr} - \tfrac{1}{r} u_r + \tfrac{1}{r^2} u - V(r) u_{tt} - \Gamma(r) N(u_t)_t\right) \varphi \der[r](r, t)
	\\ &= \int_{[0, R] \times \T} \left(u_r \varphi_r + \tfrac{1}{r^2} u \varphi + V(r) u_t \varphi_t + \Gamma(r) N(u_t) \varphi_t \right) \der[r](r, t)
	- R \int_\T u_r(R, t) \varphi(R, t) \der t
	\\ &= \int_{[0, R] \times \T} \left(u_r \varphi_r + \tfrac{1}{r^2} u \varphi + V(r) u_t \varphi_t + \Gamma(r) N(u_t) \varphi_t \right) \der[r](r, t)
	- R \sum_{k \in \regularK} \frac{\phi_k'(R)}{\phi_k(R)} \hat u_k(R) \overline{\hat \varphi_k(R)}
	\\ &= \efct'(u)[\varphi]. 
\end{align*}
Here we used $\hat \varphi_k(R) = 0$ for $k \in \singularK$ and that by Plancherel 
\begin{align*}
	R\int_\T u_r(R, t) \varphi(R, t) \der t = R\int_\T u_r(R, t) \overline{\varphi}(R, t) \der t =R\sum_{k \in \regularK} \frac{\phi_k'(R)}{\phi_k(R)} \hat u_k(R) \overline{\hat \varphi_k(R)}
\end{align*}
so that this quantity is real and thus coincides with $\efctB'(u)[\varphi] = \Re \left[R\sum_{k \in \regularK} \frac{\phi_k'(R)}{\phi_k(R)} \hat u_k(R) \overline{\hat \varphi_k(R)} \right]$. 
We further used that $\efctN(u) \coloneqq \int_{[0, R] \times \T} \left(\tfrac14 \Gamma(r) N(u_t) u_t\right) \der[r](r, t)$ satisfies
\begin{align*}
	\efctN'(u)[\varphi] = \int_{[0, R] \times \T} \left(\Gamma(r) N(u_t) \varphi_t\right) \der[r](r, t).
\end{align*}
Indeed, for $N = \Nins$ we have
\begin{align*}
	\efctN(u) = \tfrac14 \int_{[0, R] \times \T} \left(\Gamma(r) u_t^4\right) \der[r](r, t),
	\quad\text{hence}\quad
	\efctN'(u)[\varphi] = \int_{[0, R] \times \T} \left(\Gamma(r) u_t^3 \varphi_t\right) \der[r](r,t).
\end{align*}
If $N = \Nav$, using that $\kappa$ is even by \eqref{eq:ass:kappa} one has
\begin{align*}
	\int_\T (\kappa \ast (u_t \varphi_t)) u_t^2 \der t
	&= \int_\T \int_\T \kappa(t-\tau) u_t(\tau) \varphi_t(\tau) u_t(t)^2 \der \tau \der t
	\\ &= \int_\T \int_\T \kappa(\tau - t) u_t(\tau) \varphi_t(\tau) u_t(t)^2 \der t \der \tau
	= \int_\T (\kappa \ast u_t^2) u_t \varphi_t \der \tau,
\end{align*}
and therefore $\efctN(u) = \tfrac14 \int_{[0,R]\times\T} \left(\Gamma(r) (\kappa\ast u_t^2) u_t^2 \right) \der[r](r,t)$ does satisfy
\begin{align*}
	\efctN'(u)[\varphi] 
	&= \tfrac12 \int_{[0, R] \times \T} \left( \Gamma(r) (\kappa \ast u_t\varphi_t) u_t^2 + \Gamma(r) (\kappa \ast u_t^2) u_t\varphi_t \right) \der[r](r,t)
	\\ &= \int_{[0, R] \times \T} \left( \Gamma(r) (\kappa \ast u_t^2) u_t\varphi_t \right) \der[r](r,t).
\end{align*}

As a next step we properly define the functional $\efct$ and investigate its properties.

\begin{definition}
	Define the norm $\normN{}$ depending on the nonlinearity $N$ by
	\begin{align*}
		\normNins{v} \coloneqq \left(\int_{[0, R] \times \T} v^4 \der[r](r, t)\right)^{\nicefrac14} = \norm{v}_{\Lrad[4]{[0, R] \times \T}}
	\end{align*}
	and 
	\begin{align*}
		\normNav{v} \coloneqq \left(\int_0^R \left( \int_\T v^2 \der t \right)^2 \der[r]r \right)^{\nicefrac14} = \norm{v}_{\Lrad[4]{[0, R]; L^2(\T)}}.
	\end{align*}
\end{definition}

\begin{remark}\label{rem:normN}
	We have 
	\begin{align*}
		\norm{v}_{\Lrad{[0, R] \times \T}} \leq \tfrac{\sqrt{R}}{\sqrt[4]{2}} \normNav{v},
		\quad
		\normNav{v} \leq \normNins{v},
		\quad\text{and}\quad
		\int_{[0, R] \times \T} \Gamma(r) N(v) v \der[r](r, t) \eqsim \normN{v}^4.
	\end{align*}
	The first two estimates immediately follow from H\"older's inequality. The last estimate is clear for $N = \Nins$ since $N(v)v = v^4$ and $\Gamma$ is bounded and strictly positive by assumption~\ref{ass:radial:VandGamma}. For $N = \Nav$ we have
	\begin{align*}
		\int_{\T} N(v) v \der(t) = \int_\T \int_\T \kappa(t - \tau) v^2(r, \tau) v^2(r, t) \der \tau \der t
	\end{align*}
	so that 
	\begin{align*}
		\essinf_{[0, R]} \Gamma \cdot \min \kappa \cdot \normN{v}^4 \leq \int_{[0, R] \times \T} \Gamma(r) N(v) v \der[r](r, t) \leq \esssup_{[0, R]} \Gamma \cdot \max \kappa \cdot \normN{v}^4.
	\end{align*}
\end{remark}

\begin{proposition}\label{prop:energy_properties}
	The functionals $\efct, \efctI$, $\efctB$ given by \eqref{eq:radial:energy} are well-defined and differentiable on the reflexive Banach space
	\begin{align*}
		\espace \coloneqq \set{u \in W^{1,1}_{\mathrm{loc},\mathrm{anti}}([0, R] \times \T) \middle\vert
		\begin{array}{l}
			u_r, \tfrac{1}{r} u \in \Lrad{[0, R] \times \T}, \\
			\normN{u_t} < \infty, \hat u_k(R) = 0 \text{ for } k \in \singularK
		\end{array}}
	\end{align*}
	with norm
	\begin{align*}
		\norm{u}_{\espace} \coloneqq \norm{u_r}_{\Lrad{[0, R] \times \T}} + \norm{\tfrac{1}{r} u}_{\Lrad{[0, R] \times \T}} + \normN{u_t}.
	\end{align*}
	The derivative is given by 
	\begin{align*}
		\efct'(u)[\varphi] 
		= \int_{[0, R] \times \T} \left( u_r \varphi_r + \tfrac{1}{r^2} u \varphi + V(r) u_t \varphi_t + \Gamma(r) N(u_t) \varphi_t \right) \der[r](r,t)
		- R\sum_{k \in \regularK} \frac{\phi_k'(R)}{\phi_k(R)} \hat u_k(R) \overline{\hat \varphi_k(R)}.
	\end{align*}
	Furthermore, $\efctI$ is sequentially weakly lower semicontinuous, $\efctB$ is sequentially weakly continuous, and $\efct$ is sequentially weakly lower semicontinuous as well as coercive. Therefore $\efct$ attains its minimum $\emin \coloneqq \inf \efct = \efct(\earg)$ and $\earg$ is a critical point of $\efct$. 
\end{proposition}
\begin{proof}
	Using assumption~\ref{ass:radial:VandGamma} and \cref{rem:normN} one can show in a standard way that $\efctI$ is well-defined and differentiable. The formula for the derivative follows from the calculations above. Since $V\geq 0$ the quadratic terms of $\efctI$ are convex, and the same holds for the remaining part $\efctN$ since $\Gamma\geq 0$ together with assumption \ref{eq:ass:kappa} in the time-averaged case. Therefore $\efctI$ is (sequentially) weakly lower semicontinuous.

	With \eqref{eq:ass:radial:radial:estimates:reformulation} we obtain $\abs{\efctB(u)} \leq C_0 \norm{u(R, \impvar)}_{H^{\nicefrac12}(\T)}^2$, so from compactness of the trace (see \cref{lem:trace}) it follows that $\efctB$ is sequentially weakly continuous and in particular continuous.

	It remains to show that $\efct$ is coercive. Using \cref{rem:normN} and \cref{lem:trace} with $\eps \coloneqq \tfrac{1}{4 C_0}$ we estimate
	\begin{align*}
		\efct(u) 
		&\geq \tfrac12 \norm{u_r}_{\Lrad{}}^2 + \tfrac12 \norm{\tfrac{1}{r} u}_{\Lrad{}}^2 + \tfrac{c_1}{4} \normN{u_t}^4
		- C_0 \left( \eps \norm{u_r}_{\Lrad{}}^2 + C(\eps) \normN{u_t}^2 \right)
		\\ &= \tfrac14 \norm{u_r}_{\Lrad{}}^2 + \tfrac12 \norm{\tfrac{1}{r} u}_{\Lrad{}}^2 + \tfrac{c_1}{4} \normN{u_t}^4 - C_0 C(\eps) \normN{u_t}^2
	\end{align*}
	for some $c_1 > 0$.
	Thus $\efct(u) \to \infty$ as $\norm{u}_{\espace} \to \infty$. Using \cite[Chapter~I, Theorem~1.2]{struwe} we find that $\efct$ attains its infimum at a critical point, which completes the proof.
\end{proof}

Next we show that assumption \ref{ass:radial:nontrivial_sol} is a sufficient condition for the solution $\earg$ obtained above to be nontrivial.\footnote{One can show that \ref{ass:radial:nontrivial_sol} is also necessary for $\earg \neq 0$ in the case where $V$ is constant on $[0, R]$.} 

\begin{proposition}\label{prop:nonzero_minimizer}
	The minimal energy level of $\efct$ satisfies $\emin < 0$ and hence $\earg \neq 0$.
\end{proposition}
\begin{proof}
	Let $k_0 \in \regularK$ be as in \ref{ass:radial:last} and recall that $\lambda =\omega \norm{V}_{L^\infty([0, R])}^{\nicefrac12}$. Define
	\begin{align*}
		f(r) \coloneqq I_1(\lambda k_0 r), 
		\qquad
		u(r, t) \coloneqq \eps f(r) \left(e_{k_0}(t) + e_{-{k_0}}(t)\right)
	\end{align*}
	where $I_1$ is the modified Bessel functions of first kind, i.e., it satisfies
	\begin{align*}
		\left( - \partial_r^2 - \tfrac{1}{r} \partial_r + \tfrac{1}{r^2} + 1 \right) I_1 = 0, 
		\qquad
		I_1(0) = 0.
	\end{align*}
	We calculate
	\begin{align*}
		\efct(u) 
		&= \int_{(0, R) \times \T} \left( \tfrac12 u_r^2 + \tfrac12 (\tfrac{1}{r} u)^2 + \tfrac12 V(r) u_t^2 + \tfrac14 \Gamma(r) N(u_t) u_t \right) \der[r] (r, t) 
		- \frac{R}{2} \sum_{k \in \Zodd} \frac{\phi_k'(R)}{\phi_k(R)} \abs{\hat u_k(R)}^2
		\\ &= \eps^2 \left(\int_0^R \left( (f')^2 + (\tfrac{1}{r} f)^2 + \omega^2 k_0^2 V(r) f^2 \right) \der[r] r - \frac{R \phi_{k_0}'(R)}{\phi_{k_0}(R)} f(R)^2 \right) + \landauO(\eps^4)
		\\ &\leq \eps^2 \left(\int_0^R f \left( - f'' - \tfrac{1}{r} f' + \tfrac{1}{r^2} f + \lambda^2 k_0^2 f \right) \der[r] r + \left[ r f f' \right]_0^R
		- \frac{R \phi_{k_0}'(R)}{\phi_{k_0}(R)} f(R)^2 \right) + \landauO(\eps^4)
		\\ &= \eps^2 R f(R)^2 \left(\frac{f'(R)}{f(R)}  - \frac{\phi_{k_0}'(R)}{\phi_{k_0}(R)} \right) + \landauO(\eps^4).
	\end{align*}
	We have $f(R) > 0$ and by assumption \ref{ass:radial:last} also
	\begin{align*}
		\frac{f'(R)}{f(R)}  - \frac{\phi_{k_0}'(R)}{\phi_{k_0}(R)}
		= \frac{\lambda k_0 I_1'(\lambda k_0 R)}{I_1(\lambda k_0 R)} - \frac{\phi_{k_0}'(R)}{\phi_{k_0}(R)} < 0.
	\end{align*}
	Thus $\efct(u) < 0$ for $\eps > 0$ sufficiently small, which completes the proof.
\end{proof}

	\section{Approximation by finitely many harmonics}\label{sec:approximation}

In this section we discuss approximations of the minimizers of $\efct$ by finitely many harmonics
\begin{align*}
	u(r, t) \approx \sum_{\substack{k \in \Zodd \\ \abs{k} \leq K}} \hat u_k(r) e_k(t),
\end{align*}
that is we consider $\efct$ on the subspace $\espaceF$ of $\espace$ defined next.
\begin{definition}
	Let $K \in \Nodd$. Then we define 
	\begin{align*}
		\espaceF \coloneqq \set{u \in \espace \colon \hat u_k \equiv 0 \text{ for } \abs{k} > K}.
	\end{align*}
\end{definition}

First we discuss the canonical projection from $\espace$ to $\espaceF$

\begin{lemma}\label{lem:projection}
	For $K \in \Nodd$, define the operator
	\begin{align*}
		S^K \colon \espace \to \espaceF, 
		\quad
		S^K[u](r, t) = \sum_{\substack{k \in \Zodd \\ \abs{k} \leq K}} \hat u_k(r) e_k(t).
	\end{align*}
	Then the operators $S^K$ are uniformly bounded in $\calB(\espace)$ and $S^K u \to u$ in $\espace$ as $K \to \infty$ for all $u \in \espace$.
\end{lemma}
\begin{proof}
	For $p \in (1, \infty)$ the Fourier cutoff operators $S^K$ defined by 
	\begin{align*}
		S^K \colon \Lanti[p]{\T} \to \Lanti[p]{\T}, 
		\quad
		S^K[f](t) = \sum_{\substack{k \in \Zodd \\ \abs{k} \leq K}} \hat f_k e_k(t).
	\end{align*}
	are uniformly bounded and $S^K f \to f$ in $\Lanti[p]{\T}$ as $K \to \infty$ (see \cite[Theorem~4.1.8 and Corollary~4.1.3]{grafakos_classical}). By acting on the time variable only, $S^K$ extend to uniformly bounded operators
	\begin{align*}
		S^K \colon \Lrad[q]{[0, R]; \Lanti[p]{\T}} \to \Lrad[q]{[0, R]; \Lanti[p]{\T}} 
	\end{align*}
	with $S^K u \to u$ in $\Lrad[q]{[0, R]; \Lanti[p]{\T}}$ as $K \to \infty$.
	Then from $S^K[u_r] = (S^K u)_r, S^K[u_t] = (S^K u)_t$, and $S^K[\frac{1}{r} u] = \tfrac{1}{r} (S^K u)$ it follows that $S^K \colon \espace\to \espaceF$ are also uniformly bounded operators and $S^K u \to u$ in $\espace$ as $K \to \infty$.
\end{proof}

Next we show that the minimal energy level $\emin$ can be approximated from within $\espaceF$.

\begin{lemma}\label{lem:energy_limit}
For every $K\in \Nodd$ there exists $\eargF \in \espaceF$ such that $\eminF \coloneqq \inf \efctF = \efct(\eargF)$. Furthermore $\displaystyle \lim_{K \to \infty} \eminF = \emin$ holds.
\end{lemma}

\begin{proof}
    Arguing as in \cref{prop:energy_properties}, one can show that there exists a minimizer $\eargF \in \espaceF$ of $\efctF$. Setting $u^K \coloneqq S^K(\earg)$ we find
	\begin{align}\label{eq:loc:1}
		\efct(\earg) = \emin \leq \eminF \leq \efct(u^K).
	\end{align}
	Using $u^K \to \earg$ as $K \to \infty$ and that $\efct$ is continuous, the second claim follows from \eqref{eq:loc:1} in the limit $K \to \infty$.
\end{proof}

As a next step we establish uniform estimates on the minimizers $\eargF$. First, we introduce the fractional time derivative $\fracDT{s}$ and a quantity $\qN{}$ that behaves like a norm stronger than $\normN{}$.

\begin{definition}
	For $s \in \R$ we define the fractional time derivative $\fracDT{s}$ as the Fourier multiplier with symbol $\abs{\omega k}^s$, i.e. $\F_k \fracDT{s} = \abs{\omega k}^s \F_k$.
\end{definition}

\begin{definition}
	For $N = \Nins$ we define the quantity
	\begin{align*}
		\qNins{v} \coloneqq \left( \int_{[0, R] \times \T} \left( \halfDT (v \abs{v}) \right)^2 \der[r](r, t) \right)^{\nicefrac14}.
	\end{align*}
	For $N = \Nav$ we define the quantity
	\begin{align*}
		\qNav{v} \coloneqq \left( \int_0^R \|v(r,\cdot)\|_{L^2(\T)}^2 \| \halfDT v(r,\cdot)\|_{L^2(\T)}^2 \der[r] r \right)^{\nicefrac14}.
	\end{align*}
\end{definition}

\begin{remark}\label{rem:qN_estimate}
	For $N = \Nins$ by \cref{lem:nonlinear_derivative_inequality} we have
	\begin{align*}
		\int_{[0, R] \times \T} \Gamma(r) \Nins(v) \DT v \der[r](r, t)
		\geq c^\ast \qNins{v}^4,
	\end{align*}
	with $c^\ast = \tfrac12 \essinf_{[0, R]} \Gamma > 0$
	whereas for $N = \Nav$ using \cref{lem:nonlinear_derivative_inequality_av} (with constants $c_1, C_2$) we have
	\begin{align*}
		\int_{[0, R] \times \T} \Gamma(r) \Nav(v) \DT v \der[r](r, t)
		\geq c^\ast \qNav{v}^4 - C^\ast \normNav{v}^4
	\end{align*}
	with $c^\ast = c_1 \essinf_{[0, R]} \Gamma$ and $C^\ast = C_2 \esssup_{[0, R]} \Gamma$.
	In particular, 
	\begin{align*}
		\int_{[0, R] \times \T} N(v) \DT v \der[r](r, t)
		\geq c^\ast \qN{v}^4-C^\ast \normN{v}^4	
	\end{align*}
	holds for both choices of $N$.
\end{remark}

The minimizers $\eargF$ formally are solutions of
\begin{align*}
	\begin{cases}
		S^K[-u_{rr} - \tfrac{1}{r} u_r + \tfrac{1}{r^2} u - V(r) u_{tt} - \Gamma(r) N(u_t)_t] = 0 \text{ in } [0, R] \times \T,
		\\ \hat u_k'(R) = \frac{\phi_k'(R)}{\phi_k(R)} \hat u_k(R) \text{ for } k \in \regularK, \abs{k} \leq K,
		\\ \hat u_k(R) = 0 \text{ for } k \in \singularK, \abs{k} \leq K.
	\end{cases}
\end{align*}
Here the main part $-\partial_r^2 - \frac1r \partial_r + \frac{1}{r^2} - V(r) \partial_t^2 - \Gamma(r) \partial_t N(\partial_t \impvar)$ is elliptic by \ref{ass:radial:elliptic}, which is why we expect the solution $u$ to have increased regularity. Often this is shown by testing the problem against derivatives of the solution. In \cref{prop:minimizer_estimates}, we obtain improved regularity by testing the problem against $\DT \eargF$. However, with this method it is impossible to obtain even more regularity because when testing against $\fracDT{s} \eargF$ with $s > 1$ one can no longer control the appearing boundary terms.

\begin{proposition}\label{prop:minimizer_estimates}
	There exist constants $C_1, \dots, C_5 > 0$ independent of $K$ such that the following holds:
	\begin{enumerate}
		\item\label{i:minimizer_estimate:norm} $\displaystyle \norm{\eargF}_{\espace} \leq C_1$,
		\item\label{i:minimizer_estimate:reg1} $\displaystyle \norm{\halfDT \eargF_r}_{\Lrad{[0, R] \times \T}} \leq C_2$,
		\item\label{i:minimizer_estimate:reg2} $\displaystyle \norm{\tfrac{1}{r} \halfDT \eargF}_{\Lrad{[0, R] \times \T}} \leq C_3$,
		\item\label{i:minimizer_estimate:reg3} $\displaystyle \qN{\eargF_t} \leq C_4$,
		\item\label{i:minimizer_estimate:reg4} $\displaystyle \norm{\eargF(R, \impvar)}_{H^1(\T)} \leq C_5$.
	\end{enumerate}
\end{proposition}
\begin{proof}
	Since $\efct$ is coercive (see \cref{prop:energy_properties}), there exists $C_1 > 0$ such that $\efct(u) > 0$ holds for all $u \in \espace$ with $\norm{u}_{\espace} > C_1$. Using $\efct(\eargF) = \min \efctF \leq \efct(0) = 0$ we conclude $\norm{\eargF}_{\espace} \leq C_1$, so that \ref{i:minimizer_estimate:norm} holds.

	For \ref{i:minimizer_estimate:reg1}--\ref{i:minimizer_estimate:reg4} we first note that 
	\begin{align*}
		\nnorm{u} = \sum_{\substack{k \in \Zodd \\ \abs{k} \leq K}} \left(\norm{\hat u_k'}_{\Lrad{[0, R]}} + \norm{\tfrac{1}{r}\hat u_k}_{\Lrad{[0, R]}} + \norm{\hat u_k}_{\Lrad[4]{[0, R]}}\right)
	\end{align*}
	defines an equivalent norm on $\espaceF$. Thus the operators $\fracDT{s}$ are bounded on $\espaceF$ for all $s \in \R$. In particular, $\DT \eargF \in \espaceF$. Using $V \geq 0$ on $[0, R]$ and \eqref{eq:ass:radial:radial:estimates:reformulation} we calculate
	\begin{align*}
		0 
		&= \efct'(\eargF)[\DT \eargF] 
		\\ &= \int_{[0, R] \times \T} \left( \eargF_r \DT \eargF_r + \tfrac{1}{r^2} \eargF \DT \eargF + V(r) \eargF_t \DT \eargF_t + \Gamma(r) N(\eargF_t) \DT \eargF_t \right) \der[r](r, t)
		\\ &\qquad - R \sum_{k \in \Zodd} \frac{\phi_k'(R)}{\phi_k(R)} \omega \abs{k} \abs{\F_k[\eargF](R)}^2 
		\\ &\geq \int_{[0, R] \times \T} \left( \left(\halfDT \eargF_r\right)^2 + \left(\tfrac{1}{r} \halfDT \eargF\right)^2 + \Gamma(r) N( \eargF_t ) \DT \eargF_t \right) \der[r](r, t)
		\\ &\qquad - C_0 R \sum_{k \in \Zodd} \omega k^2 \abs{\F_k[\eargF](R)}^2.
	\end{align*}
	Using further
	\begin{align*}
		C_0 R \sum_{k \in \Zodd} \omega k^2 \abs{\F_k[\eargF](R)}^2
		\leq \tilde C_0 \norm{\eargF(R,\cdot)}_{H^1(\T)}^2,
	\end{align*}
	\cref{rem:qN_estimate}, \cref{lem:regularized_trace} with $\eps = \frac{1}{2 \tilde C_0}$ as well as $a X^2 - b X \geq X - \frac{(b+1)^2}{4 a}$, we obtain
	\begin{align}\label{eq:loc:2}
		\begin{aligned}
			0 
			&\geq \tfrac12 \norm{\halfDT \eargF_r}_{\Lrad{}}^2 
			+ \norm{\tfrac{1}{r} \halfDT \eargF}_{\Lrad{}}^2
			+ c^\ast \qN{\eargF_t}^4- C^\ast \normN{\eargF_t}^4
			 - \widetilde C_0 C(\eps) \qN{\eargF_t}^2
			\\ &\geq \tfrac12 \norm{\halfDT \eargF_r}_{\Lrad{}}^2 
			+ \norm{\tfrac{1}{r} \halfDT \eargF}_{\Lrad{}}^2
			+ \qN{\eargF_t}^2
			- \frac{(\widetilde C_0 C(\eps) + 1)^2}{4 c^*} - C^\ast C_1^4.
		\end{aligned}
	\end{align}
	With $C \coloneqq \frac{(\widetilde C_0 C(\eps) + 1)^2}{4 c^\ast}+C^\ast C_1^4$ the estimates \ref{i:minimizer_estimate:reg1}--\ref{i:minimizer_estimate:reg3} follow from \eqref{eq:loc:2} where
	\begin{align*}
		C_2 \coloneqq \sqrt{2 C}, 
		\quad
		C_3 \coloneqq C_4 \coloneqq \sqrt{C},
	\end{align*}
	and lastly \ref{i:minimizer_estimate:reg4} follows from \ref{i:minimizer_estimate:reg1} and \ref{i:minimizer_estimate:reg3} using \cref{lem:regularized_trace} again.
\end{proof}

The following result is the most important result in this section. It shows how a minimizer $u$ of $\efct$ gains additional regularity via the approximation by finitely many harmonics. This will be the key to establish regularity of the solutions of \eqref{eq:radial:VGproblem} across the boundary at $r=R$. 
\begin{proposition}\label{prop:limit_estimates}
	Up to a subsequence, the limit $u = \lim_{K \to \infty} \eargF$ exists in $\espace$. The function $u$ is a minimizer of $\efct$ and satisfies
	\begin{enumerate}
		\item
		$\displaystyle \norm{u}_{\espace} \leq C_1$,
		\item\label{i:limit_estimate:reg1} 
		$\displaystyle \norm{\halfDT u_r}_{\Lrad{[0, R] \times \T}} \leq C_2$,
		\item
		$\displaystyle \norm{\tfrac{1}{r} \halfDT u}_{\Lrad{[0, R] \times \T}} \leq C_3$,
		\item
		$\displaystyle \qN{u_t} \leq C_4$,
		\item\label{i:limit_estimate:reg4} $\displaystyle \norm{u(R, \impvar)}_{H^1(\T)} \leq C_5$,
	\end{enumerate}
	where the constants $C_1, \dots, C_5$ are the same as in \cref{prop:minimizer_estimates}.
\end{proposition}
\begin{proof}
	We only consider the case $N = \Nav$, as for $N = \Nins$ one can argue similarly. Then due to \cref{prop:minimizer_estimates} and the definition of $\qNav{}$ the weak limits
	\begin{align*}
		\eargF &\wto u \text{ in } \espace,
		\\ \halfDT \eargF_r &\wto f \text{ in } \Lrad{[0, R] \times \T},
		\\ \tfrac{1}{r} \halfDT \eargF &\wto g \text{ in } \Lrad{[0, R] \times \T}, 
		\\  \|\eargF_t\|_{L^2(\T)} \halfDT \eargF_t &\wto h \text{ in } \Lrad{[0, R] \times \T},
		\\ \eargF(R, \impvar) &\wto b \text{ in } H^1(\T)
	\end{align*}
	exist for $K\to \infty$ up to a subsequence and satisfy $\norm{u}_{\espace} \leq C_1$, $\norm{f}_{\Lrad{}} \leq C_2$, $\norm{g}_{\Lrad{}} \leq C_3$, $\norm{h}_{\Lrad{}} \leq C_4$, $\norm{b}_{H^1} \leq C_5$.
	Using the properties of the functional $E$ from \cref{prop:energy_properties,lem:energy_limit} we further obtain
	\begin{align*}
		\emin \leq \efct(u) \leq \lim_{K \to \infty} \efct(\eargF) = \lim_{K \to \infty} \eminF = \emin,
	\end{align*}
	so that $\efct(u) = \emin = \lim_{K \to \infty} \efct(\eargF)$. In particular $u$ is a minimizer of $\efct$. 

	Also, since $\efctB(\eargF) \to \efctB(u)$ for $K \to \infty$, we obtain $\efctI(\eargF) \to \efctI(u)$ as $K \to \infty$. From this it follows that $\eargF \to u$ in $\espace$ as $K \to \infty$ as we show next. Since $\eargF \wto u$ we see that 
	\begin{align*}
		\tfrac{1}{r}\eargF \wto \tfrac{1}{r} u, ~ \eargF_r \wto u_r \mbox{ in } \Lrad[2]{[0, R] \times \T}, 
		\qquad
		\eargF_t \wto u_t \mbox{ in } \normN{}.
	\end{align*} 
	Moreover, by weak sequential lower semicontinuity we have 
    \begin{align*}
		&\efctI(u) 
		= \tfrac12 \norm{u_r}_{\Lrad{}}^2 
		+ \tfrac12 \norm{\tfrac{1}{r} u}_{\Lrad{}}^2 
		+ \tfrac12 \norm{V^{\nicefrac12} u_t}_{\Lrad{}}^2 
		+ \tfrac14 \normNav{\Gamma^{\nicefrac14} u_t}^4
		\\ &\quad\leq \tfrac12 \liminf_{K \to \infty} \norm{\eargF_r}_{\Lrad{}}^2 
		+ \tfrac12 \liminf_{K \to \infty}\norm{\tfrac{1}{r} \eargF}_{\Lrad{}}^2 
		+ \tfrac12 \liminf_{K \to \infty}\norm{V^{\nicefrac12} \eargF_t}_{\Lrad{}}^2 
		+ \tfrac14 \liminf_{K \to \infty}\normNav{\Gamma^{\nicefrac14}  \eargF_t}^4
		\\ &\quad\leq \tfrac12 \limsup_{K \to \infty} \norm{\eargF_r}_{\Lrad{}}^2 
		+ \tfrac12 \liminf_{K \to \infty}\norm{\tfrac{1}{r} \eargF}_{\Lrad{}}^2 
		+ \tfrac12 \liminf_{K \to \infty}\norm{V^{\nicefrac12} \eargF_t}_{\Lrad{}}^2 
		+ \tfrac14 \liminf_{K \to \infty}\normNav{\Gamma^{\nicefrac14}  \eargF_t}^4
		\\ &\quad\leq \limsup_{K \to \infty} \efctI(\eargF) = \efctI(u).
	\end{align*}
	Notice that in the second inequality we have replaced one $\liminf$ by a $\limsup$ and in the last inequality we used that $\limsup_{n\to\infty} a_n + \sum_{i=1}^p \liminf_{n\to\infty} b_n^i \leq \limsup_{n\to\infty} (a_n+\sum_{i=1}^p b_n^i)$ which follows from $\sup_{n\in\N} a_n + \sum_{i=1}^p \inf_{n\in\N} b_n^i \leq \sup_{n\in\N} (a_n+\sum_{i=1}^p b_n^i)$. It follows that $\norm{u_r}_{\Lrad{[0, R] \times \T}}^2 = \liminf_{K \to \infty} \norm{\eargF_r}_{\Lrad{[0, R] \times \T}}^2 = \limsup_{K \to \infty} \norm{\eargF_r}_{\Lrad{[0, R] \times \T}}^2$. Combining weak convergence $\eargF_r \wto u_r$ with convergence of the norms $\norm{\eargF_r}_{\Lrad{}} \to \norm{u_r}_{\Lrad{}}$, we find that $\eargF_r \to u_r$ in $\Lrad{[0, R] \times \T}$ as $K \to \infty$. With a similar argument we find $\tfrac1r \eargF \to \tfrac1r u$ in $\Lrad{[0, R] \times \T}$ and $\eargF_t \to u_t$ in $\normNav{}$ as $K \to \infty$. Together, this shows $\eargF \to u$ in $\espace$ as $K \to \infty$.

	It remains to show the estimates \ref{i:limit_estimate:reg1}--\ref{i:limit_estimate:reg4}. These follow from the identities
	\begin{align*}
		f = \halfDT u_r, 
		\quad
		g = \tfrac{1}{r} \halfDT u,
		\quad
		h = \norm{u_t}_{L^2(\T)} \halfDT u_t,
		\quad
		b = u(R, \impvar),
	\end{align*}
	where we only discuss $h = \norm{u_t}_{L^2(\T)} \halfDT u_t$ as an example. First, by definition of $\qN{}$ and convergence $\eargF \to u$ in $\espace$ we have $\norm{\eargF}_{L^2(\T)} \eargF \to \norm{u}_{L^2(\T)} u$ in $\Lrad{[0, R] \times \T}$. Taking the Fourier transform, for $k \in \Zodd$ we find
	\begin{align*}
		\F_k[\Norm{\eargF_t}_{L^2(\T)} \eargF_t] \to \F_k[\norm{u_t}_{L^2(\T)} u_t] 
	\end{align*}
	and also
	\begin{align*}
		\sqrt{\omega \abs{k}} \F_k[\Norm{\eargF_t}_{L^2(\T)} \eargF_t] = \F_k[\Norm{\eargF_t}_{L^2(\T)} \halfDT \eargF_t]
		\wto \F_k[h]
	\end{align*}
	in $\Lrad{[0, R]}$ as $K \to \infty$. Thus $\F_k[h] = \sqrt{\omega \abs{k}} \F_k[\norm{u_t}_{L^2(\T)} u_t]$, i.e. $h = \halfDT (\norm{u_t}_{L^2(\T)} u_t) = \norm{u_t}_{L^2(\T)} \halfDT u_t$. 
\end{proof}

\section{Proof of Theorems~\ref{thm:radial:main}~and~\ref{thm:radial:multiplicity}}
\label{sec:main_proof}

The proof of \cref{thm:radial:main} is split into two parts. First, using results from Sections~\ref{sec:domain_restriction}--\ref{sec:approximation}, we show in \cref{prop:weak_solution} that there exists a weak solution to the problem \eqref{eq:radial:VGproblem} in the sense of \cref{def:radial:weak_solution} below. In \cref{prop:maxwell_solution} we show that from the solution of \eqref{eq:radial:VGproblem}, one can reconstruct a solution of Maxwell's equations~\eqref{eq:maxwell}--\eqref{eq:polarization}, and that this solution has finite electromagnetic energy per unit segment in $z$-direction.

\begin{definition}\label{def:radial:weak_solution}
    A function $w \colon (0, \infty) \times \T \to \R$ is called a $T$-periodic weak solution to \eqref{eq:radial:VGproblem} if $w$ lies in
	\begin{align*}
		\pspace \coloneqq \set{w \in W^{1,1}_\mathrm{loc}((0, \infty) \times \T) \colon \tfrac{1}{r} w, w_r, w_t \in \Lrad{[0, \infty) \times \T}, \normN{w_t\vert_{[0,R]\times\T}}<\infty}.
	\end{align*}
	and satisfies the equation
	\begin{align*}%
		\int_{[0, \infty) \times \T} \left( w_r \varphi_r + \tfrac{1}{r^2} w \varphi + V(r) w_t \varphi_t + \Gamma(r) N(w_t) \varphi_t \right) \der[r] (r, t) = 0
	\end{align*}
	for all $\varphi \in \pspace$.
\end{definition}

\begin{proposition}\label{prop:weak_solution}
	There exists a nontrivial weak solution to \eqref{eq:radial:VGproblem} in the sense of \cref{def:radial:weak_solution}.
\end{proposition}

We prepare the proof of \cref{prop:weak_solution} with an estimate on the fundamental solutions $\phi_k$.

\begin{lemma}\label{lem:H1_estimate}
	There exists a constant $C > 0$ such that $\displaystyle \norm{\phi_k'}_{\Lrad{[R, \infty)}} \leq C \abs{k} \norm{\phi_k}_{\Lrad{[R, \infty)}}$ holds for all $k \in \Zodd$.
\end{lemma}
\begin{proof}
	By assumption we have
	\begin{align*}
		\norm{\phi_k'' + \tfrac{1}{r} \phi_k'}_{\Lrad{[R,\infty)}}
		= \norm{\tfrac{1}{r^2} \phi_k + k^2 \omega^2 V \phi_k}_{\Lrad{[R,\infty)}}
		\leq k^2 \left( \tfrac{1}{R^2} + \omega^2 \norm{V}_\infty \right) \norm{\phi_k}_{\Lrad{[R,\infty)}}
	\end{align*}
	Due to \cite[Lemma 5.5]{adams} the inequality
	\begin{align*}
		\norm{\phi_k'}_{\Lrad{[R,\infty)}}
		\leq C_0 \left( \eps \norm{\phi_k'' + \tfrac{1}{r} \phi_k'}_{\Lrad{[R,\infty)}} + \tfrac{1}{\eps} \norm{\phi_k}_{\Lrad{[R,\infty)}} \right)
	\end{align*}
	holds for some $C_0 > 0$. Choosing $\eps = \tfrac{1}{\abs{k}}$, the claim follows with $C = C_0 (\tfrac{1}{R^2} + \omega^2 \norm{V}_\infty + 1)$.
\end{proof}

\begin{proof}[Proof~of~\cref{prop:weak_solution}]
	Let $u$ denote the minimizer of $\efct$ obtained through \cref{prop:limit_estimates}. Then $u$ is nonzero by \cref{prop:nonzero_minimizer}. As motivated in \cref{sec:domain_restriction} we define
	\begin{align}\label{eq:solution_reconstruction}
		w(r, t) \coloneqq \begin{cases}
			u(r, t), &r < R, \\
			\sum_{k \in \singularK} \frac{\hat u_k'(R)}{\phi_k'(R)} \phi_k(r) e_k(t)
			+ \sum_{k \in \regularK} \frac{\hat u_k(R)}{\phi_k(R)} \phi_k(r) e_k(t) &r > R.
		\end{cases}
	\end{align}
	First we show that $\hat u_k'(R)$ exists for all $k \in \Zodd$. To do this, let $\eps \in (0, R)$. Then for $\psi \in C_{\mathrm c}^\infty((\eps, R); \C)$ we have
	\begin{align*}
		0 &= \efct'(u)[\Re[\psi(r) e_k(t)]]
		\\ &= \Re \left[\int_0^R \left(\hat u_k'(r) \overline{\psi'(r)} + \left[ \tfrac{1}{r^2} \hat u_k(r) + k^2 \omega^2 V(r) \hat u_k(r) - \ii k \omega \Gamma(r) \F_k[N(u_t)](r) \right] \overline{\psi(r)} \right) \der[r] r \right].
	\end{align*}
	Since $\psi$ was arbitrary, this shows that $\hat u_k \in H^1([\eps, R])$ is a weak solution to
	\begin{align}\label{eq:loc:pointwise_trace:1}
		\hat u_k'' = - \tfrac{1}{r} \hat u_k' + \tfrac{1}{r^2} \hat u_k + k^2 \omega^2 V \hat u_k - \ii k \omega \Gamma \F_k[N(u_t)]
		\quad\text{on}\quad
		[\eps, R].
	\end{align}
	Note that the right-hand side of \eqref{eq:loc:pointwise_trace:1} lies in $L^{\nicefrac43}([\eps, R])$. Thus $\hat u_k \in W^{2, \nicefrac43}([\eps, R])$ and solves \eqref{eq:loc:pointwise_trace:1} pointwise. In particular, we have $\hat u_k \in C^1([\eps, R])$ and therefore $\hat u_k'(R)$ exists.

	Next we show that $w$ lies in $X$. Clearly, $w$ is real-valued, and $\frac{1}{r}w, w_r, w_t \in \Lrad{[0,R]\times\T}$ and $N(w_t)w_t\in \Lrad[1]{[0, R] \times \T}$. Since the antiperiodicity of $w$ forces the zero-th Fourier mode to vanish, we see that $\|w\|_{\Lrad{[R, \infty) \times \T}}$ and hence $\|\frac{1}{r}w\|_{\Lrad{[R, \infty) \times \T}}$ are bounded by $\|w_t\|_{\Lrad{[R, \infty) \times \T}}$. Therefore, it remains to show that $w_r, w_t \in \Lrad{[R, \infty) \times \T}$ since the function values at $r=R$ match according to the construction of $w$. 
	
	Using \eqref{eq:ass:radial:radial:estimates:reformulation}, \cref{prop:limit_estimates}, and \cref{lem:H1_estimate} we find
	\begin{align*}
		\sum_{k \in \regularK} \norm{\hat w_k'}_{\Lrad{[R, \infty)}}^2
		&= \sum_{k \in \regularK} \abs{\frac{\hat u_k(R)}{\phi_k(R)}}^2 \norm{\phi_k'}_{\Lrad{[R, \infty)}}^2
		\\ &\lesssim \sum_{k \in \regularK} k^2 \abs{\hat u_k(R)}^2 
		\lesssim \norm{u(R, \impvar)}_{H^1(\T)}^2
		< \infty,
		\\ \sum_{k \in \regularK} \omega^2 k^2 \norm{\hat w_k}_{\Lrad{[R, \infty)}}^2
		&= \sum_{k \in \regularK} \omega^2 k^2 \abs{\frac{\hat u_k(R)}{\phi_k(R)}}^2 \norm{\phi_k}_{\Lrad{[R, \infty)}}^2
		\\ &\lesssim \sum_{k \in \regularK} k^2 \abs{\hat u_k(R)}^2 
		\lesssim \norm{u(R, \impvar)}_{H^1(\T)}^2
		< \infty.
	\end{align*}
	Since the finite sum $\sum_{k \in \singularK} \frac{\hat u_k'(R)}{\phi_k'(R)} \phi_k(r) e_k(t)$ belongs to $\Hrad{[R, \infty) \times \T}$ this shows that the sum $w(r, t) = \sum_{k \in \Zodd} \hat w_k(r) e_k(t)$ converges in $\Hrad{[R, \infty) \times \T}$. It remains to show that $w$ is a weak $T$-periodic solution to \eqref{eq:radial:VGproblem} in the sense of \cref{def:radial:weak_solution}. That is, we need to verify
	\begin{align*}
		I[\varphi] \coloneqq \int_{[0, \infty) \times \T} \left( w_r \varphi_r + \tfrac{1}{r^2} w \varphi + V(r) w_t \varphi_t + \Gamma(r) N(w_t) \varphi_t \right) \der[r] (r, t) = 0
	\end{align*}
	for all $\varphi \in \pspace$.
	Since $w_r, w, w_t, N(w_t)$ are $\nicefrac{T}{2}$-antiperiodic in time, it follows that $I[\varphi] = 0$ for $\nicefrac{T}{2}$-periodic $\varphi$. So from now on let $\varphi$ be $\nicefrac{T}{2}$-antiperiodic in time.
	We calculate
	\begin{align}\label{eq:loc:3}
		\begin{aligned}
			I[\varphi]
			&= \int_{[0, R] \times \T} \left( u_r \varphi_r + \tfrac{1}{r^2} u \varphi + V(r) u_t \varphi_t + \Gamma(r) N(u_t) \varphi_t \right) \der[r] (r, t)
			\\ &\quad + \sum_{k \in \singularK} \frac{\hat u_k'(R)}{\phi_k'(R)} \int_R^\infty \left(\phi_k' \overline{\hat \varphi_k'} + \tfrac{1}{r^2} \phi_k \overline{\hat \varphi_k} + k^2 \omega^2 V(r) \phi_k \overline{\hat \varphi_k}\right) \der[r] r
			\\ &\quad + \sum_{k \in \regularK} \frac{\hat u_k(R)}{\phi_k(R)} \int_R^\infty \left(\phi_k' \overline{\hat \varphi_k'} + \tfrac{1}{r^2} \phi_k \overline{\hat \varphi_k} + k^2 \omega^2 V(r) \phi_k \overline{\hat \varphi_k}\right) \der[r] r
			\\ &= \int_{[0, R] \times \T} \left( u_r \varphi_r + \tfrac{1}{r^2} u \varphi + V(r) u_t \varphi_t + \Gamma(r) N(u_t) \varphi_t \right) \der[r] (r, t)
			\\ &\quad - \sum_{k \in \singularK} \frac{\hat u_k'(R)}{\phi_k'(R)} \cdot R \phi_k'(R) \overline{\hat \varphi_k(R)}
			- \sum_{k \in \regularK} \frac{\hat u_k(R)}{\phi_k(R)} \cdot R \phi_k'(R) \overline{\hat \varphi_k(R)}
		\end{aligned}
  	\end{align}
	If in addition $\hat \varphi_k(R) = 0$ holds for all $k \in \singularK$, then 
	\begin{align*}
		I[\varphi] &= \int_{[0, R] \times \T} \left( u_r \varphi_r + \tfrac{1}{r^2} u \varphi + V(r) u_t \varphi_t + \Gamma(r) u_t^3 \varphi_t \right) \der[r] (r, t)
		- R \sum_{k \in \regularK} \frac{\phi_k'(R)}{\phi_k(R)} \hat u_k(R) \overline{\hat \varphi_k(R)}
		\\ &= \efct'(u)[\varphi \vert_{[0, R] \times \T}] = 0
	\end{align*}
	where we have used $\varphi \vert_{[0, R] \times \T} \in \espace$.
	Now we want to conclude $I[\varphi] = 0$ in the general case where $\varphi\in X$ but $\hat\varphi_k(R)\neq 0$ for some $k\in\singularK$. Note that since $\varphi$ is real-valued we have the decomposition
	\begin{align*}
		X = \{\varphi\in X: \hat\varphi_k(R)=0 \text{ for all } k\in\singularK, k>0\} \oplus \operatorname{lin}_\R\{ \Re[\psi e_k], \Re[\ii \psi e_k] \colon k\in \singularK, k>0 \}	
	\end{align*}
	for any $\psi\in C_c^\infty((0,\infty))$ with $\psi(R)\not =0$. By linearity it suffices to show the identity $I[\Re[\psi(r) e_k(t)]] = 0 = I[\Re[\ii\psi(r) e_k(t)]]$ for all $k \in \singularK$. Using \eqref{eq:loc:3} we calculate
	\begin{align*}
		&I[\Re[\psi(r) e_k(t)]]
		\\ &\quad= \Re\left[ \int_0^R \left(\hat u_k' \overline{\psi'} + \left[ \tfrac{1}{r^2} \hat u_k + k^2 \omega^2 V(r) \hat u_k - \ii k \omega \Gamma(r) \F_k[N(u_t)](r) \right] \overline{\psi}\right) \der[r]r 
		- R \hat u_k'(R) \overline{\psi(R)} \right] 
		\\ &\quad= \Re\left[ \int_0^R \left(- \hat u_k'' - \tfrac{1}{r} \hat u_k' + \tfrac{1}{r^2} \hat u_k + k^2 \omega^2 V(r) \hat u_k - \ii k \omega \Gamma(r) \F_k[N(u_t)] \right) \overline{\psi} \der[r]r \right]
		= 0,
	\end{align*}
	where the last equality follows from \eqref{eq:loc:pointwise_trace:1} with $\eps\coloneqq\min\supp\psi$. Replacing $\psi$ by $\ii \psi$ in the above calculation, we obtain also $I[\Re[\ii \psi(r) e_k(t)]] = 0$.
\end{proof}	

\begin{proposition}\label{prop:maxwell_solution}
	Let $w$ be a $T$-periodic weak solution to \eqref{eq:radial:VGproblem} in the sense of \cref{def:radial:weak_solution}. Then the fields $\bfD, \bfE, \bfB, \bfH$ given by
	\begin{align*}
		\bfD(\bfx, t) 
		&= \eps_0 \left( (1 + \chi_1(\bfx)) w_t(r, t - \tfrac{1}{c} z) + \chi_3(\bfx) N(w_t)(r, t - \tfrac{1}{c} z) \right) \cdot (-\tfrac{y}{r}, \tfrac{x}{r}, 0)^\top,
		\\ \bfE(\bfx, t) 
		&= w_t(r, t - \tfrac{1}{c} z) \cdot (-\tfrac{y}{r}, \tfrac{x}{r}, 0)^\top,
		\\ \bfB(\bfx, t)
		&= - \left(\tfrac{1}{r} w(r, t - \tfrac{1}{c} z) + w_r(r, t - \tfrac{1}{c} z)\right) \cdot (0, 0, 1)^\top
		- \tfrac{1}{c} w_t(r, t - \tfrac{1}{c} z) \cdot (\tfrac{x}{r}, \tfrac{y}{r}, 0)^\top,
		\\ \bfH(\bfx, t)
		&= \tfrac{1}{\mu_0} \bfB(\bfx, t),
	\end{align*}
	where $\bfx = (x, y, z)$ and $r = \sqrt{x^2 + y^2}$ are weak solutions to Maxwell's equations 
		\eqref{eq:maxwell}--\eqref{eq:polarization} in the sense of \cref{def:weak_maxwell}.
	Furthermore, the electromagnetic energy is finite orthogonal to the direction of propagation, i.e.
	\begin{align*}
		\int_{\R \times \R \times [z_0, z_0 + 1]} \bigl(\bfD \cdot \bfE + \bfB \cdot \bfH\bigr) \der(x, y, z)
	\end{align*}
	is uniformly bounded w.r.t. $z_0, t_0$.
\end{proposition}
\begin{proof}
	We use cylindrical coordinates $(x, y, z) = (r \cos(\theta), r \sin(\theta), z)$. We abbreviate
	\begin{align*}
		e_r = (\tfrac{x}{r}, \tfrac{y}{r}, 0)^\top, 
		\quad
		e_\theta = (- \tfrac{y}{r}, \tfrac{x}{r}, 0)^\top,
		\quad
		e_z = (0, 0, 1)^\top
	\end{align*}
	and use the representations
	\begin{align*}
		\nabla \phi 
		&= \partial_r \phi \cdot e_r 
		+ \tfrac{1}{r} \partial_\theta \phi \cdot e_\theta
		+ \partial_z \phi \cdot e_z,
		\\ 
		\nabla \times \Phi 
		&= \left(\tfrac{1}{r} \partial_\theta \Phi^z - \partial_z \Phi^\theta\right) e_r
		+ \left(\partial_z \Phi^r - \partial_r \Phi^z\right) e_\theta
		+ \tfrac{1}{r} \left( \partial_r(r \Phi^\theta) - \partial_\theta \Phi^r \right) e_z
	\end{align*}
	where $\Phi = \Phi^r e_r + \Phi^\theta e_\theta + \Phi^z e_z$.
	For better readability, we omit the domain $[0, \infty) \times [0, 2 \pi] \times \R \times \R$ when integrating with cylindrical coordinates as well as arguments, so e.g. $w = w(r, t - \tfrac{1}{c} z)$. In particular, $\partial_z w = - \tfrac{1}{c} \partial_t w = - \tfrac{1}{c} w_t$ holds, which we use below. Now let $\phi \in C_{\mathrm c}^\infty(\R^4; \R)$ and $\Phi \in \C_{\mathrm c}^\infty(\R^4; \R^3)$. Identities \eqref{eq:material} and \eqref{eq:polarization} hold by definition, so it remains to check the four integral identities of \cref{def:weak_maxwell}, beginning with
	\begin{align*}
		\int_{\R^4} \bigl(\bfD \cdot \nabla \phi\bigr) \der(x, y, z, t)
		= \int \bigl( \bfD^\theta \partial_\theta \phi\bigr) \der[r](r, \theta, z, t)
		= 0,
	\end{align*}
	where the integral above is zero because $\bfD$ is independent of $\theta$. Next,
	\begin{align*}
		&\int_{\R^4} \bigl(\bfB \cdot \nabla \phi\bigr) \der (x, y, z, t)
		= \int \bigl(-(\tfrac{1}{r} w + w_r) \partial_z \phi - \tfrac{1}{c} w_t \partial_r \phi\bigr) \der[r](r, \theta, z, t)
		\\ &\quad= \int \bigl(- \tfrac{1}{r}\partial_r (r w) \partial_z \phi + \partial_z w \partial_r \phi\bigr) \der[r](r, \theta, z, t)
		= \int \bigl( w (\partial_r \partial_z \phi - \partial_z \partial_r \phi) \bigr) \der[r](r, \theta, z, t)
		= 0.
	\end{align*}
	For the third integral we have
	\begin{align*}
		&\int_{\R^4} \bigl(\bfE \cdot \nabla \times \Phi - \bfB \cdot \partial_t \Phi \bigr) \der (x, y, z, t)
		\\ &\quad = \int \bigl( w_t (\partial_z \Phi^r - \partial_r \Phi^z) + (\tfrac{1}{r} w + w_r) \partial_t \Phi^z + \tfrac{1}{c} w_t \partial_t \Phi^r \bigr) \der[r](r, \theta, z, t)
		\\ &\quad = \int \bigl( \partial_t w (\partial_z \Phi^r - \partial_r \Phi^z) + \tfrac{1}{r} \partial_r(r w) \partial_t \Phi^z - \partial_z w \partial_t \Phi^r \bigr) \der[r](r, \theta, z, t)
		\\ &\quad = \int \bigl( w \left( - \partial_t \partial_z \Phi^r + \partial_t \partial_r \Phi^z - \partial_r \partial_t \Phi^z + \partial_z \partial_t \Phi^r \right) \bigr) \der[r] (r, \theta, z, t)
		= 0.
	\end{align*}
	For the last identity, using integration by parts, that integrals with $\partial_\theta$ vanish, and the definitions of $V, \Gamma$ in \eqref{eq:def:VandGamma}, we have
	\begin{align*}
		&\int_{\R^4} - \bfH \cdot \nabla \times \Phi - \bfD \cdot \partial_t \Phi \der(\bfx, t)
		\\ &\quad= \tfrac{1}{\mu_0} \int \bigl((\tfrac{1}{r} w + w_r) \tfrac{1}{r} (\partial_r(r \Phi^\theta) - \partial_\theta \Phi^r) + \tfrac{1}{c} w_t (\tfrac{1}{r} \partial_\theta \Phi^z - \partial_z \Phi^\theta) \der[r](r, \theta, z, t)
		\\ &\qquad - \int \eps_0 \left((1 + \chi_1) w_t + \chi_3 N(w_t)\right) \partial_t \Phi^\theta \bigr)\der[r](r, \theta, z, t)
		\\ &\quad = \tfrac{1}{\mu_0} \int \bigl( \partial_r w \partial_r \Phi^\theta + \tfrac{1}{r}(\partial_r w \Phi^\theta + w \partial_r \Phi^\theta) + \tfrac{1}{r^2} w \Phi^\theta - \tfrac{1}{c} \partial_z w \partial_t \Phi^\theta \bigr) \der[r](r, \theta, z, t)
		\\ &\qquad - \tfrac{1}{\mu_0} \int \bigl( \eps_0 \mu_0 (1 + \chi_1) \partial_t w \partial_t \Phi^\theta + \eps_0 \mu_0 \chi_3 N(\partial_t w) \partial_t \Phi^\theta \bigr) \der[r](r, \theta, z, t)
		\\ &\quad = \tfrac{1}{\mu_0} \int \bigl(\partial_r w \partial_r \Phi^\theta + \tfrac{1}{r^2} w \Phi^\theta + V(r) \partial_t w \partial_t \Phi^\theta + \Gamma(r) N(w_t) \partial_t \Phi^\theta \bigr)\der[r](r, \theta, z, t)
		\\ &\quad = \tfrac{1}{\mu_0} \int_{[0, \infty) \times \T} \bigl(w_r \varphi_r + \tfrac{1}{r^2} w \varphi + V(r) w_t \varphi_t + \Gamma(r) N(w_t) \varphi_t\bigr)\der[r](r, t) = 0.
	\end{align*}
	where in the last line $w = w(r, t)$ is no longer in traveling coordinates, $\varphi$ is given by
	\begin{align*}
		\varphi(r, t) \coloneqq T \sum_{k \in \Z} \int_{[0, 2 \pi] \times \R} \Phi^\theta(r, \theta, z, t + k T + \tfrac{1}{c} z) \der(\theta, z),
	\end{align*}
	and the last equality holds due to \cref{def:radial:weak_solution}.
	To show finiteness of the energy, using
	\begin{align*}
		\bfD \cdot \bfE + \bfB \cdot \bfH
		= \eps_0 (1 + \chi_1) w_t^2 + \eps \chi_3 N(w_t) w_t + \tfrac{1}{\mu_0} (\tfrac{1}{r} w + w_r)^2 + \tfrac{1}{c^2 \mu_0} w_t^2
	\end{align*}
	we calculate
	\begin{align*}
		&\int_{\R \times \R \times [z_0, z_0 + 1]} \bigl( \bfD \cdot \bfE + \bfB \cdot \bfH \bigr) \der(x, y, z)
		\\ &\quad= \tfrac{2 \pi c}{\mu_0} \int_{[0, \infty) \times [t_0 - \nicefrac{(z_0 + 1)}{c}, t_0 - \nicefrac{z_0}{c}]}
		\Bigl((-V(r) + \tfrac{2}{c^2}) w_t^2 - \Gamma(r) N(w_t) w_t + (\tfrac{1}{r} w + w_r)^2 \Bigr) \der[r](r, t),
	\end{align*}
	which is uniformly bounded w.r.t. $t_0$ and $z_0$ because $V, \Gamma$ are bounded and $w$ lies in $\pspace$.
\end{proof}

Now that we have completed the proof of \cref{thm:radial:main} it remains to show the multiplicity result of \cref{thm:radial:multiplicity}.

\begin{proof}[Proof of \cref{thm:radial:multiplicity}]
	Let $\multK$ denote the (infinite) set of numbers $k_0 \in \Nodd$ for which \ref{ass:radial:nontrivial_sol} holds. For $k_0 \in \multK$ we consider the subspace
	\begin{align*}
		\espacemult \coloneqq \bigl\{u \in \espace \bigm\vert u \text{ is } \tfrac{T}{2 k_0} \text{-antiperiodic in time}\bigr\} \subseteq \espace.
	\end{align*}
	Similarly to the proof of \cref{prop:weak_solution} one can show that $\efct$ attains a minimum value on $\espacemult$ and that from the minimizer, one can construct a weak solution of \eqref{eq:radial:VGproblem} using \eqref{eq:solution_reconstruction}. Here we use that problem \eqref{eq:radial:VGproblem} is compatible with considering $\frac{T}{2 k_0}$-antiperiodic in time functions, i.e., $N(w_t)$ is $\frac{T}{2 k_0}$-antiperiodic in time if $w_t$ has this property. The solution of \eqref{eq:radial:VGproblem} gives rise to a solution of Maxwell's equations by \cref{prop:maxwell_solution}. 
	
	Repeating this for all $k_0 \in \multK$, we obtain a family $\set{(\bfD_{k_0}, \bfE_{k_0},\bfB_{k_0},\bfH_{k_0},) \colon k_0 \in \multK}$ of solutions to Maxwell's equations. 
	Each solution has a minimal nonzero time-period that is a divisor of $\frac{T}{k_0}$. 
	Thus, this family has minimal periods becoming arbitrarily small and therefore infinitely many among the solutions must be mutually distinct.
\end{proof}

\section{Modifications in the slab setting}\label{sec:slab_modifications}

Here we sketch modifications that have to be done in \cref{sec:domain_restriction,sec:approximation,sec:main_proof} in order to prove \cref{thm:slab:main,thm:slab:multiplicity}. First our solution ansatz becomes
\begin{align*}%
	w(x, t) = \begin{cases}
		u(x, t), & \abs{x} < R, \\
		\sum_{k \in \Zodd} \alpha_k \tildePhi_k(\abs{x}) e_k(t), & \abs{x} > R
	\end{cases}
\end{align*}
where $u \in H^1_{\mathrm{anti},\mathrm{even}}([-R, R] \times \T)$ is to be determined and 
\begin{alignat*}{3}
	&\alpha_k = \frac{\hat u_k'(R)}{\tildePhi_k'(R)}, \quad
	&&\hat u_k(R) = 0 \quad
	&&\text{for } k \in \singularK,
	\\
	&\alpha_k = \frac{\hat u_k(R)}{\tildePhi_k(R)}, \quad
	&&\hat u_k'(R) = \frac{\tildePhi_k'(R)}{\tildePhi_k(R)} \hat u_k(R) \quad
	&&\text{for } k \in \regularK.
\end{alignat*}
We use the subscript ``even'' to denote functions that are even in space. 

The restriction to even functions is done in order to shorten this chapter, but it is not necessary. 
For example, one could instead look for functions $u$ that are odd in space, or not impose any spatial symmetry. In the latter case one need not make any symmetry assumptions on $V, \Gamma$ (see assumption \ref{ass:slab:first}) if instead one requires fundamental solutions to exist both on $[R, \infty)$ and $(-\infty, -R]$.

Going back to the problem, we (formally) obtain the boundary value problem 
\begin{align*}
	\begin{cases}
		- u_{xx} - V(x) u_{tt} - \Gamma(x) N(u_t)_t = 0 \text{ in } [0, R] \times \T, \\
		\hat u_k'(R) = \frac{\tildePhi_k'(R)}{\tildePhi_k(R)} \hat u_k(R) \text{ for } k \in \regularK, \\
		\hat u_k(R) = 0 \text{ for } k \in \singularK \\
		u_x(0, \impvar) = 0
	\end{cases}
\end{align*}
for $u$, where the last condition comes from $u$ being even in space. This problem has variational structure as solutions are critical points of 
\begin{align*}
	\widetilde \efct(u) = \int_{[0, R] \times \T} \tfrac12 u_x^2 + \tfrac12 V(x) u_t^2 + \tfrac14 \Gamma(x) N(u_t)u_t \der (x, t) - \tfrac12 \sum_{k \in \regularK} \frac{\tildePhi_k'(R)}{\tildePhi_k(R)} \abs{\hat u_k(R)}^2
\end{align*}
subject to the constraints $\hat u_k(R) = 0$ for $k \in \singularK$.
We can proceed like in \cref{sec:domain_restriction,sec:approximation,sec:main_proof} in order to prove existence, regularity, and multiplicity of some minimizers of $\widetilde \efct$. The main differences to the radial setting are the following: 
First, we do not work in radially weighted Sobolev spaces, so $\der[r]r$ is replaced by $\der x$ and $\Lrad[p]{}$ by $L^p$. Further, the radial Laplacian $\partial_r^2 + \tfrac{1}{r} \partial_r$ is replaced by the $1$d Laplacian $\partial_x^2$. In addition, the term $\tfrac{1}{r^2} w$ is absent in problem~\eqref{eq:slab:problem}, so that this term (and related terms, e.g. $\tfrac{1}{r} u$ in $\efct$ and part~\ref{i:minimizer_estimate:reg2} of \cref{prop:minimizer_estimates}) do not appear in the slab setting.

So we define $\normNtilde{}$ and $\qNtilde{}$ like $\normN{}$ and $\qN{}$ but without the radial weight. Notice that $\widetilde \efct$ is well-defined on the reflexive Banach space
\begin{align*}
	\espacetilde \coloneqq \set{u \in H^1_{\mathrm{anti},\mathrm{even}}([-R, R] \times \T) \colon \normN{u_t}^{\sim} < \infty, \hat u_k(R) = 0 \text{ for } k \in \singularK}.
\end{align*}
More noticeable changes have to be made in the proof of \cref{prop:nonzero_minimizer}. There we made the ansatz 
\begin{align*}
	u(r, t) = \eps I_1(\lambda k_0 r) \left(e_{k_0}(t) + e_{-k_0}(t)\right)
\end{align*}
in order to show that $\inf \efct < 0$, and $I_1$ was a solution of 
\begin{align*}
	(-\partial_r^2 - \tfrac{1}{r}\partial_r + \tfrac{1}{r^2} + 1) I_1 = 0.
\end{align*}
For the slab setting the natural ansatz is
\begin{align*}
	u(x, t) = \eps \cosh(\lambda k_0 x) \left(e_{k_0}(t) + e_{-k_0}(t)\right)
\end{align*}
since $(-\partial_x^2 + 1) \cosh = 0$, which also explains the way we formulated assumption~\ref{ass:slab:last}.

We note that the trace embeddings can be adapted to the slab setting, i.e., the trace map $\tr \colon \espacetilde \to H^{\nicefrac12}(\T), v \mapsto v(R, \impvar)$ is compact and the estimates appearing in \cref{lem:trace,lem:regularized_trace} also hold with $\Lrad[p]{},  \normN{}, \qN{}$ replaced by $L^p, \normNtilde{}, \qNtilde{}$. 
This is because the trace of $v$ only depends on the function $v$ in a small neighborhood of $x = R$, and the radial weight is not singular at $x = R$.

Lastly, the electromagnetic waves reconstructed from the profile $w$ for the slab geometry are given by
	\begin{align*}
		\bfD(\bfx, t) 
		&= \eps_0 \left( (1 + \chi_1(\bfx)) w_t(x, t - \tfrac{1}{c} y) + \chi_3(\bfx) N(w_t)(x, t - \tfrac{1}{c} y) \right) \cdot (0, 0, 1)^\top,
		\\ \bfE(\bfx, t) 
		&= w_t(r, t - \tfrac{1}{c} y) \cdot (0, 0, 1)^\top,
		\\ \bfB(\bfx, t)
		&= (\tfrac{1}{c} w_t(x, t - \tfrac{1}{c} y), w_x(x, t - \tfrac{1}{c} y), 0)^\top,
		\\ \bfH(\bfx, t)
		&= \tfrac{1}{\mu_0} \bfB(\bfx, t),
	\end{align*}
	which can be shown similar to \cref{prop:maxwell_solution} for the cylindrical geometry.

\section{Further regularity estimate and bifurcation phenomenon}\label{sec:further_estimates}

Checking the assumptions \ref{ass:radial:first}--\ref{ass:radial:last} and \ref{ass:slab:first}--\ref{ass:slab:last} one sees that they depend not directly on $\chi_1$ but on $\chi_1^c:= \chi_1- c^{-2}$.
As we show next, for every solution of \cref{thm:radial:main} or \cref{thm:radial:multiplicity} the $L^\infty([0,R]; L^2(\T))$-norm of the $\bfE$-field is finite and can be bounded by a constant depending only on $\chi_1^c$ (as well as on $\chi_3$ and $\kappa$). A possible physical interpretation of this result is described below.

\begin{proposition}%
	Let $\bfD, \bfE, \bfB, \bfH$ be a solution of Maxwell's equation as in \cref{thm:radial:main} or \cref{thm:radial:multiplicity}. Then $\norm{\bfE}_{L^\infty([0,R]; L^2(\T))}$ is finite. The same holds true in the slab setting for solutions from \cref{thm:slab:main} or \cref{thm:slab:multiplicity}.
\end{proposition}
\begin{proof}
	We focus on the radial setting and time-averaged nonlinearity. As in \cref{sec:main_proof} let $\bfD, \bfE, \bfB, \bfH$ be a solution of Maxwell's equations such that $\abs{\bfE}^2 = w_t^2$ where $w$ is a weak solution of \eqref{eq:radial:problem} in the sense of \cref{def:radial:weak_solution}.
	We begin by formally multiplying \eqref{eq:radial:VGproblem} with $-w_{tt}$ and integrating w.r.t. $t$ to obtain
	\begin{align*}
		0 &= \int_\T \left(-w_{rr} - \tfrac{1}{r} w_r + \tfrac{1}{r^2} w - V(r) w_{tt} - \Gamma(r) ((\kappa \ast w_t^2) w_t)_t\right) (- w_{tt}) \der t
		\\ &= \int_\T - w_{trr} w_t - \tfrac{1}{r} w_{tr} w_t + \frac{1}{r^2} w_t^2 + V(r) w_{tt}^2 + 2 \Gamma(r) (\kappa \ast w_t w_{tt}) w_t w_{tt} + \Gamma(r) (\kappa \ast w_t^2) w_{tt}^2 \der t.
	\end{align*}
	Writing $f(r) \coloneqq \tfrac12 \int_\T w_t^2 \der t$, we have
	\begin{align*}
		0 &= -f'' - \tfrac{1}{r} f' + \left[\int_\T \tfrac{1}{r^2} w_t^2+ w_{tr}^2 + V(r) w_{tt}^2 \der t+ \Gamma(r) J''(w_t)[w_{tt},w_{tt}]\right]
	\end{align*}
	where $J(v) \coloneqq \tfrac14 \int_\T (\kappa \ast v^2) v^2 \der t$ is convex by assumption \eqref{eq:ass:kappa} and therefore all terms in the square bracket are non-negative. 
	This combined with $f(0) = 0, f(R) \geq 0$ shows that $f$ is increasing on $[0, R]$. Thus $\norm{w_t}_{L^2(\T)}$ is bounded on $[0, R]$ by $\norm{w_t(R, \impvar)}_{L^2(\T)}$, which is finite by \cref{prop:limit_estimates}.

	To justify this formal calculation, we argue as in \cref{sec:approximation}: since $w\vert_{[0, R] \times \T}$ was obtained as the limit in $\espace$ of a sequence $\eargF$ defined in \cref{lem:energy_limit}, we set $f^K(r) \coloneqq \tfrac12 \int_\T (\eargF_t)^2 \der t$ and get that $f^K \to f$ in $\Lrad{[0, R]}$.
	Since $\eargF \in \espaceF$ and time-derivatives are bounded on $\espaceF$, we have $(f^K)', \tfrac{1}{r} f^K \in L^{1}([0, R])$ so that $f^K$ is continuous and it indeed satisfies $f^K(0) = 0$. The formal argument above can therefore be applied to $f^K$ and yields that $f^K$ is monotone increasing on $[0,R]$. Thus $f$ is monotone increasing and hence bounded by the constant $\tfrac12 C_5$ from \cref{prop:limit_estimates}, completing the proof.

	The proof for the slab setting is similar; the main difference is that at zero we have a Neumann condition $w_x(0, t) = 0$ instead of a Dirichlet condition.
	The proof with the instantaneous nonlinearity follows by setting $\kappa = \delta_0$ above.
\end{proof}

Recall the constitutive relation
\begin{align*}
	\bfD = \eps_0 \bfE + \bfP(\bfE) 
	= \eps_0 (1 + \chi_1(\bfx)) \bfE + \eps_0 \chi_3(\bfx) (\kappa \ast \abs{\bfE}^2) \bfE
\end{align*}
for the time-averaged nonlinearity. The quantity 
\begin{align*}
	\eps_0 (1 + \chi_1(\bfx)) + \eps_0 \chi_3(\bfx) (\kappa \ast \abs{\bfE}^2) 	
\end{align*}
may be called the effective permittivity and can be estimated from below by 
\begin{align*}
	&\eps_0 (1 + \chi_1(\bfx)) + \eps_0 \chi_3(\bfx) (\kappa \ast \abs{\bfE}^2) 
	\\ &\quad \geq \eps_0 (1 +c^{-2} - \|\chi_1^c\|_{L^\infty(\R^3)}) - \eps_0 \|\chi_3\|_{L^\infty(\R^3)} \norm{\kappa}_{L^\infty(\T)} \norm{\bfE}_{L^\infty([0,R]; L^2(\T))}^2.
\end{align*}
As described above, the existence of $\bfE$ hinges on $\chi_1^c=\chi_1- c^{-2}$ and the norm $\norm{\bfE}_{L^\infty([0,R]; L^2(\T))}^2$ only depends on $\chi_1^c$ and not on $\chi_1$. Hence, if $c>0$ is sufficiently small then the effective permittivity is positive, which gives the waveguide the character of a dielectric. In other words, for time-averaged nonlinearities and for sufficiently small propagation speed $c>0$, the fields are not strong enough to change the dielectric character of the waveguide. It is open if the same holds for instantaneous nonlinearities.

Finally, we comment on the bifurcation phenomenon outlined in \cref{sec:discussion_examples} in the context of the cylindrical geometry. We consider $V_d(r)= -(\tilde\chi_{1,d}(r)+1-c^{-2})$, $\Gamma=-\tilde\chi_3(r)$, where the material parameters $\tilde\chi_{1,d}$, $\tilde\chi_3$ are as in \cref{sec:discussion_examples} and where we emphasize the $d$-dependance of $\tilde\chi_1$ and $V$ by adding a lower index $d$. In fact, $d$ will be seen as a bifurcation parameter. Due to the ansatz $\bfE(\bfx, t) = w_t(r, t - \tfrac{1}{c}z) \cdot (-\tfrac{y}{r},\tfrac{x}{r},0)^\top$ and the fact that $u(\cdot,t)=w(\cdot,t)\mid_{[0,R]}$ solves the boundary value problem \eqref{eq:radial:bounded_problem}, the bifurcation phenomenon can be explained on the level of $u$ as a solution of the $d$-dependent boundary value problem \eqref{eq:radial:bounded_problem} on $[0,R]\times\T$. Recall that on $[0,R]$ the function $V_d(r)=-(d+1-c^{-2})$ is just a positive constant.

Let us first fix a value $d^\ast$ as in Theorem~\ref{thm:main_example} so that assumptions \ref{ass:radial:first}--\ref{ass:radial:last} hold. Then we consider the linear eigenvalue problem 
\begin{align}\label{eq:radial:ev_problem}
	\begin{cases}
		- u_{rr} - \tfrac{1}{r} u_r + \tfrac{1}{r^2} u +\underbrace{(d^\ast+1-c^{-2})}_{-V_{d^\ast}(r)} u_{tt} = \lambda u_{tt} \text{ in } [0, R] \times \T,
		\\ \hat u_k'(R) = \frac{\phi_k'(R)}{\phi_k(R)} \hat u_k(R) \text{ for } k \in \regularK,
		\\ \hat u_k(R) = 0 \text{ for } k \in \singularK.
	\end{cases}
\end{align}
The smallest eigenvalue $\lambda$ can be obtained by minimizing 
$$
E_{d^\ast,\mathrm{lin}}(u)= \int_{[0, R] \times \T} \left( u_r^2 + \left(\tfrac{1}{r} u\right)^2 + V_{d^\ast}(r) u_t^2\right)  \der[r] (r, t) - 2\efctB(u)
$$
subject the constraint
$$
\int_{[0, R] \times \T} u_t^2 \der[r](r,t)=1
$$
on the space 
$$
Y_\mathrm{lin} = \left\{u \in W^{1,1}_{\mathrm{loc},\mathrm{anti}}((0, R] \times \T) \mid u_r, \tfrac{1}{r} u, u_t \in \Lrad{[0, R] \times \T}\right\}.
$$
Since assumptions \ref{ass:radial:first}--\ref{ass:radial:last} hold for $d^\ast$, the negative minimum $\lambda<0$ is attained. It appears as a Lagrange multiplier which coincides with the smallest eigenvalue. Moreover, the minimizer $u_\mathrm{lin}$ satisfies \eqref{eq:radial:ev_problem} so that 
\begin{align*}
	\begin{cases}
		- u_{\mathrm{lin},rr} - \tfrac{1}{r} u_{\mathrm{lin},r} + \tfrac{1}{r^2} u_\mathrm{lin} +\underbrace{(d+1-c^{-2})}_{-V_{d}(r)} u_{\mathrm{lin},tt} = (d-d_\ast) u_{\mathrm{lin},tt} \text{ in } [0, R] \times \T,
		\\ \hat u_{\mathrm{lin},k}'(R) = \frac{\phi_k'(R)}{\phi_k(R)} \hat u_{\mathrm{lin},k}(R) \text{ for } k \in \regularK,
		\\ \hat u_{\mathrm{lin},k}(R) = 0 \text{ for } k \in \singularK.
	\end{cases}
\end{align*}
where we have set $d_\ast=d^\ast+\lambda$. In particular, for the bifurcation parameter $d\in (d_\ast,d^\ast]$ we find that 
$$
E_{d,\mathrm{lin}}(u_\mathrm{lin})=\int_{[0, R] \times \T} \left( \tfrac12 u_{\mathrm{lin},r}^2 + \tfrac12 \left(\tfrac{1}{r} u_\mathrm{lin}\right)^2 + \tfrac12 V_{d}(r) u_{\mathrm{lin},t}^2\right)  \der[r] (r, t) - \efctB(u_\mathrm{lin})<0.
$$
Hence, for a sufficiently small multiple $\varepsilon>0$ we can insert $\varepsilon u_\mathrm{lin}$ into the functional $\efct_d$ for the nonlinear problem and get $\efct_d(\varepsilon u_\mathrm{lin})<0$. This shows that $E_d^\star=\inf_{Y_N} \efct_d<0$ and it is therefore the substitute for \ref{ass:radial:last} which we do not verify for $d\in (d_\ast,d^\ast)$. Since \ref{ass:radial:first}--\ref{ass:radial:secondlast} continue to hold for all $d\in (d_\ast,d^\ast]$ we conclude that the nonlinear problem 
\begin{align*}
	\begin{cases}
		- u_{rr} - \tfrac{1}{r} u_r + \tfrac{1}{r^2} u - V_d(r) u_{tt} - \Gamma(r) N(u_t)_t = 0 \text{ in } [0, R] \times \T,
		\\ \hat u_k'(R) = \frac{\phi_k'(R)}{\phi_k(R)} \hat u_k(R) \text{ for } k \in \regularK,
		\\ \hat u_k(R) = 0 \text{ for } k \in \singularK.
	\end{cases}
\end{align*}
has a nontrivial ground state $u^d$. Let us now show that indeed $u^d\to 0$ in suitable norms as $d\to d_\ast$, which shows bifurcation from the zero-solution at $d=d_\ast$ and continuation of solutions as $d$ runs from $d_\ast$ up to the primarily chosen value $d^\ast$. 

\begin{lemma} For $d\in [d_\ast,d^\ast]$ any minimizer $u^d$ of $E_d$ satisfies 
$$
\norm{u^d}_{\espace}  = \landauO((d-d_\ast)^{\nicefrac14})
$$
as $d\searrow d_\ast$.
\end{lemma}

\begin{proof} We first show that $\normN{u^d_t}$ is uniformly bounded for $d\in [d_\ast,d^\ast]$. As in the proof of \cref{prop:energy_properties} we find 
\begin{align*}
E_d(u) &= \int_{[0, R] \times \T} \left( \tfrac12 u_r^2 + \tfrac12 \left(\tfrac{1}{r} u\right)^2 + \tfrac12 V_d(r) u_t^2 + \tfrac14 \Gamma(r) N(u_t) u_t\right) \der[r] (r, t) - \efctB(u)  \\
& \geq \int_{[0, R] \times \T}\left( \tfrac12 u_r^2 + \tfrac14 \Gamma(r) N(u_t) u_t\right) \der[r] (r, t) - C_0 \norm{u(R, \impvar)}_{H^{\nicefrac12}(\T)}^2 \\
& \geq \int_{[0, R] \times \T}\left( \tfrac12 u_r^2 + \tfrac14 \Gamma(r) N(u_t) u_t\right) \der[r] (r, t) - \frac{1}{4}\norm{u_r}_{\Lrad[2]{[0,R]\times\T}}^2  - C_0 C(\tfrac{1}{4C_0}) \normN{u_t}^2 \\
& \gtrsim \normN{u_t}^2(\normN{u_t}^2-C_0 C(\tfrac{1}{4C_0})).
\end{align*}
If we insert $u^d$ and use $E_d^\star=E_d(u^d) \leq 0$ the claim on the uniform boundedness of $\normN{u^d_t}$ follows. 

Next we claim that $E_d^\star= \landauO(d-d_\ast)$. To see this, we find
\begin{align*}
E_d(u) & \geq E_{d,\mathrm{lin}}(u) = \int_{[0, R] \times \T} \left( \tfrac12 u_r^2 + \tfrac12 \left(\tfrac{1}{r} u\right)^2 + \tfrac12 V_d(r)\right) u_t^2 \der[r] (r,t) - \efctB(u) \\
& \geq (d_\ast-d)\int_{[0, R] \times \T} u_t^2 \der[r] (r,t) \\
& \gtrsim (d_\ast-d) \normN{u_t}^2
\end{align*}
by \cref{rem:normN}. The claim follows by inserting $u=u^d$. 

Now we can use the equality 
$$
0>E_d^\star = E_d(u^d)-\frac{1}{2} E_d'(u^d)[u^d] = -\tfrac14 \int_{[0, R] \times \T} \Gamma(r) N(u^d_t) u^d_t \der[r] (r, t)
$$
and the previous step to conclude $\normN{u_t^d}^4 = \landauO(E_d^\ast)=\landauO(d-d_\ast)$. Finally, 
\begin{align*}
0 &= E_d'(u^d)[u^d] = \int_{[0, R] \times \T} \left((u^d_r)^2 + \left(\tfrac{1}{r} u^d\right)^2 + V_d(r) (u^d_t)^2 + \Gamma(r) N(u^d_t) u^d_t\right) \der[r] (r, t) - 2\efctB(u^d) \\
&\geq \tfrac12 \norm{u^d_r}_{\Lrad[2]{[0,R]\times\T}}^2 +\norm{\tfrac{u^d}{r}}_{\Lrad[2]{[0,R]\times\T}}^2+(c^{-2}-1-d^\ast)\norm{u^d_t}_{\Lrad[2]{[0,R]\times\T}}^2 - 2C_0 C(\tfrac{1}{4C_0}) \normN{u_t^d}^2
\end{align*}
implies that $\norm{u^d_r}_{\Lrad[2]{[0,R]\times\T}}, \norm{\tfrac{u^d}{r}}_{\Lrad[2]{[0,R]\times\T}}, \norm{u^d_t}_{\Lrad[2]{[0,R]\times\T}} = \landauO(\normN{u_t^d}) = \landauO((d-d_\ast)^{\nicefrac14})$ as claimed.
\end{proof}

	\appendix

\section{The fractional Laplacian} \label{sec:fractional}

In this section, we present some results on the fractional Laplacian on the torus, and related spaces. 
They are not given in the most general form available, but in a form which is sufficient for our applications.
We begin by giving the definition of the fractional Sobolev-Slobodeckij space $W^{s,p}(\T)$.

\begin{definition}
	For $s \in (0, 1)$, $\sigma>0$ and $p \in [1, \infty)$, we set
	\begin{align*}
		\seminorm{f}_{W^{s, p}(\T)}^p &\coloneqq \int_\T \int_\R \frac{\abs{f(t) - f(t+h)}^p}{\abs{h}^{1+sp}} \der h \der t
		\\ \norm{f}_{W^{s, p}(\T)}^p &\coloneqq \norm{f}_{L^p(\T)}^p + \seminorm{f}_{W^{s, p}(\T)}^p,
		\\ W^{s, p}(\T) &\coloneqq \set{f \in L^p(\T) \colon \seminorm{f}_{W^{s, p}(\T)} < \infty},
	\end{align*}
	as well as $W^{0,p}(\T) \coloneqq L^p(\T)$ and $W^{s, \infty}(\T) \coloneqq C^s(\T)$. Setting
	\begin{align*}
		\tilde K_\sigma(h) \coloneqq T \sum_{k \in \Z} \abs{h + k T}^{-1-\sigma}
		\quad\text{we have}\quad
		\seminorm{f}_{W^{s, p}(\T)}^p = \int_\T \int_\T \tilde K_{sp}(h) \abs{f(t) - f(t+h)}^p \der h \der t
	\end{align*}
	Note that $\tilde K_\sigma(h) \simeq d(0, h)^{-1-\sigma}$ where $d$ denotes the metric on $\T$.
\end{definition}

Next we show that the fractional Laplacian $\fracDT{s}f = \F^{-1}[\abs{\omega k}^s \hat f_k]$ can be expressed using a singular integral.
\begin{lemma}\label{lem:singular_integral_fractional_laplace}
	Let $s \in (0, 2)$ and $f \in C^{1,1}(\T)$. Then 
	\begin{align*}
		\fracDT{s} f(t) 
		&= C_s \int_\R \frac{2 f(t) - f(t+h) - f(t-h)}{\abs{h}^{1+s}} \der h
		\\ &= \int_\T K_s(h) \left( 2 f(t) - f(t+h) - f(t-h) \right) \der h
	\end{align*}
	holds where
	\begin{align*}
		C_s \coloneqq  \left(2 \int_\R \frac{1 - \cos(\eta)}{\abs{\eta}^{1+s}} \der \eta \right)^{-1}
		\qquad \text{and} \qquad
		K_s(h) \coloneqq C_s \tilde K_s(h).
	\end{align*}
\end{lemma}

For a proof see \cite[Theorem~2.5]{roncal_stinga} where the case $T = 2 \pi$ is discussed\footnote{The constant in \cite[Theorem~2.5]{roncal_stinga} has a typo: $\sigma$ needs to be replaced by $2 \sigma$. Then the constant in \cite{roncal_stinga} coincides with $C_s$ up to a factor of $2$.}. In \cite{roncal_stinga} a principal value formulation is used which can be avoided by using the above symmetric representation, as discussed in \cite[Lemma~3.2]{hitch}.

Related to \cref{lem:singular_integral_fractional_laplace} is the fact that the seminorm $\seminorm{f}_{W^{s,2}(\T)}$ coincides with $\norm{\fracDT{s}f}_2$ up to a constant, and in particular that $W^{s,2}(\T) = H^s(\T)$. 

\begin{lemma}%
	Let $s \in (0, 1)$ and $u \in H^s(\T)$. Then
	\begin{align*}
		\norm{\fracDT{s} u}_{L^2(\T)}^2 
		= C_{2s} \seminorm{u}_{W^{s,2}(\T)}^2
		= \int_\T \int_\T K_{2s}(h) \abs{f(t) - f(t+h)}^2 \der h \der t
	\end{align*}
	holds.
\end{lemma}

This can be shown in the same way as \cite[Proposition~3.4]{hitch}. 
Formally, it follows from \cref{lem:singular_integral_fractional_laplace} (with $2s$ instead of $s$) by multiplying the identity with $f$ and then integrating.

Next, let us note that the fractional Gagliardo-Nirenberg inequality of \cref{lem:gagliardo-nirenberg} and the Sobolev embedding theorem of \cref{lem:sobolev_embedding} hold on the torus.

\begin{lemma} \label{lem:gagliardo-nirenberg}
	Let $s_1, s_2 \in [0, 1)$, $\theta \in (0, 1)$, $p_1, p_2 \in [1, \infty]$ and $s = \theta s_1 + (1 - \theta) s_2$, $\frac{1}{p} = \frac{\theta}{p_1} + \frac{1 - \theta}{p_2}$.
	Then $\norm{f}_{W^{s,p}(\T)} \lesssim \norm{f}_{W^{s_1, p_1}(\T)}^\theta \norm{f}_{W^{s_2, p_2}(\T)}^{1-\theta}$ holds.
\end{lemma}
\begin{lemma} \label{lem:sobolev_embedding}
	Let $s_1, s_2 \in (0, 1), p_1, p_2 \in [1, \infty]$ with $s_2 < s_1$ and $\frac{1}{p_2} - s_2 \geq \frac{1}{p_1} - s_1$  with strict inequality for $s_1 p_1 = 1$. Then $\norm{f}_{W^{s_2, p_2}(\T)} \lesssim \norm{f}_{W^{s_1,p_1}(\T)}$ holds.
\end{lemma}

\begin{proof}[Proof of \cref{lem:gagliardo-nirenberg,lem:sobolev_embedding}]
	We first remark that both results hold if $W^{s,p}(\T)$ is replaced by $W^{s, p}(I)$ where $I$ is a bounded interval. This space is defined by
	\begin{align*}
		W^{s, p}(I) = \set{f \in L^p(I) \colon \seminorm{f}_{W^{s,p}(I)} \coloneqq \left(\int_I \int_I \frac{\abs{f(x) - f(y)}^p}{\abs{x-y}^{1+sp}} \der x \der y\right)^{\frac1p} < \infty}
	\end{align*}
	for $s \in (0, 1)$, $1\leq p < \infty$, and $W^{s,\infty}(I) = C^s(I)$, $W^{0, p}(I) = L^p(I)$.

	Indeed, on intervals the Gagliardo-Nirenberg inequality holds by \cite{brezis_mironescu}. Also on intervals we have from \cite{hitch} that $W^{s_1, p_1}(I) \embeds L^q(I)$ for $\frac{1}{q} = \frac{1}{p_1} - s_1$ if $s_1 p_1 < 1$, $W^{s_1, p_1}(I) \embeds L^q(I)$ for $1\leq q < \infty$ if $s_1 p_1 = 1$, and that $W^{s_1, p_1}(I) \embeds C^\alpha(I)$ for $- \alpha = \frac{1}{p_1} - s_1$ for $s_1 p_1 > 1$. From these properties we can deduce the claimed embedding estimate on intervals by applying the Gagliardo-Nirenberg inequality (on intervals). 

	Then, the statements of \cref{lem:gagliardo-nirenberg,lem:sobolev_embedding} follow from the results on intervals since the norms $\norm{f}_{W^{s,p}(\T)}$ and $\norm{f}_{W^{s, p}([0, 2 T])}$ are equivalent for periodic $f$.
\end{proof}

Additionally, the following version of the Kenig-Ponce-Vega inequality (cf. \cite{kenig_ponce_vega}) holds on the torus.

\begin{lemma} \label{lem:leibniz_defect_estimate}
	Let $f, g \in C^\infty(\T)$, $s \in (0, 2)$, $s_1, s_2 \in (0, 1)$ with $s < s_1 + s_2$, and $p, p_1, p_2 \in [1, \infty]$ with $\frac{1}{p} = \frac{1}{p_1} + \frac{1}{p_2}$. Then
	\begin{align*}
		\norm{\fracDT{s} (fg) - f \fracDT{s} g - g \fracDT{s} f}_{L^p(\T)} 
		\lesssim \seminorm{f}_{W^{s_1, p_1}(\T)} \seminorm{g}_{W^{s_2, p_2}(\T)}
	\end{align*}
	holds.
\end{lemma}

\begin{proof}
	We only consider the case $p_1,p_2 < \infty$. Using \cref{lem:singular_integral_fractional_laplace}, we write
	\begin{align*}
		&\fracDT{s}(fg)(t) - f(t) \fracDT{s} g(t) - g(t) \fracDT{s} f(t)
		\\ &\quad = \int_\T K_s(h) \left(2 f(t)g(t) - f(t+h) g(t+h) - f(t-h) g(t-h)\right) \der h
		\\ &\qquad - f(t) \int_\T K_s(h) \left(2 g(t) - g(t+h) - g(t-h)\right) \der h
		\\ &\qquad - g(t) \int_\T K_s(h) \left(2 f(t) - f(t+h) - f(t-h)\right) \der h
		\\ &\quad = - \int_\T K_s(h) \bigl( \left(f(t) - f(t+h)\right)\left(g(t) - g(t+h)\right) + \left(f(t) - f(t-h)\right)\left(g(t) - g(t-h)\right) \bigr) \der h
		\\ &\quad = - 2 \int_\T K_s(h) \left(f(t) - f(t+h)\right)\left(g(t) - g(t+h)\right) \der h
	\end{align*}

	Now let $r \coloneqq 1 - \frac{1}{p} + s - s_1 - s_2 < \frac{1}{p'}$. Using H\"older's inequality twice we estimate
	\begin{align*}
		&\norm{\int_\T K_s(h) \left(f(t) - f(t+h)\right)\left(g(t) - g(t+h)\right) \der h}_{L^p(\T)}
		\\ &\quad \leq \int_\T \norm{K_s(h) \abs{f(t) - f(t+h)}\abs{g(t) - g(t+h)}}_{L^p(\T)} \der h
		\\ &\quad \leq \norm{d(0,h)^{-r}}_{L^{p'}(\T)} \norm{K_s(h) d(0, h)^r \abs{f(t) - f(t+h)} \abs{g(t) - g(t+h)}}_{L^p(\T \times \T)}
		\\ &\quad \lesssim \norm{\sqrt[p_1]{\tilde K_{s_1 p_1}(h)} \abs{f(t) - f(t+h)}}_{L^{p_1}(\T \times \T)}
		\norm{\sqrt[p_2]{\tilde K_{s_2 p_2}(h)} \abs{f(t) - f(t+h)}}_{L^{p_2}(\T \times \T)}
		\\ &\quad = \seminorm{f}_{W^{s_1,p_1}(\T)} \seminorm{g}_{W^{s_2,p_2}(\T)}
	\end{align*}
	where we have also used that
	\begin{align*}
		K_s(h) d(0,h)^r 
		\simeq d(0,h)^{-1-s + r} 
		= d(0,h)^{-\frac{1}{p_1} - s_1 - \frac{1}{p_2} - s_2}
		\simeq \sqrt[p_1]{\tilde K_{s_1 p_1}(h)} \sqrt[p_2]{\tilde K_{s_2 p_2}(h)}.
		&\qedhere
	\end{align*}
\end{proof}

Lastly, we make the following observation on derivatives of time-antiperiodic functions. The proof, which follows via the Fourier transform from the fact that the zero Fourier mode vanishes, is omitted.
\begin{lemma}\label{lem:derivative_estimates}
	Let the function $v\in L^1(\T)$ be $\tfrac{T}{2}$-antiperiodic in time. Then for any $s > 0$ and $\sigma\in\R$ we have
	\begin{align*}%
		\norm{v}_{L^2(\T)} \leq \frac{1}{\omega^s} \norm{\fracDT{s} v}_{L^2(\T)}
		\quad\text{and thus}\quad
		\norm{\fracDT{\sigma} v}_{L^2(\T)} \leq \frac{1}{\omega^s} \norm{\fracDT{\sigma + s} v}_{L^2(\T)}.
	\end{align*}
	If furthermore $\F_k[v] = 0$ for $\abs{k} < K$, then these estimates can be improved to
	\begin{align*}%
		\norm{v}_{L^2(\T)} \leq \frac{1}{(K \omega)^s} \norm{\fracDT{s} v}_{L^2(\T)}
		\quad\text{and}\quad
		\norm{\fracDT{\sigma} v}_{L^2(\T)} \leq \frac{1}{(K \omega)^s} \norm{\fracDT{\sigma + s} v}_{L^2(\T)}.
	\end{align*}
\end{lemma}

\section{Properties of the nonlinearities}\label{sec:functional_av}

For the instantaneous nonlinearity, it is clear that the function $\efctN$ is convex. In the time-averaged case this follows from assumption~\eqref{eq:ass:kappa} together with $\Gamma \geq 0$. Next we discuss two conditions that are sufficient for convexity in the time-averaged setting, i.e. \eqref{eq:ass:kappa}.

\begin{lemma}\label{lem:convexity_conditions}
	The convexity assumption of \eqref{eq:ass:kappa} on $\kappa$ is satisfied for example if the other assumptions hold and either $\max \kappa \leq 2 \min \kappa$ or $\hat \kappa_k \geq 0$ for all $k \in \Z$, or more generally if $\kappa$ is a sum of functions satisfying these conditions.
\end{lemma}
\begin{proof}
	Set $f \colon L^4(\T) \to \R, f(v) = \int_\T (\kappa \ast v^2) v^2$. Then using that $\kappa$ is even we calculate
	\begin{align*}
		f''(v)[u,u] = 4 \int_\T (\kappa \ast v^2) u^2 \der t + 8 \int_\T (\kappa \ast u v) u v \der t.
	\end{align*}

	\textit{Part 1:} If $\max \kappa \leq 2 \min \kappa$, with $c \coloneqq (\min \kappa + \max \kappa) / 2$ we can estimate
	\begin{align*}
		f''(v)[u,u] 
		&= 4 \int_\T (\kappa \ast v^2) u^2 \der t + 8 c \left(\int_\T u v\der t\right)^2 + 8 \int_\T \left((\kappa - c) \ast uv\right) uv \der t
		\\ &\geq 4 \min \kappa \norm{uv}_2^2 - 8 \norm{\kappa - c}_\infty \norm{uv}_2^2 \geq 0.
	\end{align*}

	\textit{Part 2:} If instead $\hat \kappa_k \geq 0$, we can estimate
	\begin{align*}
		f''(v)[u,u] \geq 8 \int_\T (\kappa \ast u v) uv \der t
		= 8 \sum_{k \in \Z} \hat \kappa_k \abs{\F_k(uv)}^2 \geq 0.
		&\qedhere
	\end{align*}
\end{proof}

Next we aim at lower bounds for $E'(u)[\DT u]$. Using integration by parts, one sees that the quadratic terms appearing in $E'(u)[\DT u]$ are $L^2$-norms of suitable fractional derivatives of $u$. 
In the next two lemmas, we investigate the remaining non-quadratic term $\int N(u_t) \DT u_t$. We begin with the instantaneous nonlinearity.

\begin{remark}  \label{rem:examples:kappa} 
Let us give a few examples of kernels $\tilde \kappa$ describing the nonlinear polarization (cf. \eqref{eq:nonlinear_polarization}) that lead via $\kappa(t) = T \sum_{k \in \Z} \tilde \kappa(t + k T)$ to admissible potentials $\kappa$ for \eqref{eq:ass:kappa}.
\begin{itemize} 
\item[(a)] First, we consider 
	\begin{align*}
		\tilde \kappa(t) = \begin{cases}
			0, & t < 0, \\
			\left(T^4 + 4 t^4\right)^{-1} t, & t \geq 0
		\end{cases}
	\end{align*}
	where $c > 0$. Let us show that the resulting $\kappa$ is admissible. To do this, we write
	\begin{align*}
		\tilde \kappa(t) = \frac{1}{2 T} \left( \frac{1}{T^2 + (2 t - T)^2} - \frac{1}{T^2 + (2 t + T)^2}\right)
		\quad\text{for } t \geq 0,
	\end{align*}  
	so $\kappa(t)$ is a telescoping series with value
	\begin{align*}
		\kappa(t) = \frac{1}{2 T \left(T^2 + (2 t - T)^2\right)}
		\quad\text{for } t \in [0, T).
	\end{align*}
	We see that $\kappa$ is even about $\frac{T}{2}$, and by periodicity also even about $0$, and that $\min\kappa = \kappa(0) = \frac{1}{4 T^3}, \max\kappa = \kappa(\frac{T}{2}) = \frac{1}{2 T^3}$ hold. 
	Since $\kappa$ is Lipschitz continuous, this combined with \cref{lem:convexity_conditions} show that $\kappa$ satisfies \eqref{eq:ass:kappa}. 
\item[(b)] More generally, $\tilde\kappa(t) = \bbone_{t \geq 0} \left[ g(2 t - T) - g(2 t + T)\right]$ with even, H\"older-continuous $g\colon \R \to \R$ is an admissible example if $\max_{[0, T]} g \leq 2 \min_{[0, T]} g$ holds.
\item[(c)] Let us give another example: Consider 
	\begin{align*}
		\tilde \kappa(t) = \sum_{n \in \N_0} \alpha_n \bbone_{[n T, (n+1)T)}(t)
	\end{align*}
	where $(\alpha_n) \in \ell^1$ with $\sum_{n \in \N_0} \alpha_n = \frac{1}{T}$. Then $\kappa \equiv 1$ and therefore it satisfies \eqref{eq:ass:kappa}.
\item[(d)] Finally, using a Debye-type exponential decay in the kernel function cf. \cite{debye}, let us consider $\tilde\kappa(t) = \alpha\ee^{-\beta t}\bbone_{t \geq 0}$ and its discretized version $\tilde\kappa_d(t) = \alpha\sum_{n=0}^\infty \ee^{-\beta nT} \bbone_{[nT,(n+1)T)}$ with $\alpha, \beta>0$. Subject to the choice $\alpha=(1-\ee^{-\beta T})/T$ the discretized version clearly falls into the category (c) whereas for the continuous version we get $\kappa(t) =\ee^{-\beta t}$ for $t\in [0,T)$ so that $\kappa$ is neither even nor continuous on $\T$ and hence does not satisfy \eqref{eq:ass:kappa}. Therefore our results do not apply to $\tilde\kappa$, but can be used for $\tilde\kappa_d$. Clearly, the smaller $T>0$ the better $\tilde\kappa_d$ approximates $\tilde\kappa$, and our results provide existence of breathers with frequencies tending to infinity as $T\searrow 0$. This, however, does not allow for any conclusion about breathers for nonlinear Maxwell equations with Debye-type exponential decay kernel.
\end{itemize}
\end{remark}

\begin{lemma}\label{lem:nonlinear_derivative_inequality}
	The inequality 
	\begin{align}\label{eq:loc:nonlinear_derivative_inequality}
		2 \int_\T v^3 \cdot \DT v \der t \geq \int_{\T} \left(\halfDT \left( v \abs{v} \right)\right)^2 \der t
	\end{align}
	holds for all $v \in C^\infty(\T)$.
\end{lemma}
\begin{proof}
	We first encountered an estimate similar to \eqref{eq:loc:nonlinear_derivative_inequality} in \cite[Proposition 2.3]{cordoba}, and we prove \eqref{eq:loc:nonlinear_derivative_inequality} in a similar fashion.
	Note that $v \abs{v} \in C^{1,1}(\T) \subseteq H^{\nicefrac12}(\T)$. Thus both sides of \eqref{eq:loc:nonlinear_derivative_inequality} are well-defined and we may use symmetry to obtain
	\begin{align*}
		\int_\T \left( \halfDT \left(v \abs{v}\right) \right)^2 \der t
		= \int_\T v \abs{v} \cdot \DT \left(v \abs{v}\right) \der t.
	\end{align*}
	Using the representation of \cref{lem:singular_integral_fractional_laplace}, we calculate
	\begin{align*}
		&2 \int_\T v^3 \cdot \DT v \der t - \int_\T v \abs{v} \cdot \DT(v \abs{v}) \der t
		\\ &\quad= C \int_\T \int_\R \frac{2}{h^2} v(t)^3 \bigl( 2 v(t) - v(t + h) - v(t - h) \bigr)
		\\ &\quad\phantom{= C \int_\T \int_\R} - \frac{1}{h^2} v(t) \abs{v(t)} \bigl( 2 v(t) \abs{v(t)} - v(t+h) \abs{v(t+h)} - v(t-h)\abs{v(t-h)} \bigr) \der h \der t
		\\ &\quad= C\int_\R \frac{1}{h^2} \int_\T 2 v(t)^3 \bigl( v(t) - v(t + h) \bigr) - v(t) \abs{v(t)} \bigl( v(t) \abs{v(t)} - v(t + h) \abs{v(t + h)} \bigr)
		\\ &\quad\phantom{= C\int_\R \frac{1}{h^2} \int_\T} +2 v(t)^3 \bigl( v(t) - v(t - h) \bigr) - v(t) \abs{v(t)} \bigl( v(t) \abs{v(t)} - v(t - h) \abs{v(t - h)} \bigr) \der t \der h
		\\ &\quad= C\int_\R \frac{1}{h^2} \int_\T 2 v(t)^3 \bigl( v(t) - v(t + h) \bigr) - v(t) \abs{v(t)} \bigl( v(t) \abs{v(t)} - v(t + h) \abs{v(t + h)} \bigr)
		\\ &\quad\phantom{= C\int_\R \frac{1}{h^2} \int_\T} +2 v(t + h)^3 \bigl( v(t + h) - v(t) \bigr) - v(t + h) \abs{v(t + h)} \bigl( v(t + h) \abs{v(t + h)} - v(t) \abs{v(t)} \bigr) \der t \der h
		\\ &\quad = C \int_\R \frac{1}{h^2} \int_\T v(t)^4 + v(t + h)^4 - 2 v(t)^3 v(t + h) - 2 v(t) v(t+h)^3 + 2 v(t) \abs{v(t)} v(t + h) \abs{v(t + h)} \der t \der h
	\end{align*}
	Next we claim that the last integrand is everywhere non-negative. To see this, abbreviate $a \coloneqq v(t)$, $b \coloneqq v(t + h)$. If $a$ and $b$ have the same sign, we find
	\begin{align*}
		&a^4 + b^4 - 2 a^3 b - 2 a b^3 + 2 a \abs{a} b \abs{b} = \left( a^2 + b^2 \right) \left(a - b\right)^2 \geq 0.
	\intertext{If $a$ and $b$ have opposite signs, we instead calculate}
		&a^4 + b^4 - 2 a^3 b - 2 a b^3 + 2 a \abs{a} b \abs{b} = \left(a^2 - b^2\right)^2 - 2 a b (a^2 + b^2) \geq 0.
	\end{align*}
	This completes the proof.
\end{proof}

The counterpart for the temporally averaged nonlinearity reads as follows. Its proof is very different from the proof of the previous lemma. 

\begin{lemma}\label{lem:nonlinear_derivative_inequality_av} There exists constants $c_1, C_2>0$ such that
\begin{align*}
	\int_\T (\kappa\ast v^2) v \DT v\der t
	\geq c_1 \norm{v}_{L^2(\T)}^2 \norm{\DT v}_{L^2(\T)}^2 - C_2 \norm{v}_{L^2(\T)}^4
\end{align*}
holds for all $v\in C^\infty(\T)$.  
\end{lemma}

\begin{proof} 
By \ref{eq:ass:kappa}, $\kappa \in C^{\alpha}(\T)$. Inspired by the famous Kenig-Ponce-Vega inequality~\cite{kenig_ponce_vega}, we define the Leibniz-defect for the fractional half-derivative as
\begin{align*}
	\delta = \halfDT ((\kappa \ast v^2) v) - v \halfDT(\kappa \ast v^2) - (\kappa \ast v^2) \halfDT v.
\end{align*}
Using \cref{lem:gagliardo-nirenberg} and \cref{lem:leibniz_defect_estimate} we estimate
\begin{align*}
	\norm{\delta}_{2} 
	\lesssim \seminorm{\kappa \ast v^2}_{C^{\alpha}} \seminorm{v}_{H^{\nicefrac12-\nicefrac\alpha2}}
	\leq \seminorm{\kappa}_{C^{\alpha}} \norm{v}_{2}^2 \seminorm{v}_{H^{\nicefrac12-\nicefrac\alpha2}}
	\lesssim \seminorm{\kappa}_{C^{\alpha}} \norm{v}_{2}^{2 + \alpha} \norm{v}_{H^{\nicefrac12}}^{1-\alpha}.
\end{align*}
We further have
\begin{align*}
	\halfDT (\kappa \ast v^2) 
	= (\fracDT{\nicefrac\alpha2} \kappa) \ast (\fracDT{\nicefrac12-\nicefrac\alpha2} v^2)
	= \fracDT{\nicefrac\alpha2} \kappa \ast (2 v \fracDT{\nicefrac12-\nicefrac\alpha2} v + \tilde \delta)
\end{align*}
with Leibniz-defect $\tilde \delta$ given by
\begin{align*}
	\tilde \delta = \fracDT{\nicefrac12-\nicefrac\alpha2}(v^2) - 2 v \fracDT{\nicefrac12-\nicefrac\alpha2} v.
\end{align*}
By applying \cref{lem:leibniz_defect_estimate}, \cref{lem:sobolev_embedding}, and \cref{lem:gagliardo-nirenberg} for $p$ close to $1$ we obtain the estimate
\begin{align*}
	\Norm{\tilde \delta}_{p} 
	\lesssim \seminorm{v}_{W^{2p,\nicefrac14 - \nicefrac\alpha6}}^2
	\lesssim \norm{v}_{H^{\nicefrac14 - \nicefrac\alpha8}}^2
	\lesssim \norm{v}_{2}^{1 + \nicefrac\alpha2} \norm{v}_{H^{\nicefrac12}}^{1 - \nicefrac\alpha2},
\end{align*}
so that
\begin{align*}
	\Norm{\halfDT (\kappa \ast v^2)}_\infty
	&\leq 2 \Norm{\fracDT{\nicefrac\alpha2} \kappa}_\infty \Norm{v}_2 \Norm{\fracDT{\nicefrac12 - \nicefrac\alpha2} v}_2
	+ \Norm{\fracDT{\nicefrac\alpha2} \kappa}_{p'} \Norm{\tilde \delta}_p
	\\ &\lesssim \norm{\kappa}_{C^\alpha} \left( \norm{v}_{2}^{1 + \alpha} \norm{v}_{H^{\nicefrac12}}^{1 - \alpha} + \norm{v}_{2}^{1 + \nicefrac\alpha2} \norm{v}_{H^{\nicefrac12}}^{1 - \nicefrac\alpha2} \right)
\end{align*}
where we have used \cite[Theorem~2.6]{roncal_stinga} for the estimates on $\fracDT{\nicefrac\alpha2} \kappa$. 

Next we estimate the quantity appearing in the claim: 
\begin{align*}
	\int_\T (\kappa \ast v^2) v \cdot \DT v \der t
	&= \int_\T \halfDT ((\kappa \ast v^2) v) \cdot \halfDT v \der t
	\\ &= \int_\T \left( (\kappa \ast v^2) \halfDT v + v \halfDT (\kappa \ast v^2) + \delta \right) \cdot \halfDT v \der t
	\\ &\geq \int_\T (\kappa \ast v^2) (\halfDT v)^2 \der t
	\\ &\quad - C \norm{v}_2 \norm{\kappa}_{C^\alpha} \left( \norm{v}_{2}^{1 + \alpha} \norm{v}_{H^{\nicefrac12}}^{1 - \alpha} + \norm{v}_{2}^{1 + \nicefrac\alpha2} \norm{v}_{H^{\nicefrac12}}^{1 - \nicefrac\alpha2} \right) \norm{\halfDT v}_2
	\\ &\quad - C \seminorm{\kappa}_{C^\alpha} \norm{v}_{2}^{2 + \alpha} \norm{v}_{H^{\nicefrac12}}^{1-\alpha} \norm{\halfDT v}_2.
\end{align*}
The claim now follows using
\begin{align*}
	\int_\T (\kappa \ast v^2) (\halfDT v)^2 \der t \geq \min \kappa \cdot \Norm{v}_2^2 \Norm{\halfDT v}_2^2	
\end{align*}
and Young's inequality for products.
\end{proof}

Next we prove two important trace inequalities that are adapted to the terms appearing in our functional. In \cref{lem:trace} we estimate the trace in $H^{\nicefrac12}(\T)$ against $\normN{}$, and the ``regularized'' embedding \cref{lem:regularized_trace} estimates the trace in $H^1(\T)$ against $\qN{}$.

\begin{lemma}\label{lem:trace}
	The trace map
	\begin{align*}%
		\tr \colon \espace \to H^{\nicefrac12}(\T), u \mapsto u(R, \impvar)
	\end{align*}
	is well-defined and compact. Furthermore, for all $\eps > 0$ there exists $C(\eps) > 0$ such that 
	\begin{align}\label{eq:loc:trace:bound}
		\norm{\tr u}_{H^{\nicefrac12}(\T)}^2 
		\leq \eps \norm{u_r}_{\Lrad{[0, R] \times \T}}^2 + C(\eps) \normN{u_t}^2.
	\end{align}
	holds for all $u \in \espace$.
\end{lemma}
\begin{remark}
	By \cref{rem:normN} we have the continuous embedding $\iota \colon \espace \embeds \Hrad{[0, R] \times \T}$, and it is well known that the trace maps $H^1([0, R] \times \T)$ into $H^{\nicefrac12}(\T)$, and the same holds for $\Hrad{[0, R] \times \T}$. 
	However, both the embedding $\iota$ and the trace map $\tr \colon \Hrad[]{[0, R] \times \T} \to H^{\nicefrac12}(\T)$ are noncompact maps. Their composition $\tr \circ \iota$ however is compact, as we show below. This is true because of the temporal decay in the embedding $\iota$.
\end{remark}
\begin{proof}[Proof of \cref{lem:trace}]
	Since $\normNav{} \lesssim \normNins{}$ by \cref{rem:normN} and thus $\espaceins \embeds \espaceret$, if suffices to consider the case $N = \Nav$.
	
	Let $u \in H^2_\mathrm{anti}([0, R] \times \T)$. Fix some $\psi \in C^\infty([0, R])$ with $\psi = 0$ on $[0, \tfrac12 R]$ and $\psi(R) = 1$. With $v(r, t) \coloneqq \psi(r) u(r, t)$ we calculate 
	\begin{align*}
		\norm{\tr u}_{H^{\nicefrac12}(\T)}^3
		&= \norm{\tr v}_{H^{\nicefrac12}(\T)}^3
		\leq C_0 \|\halfDT \tr v\|_{L^2(\T)}^3
		\\ &= 3 C_0 \int_0^R \left( \|\halfDT v\|_{L^2(\T)} \int_\T \halfDT v \cdot \halfDT v_r \der t\right) \der r
		\\ &= 3 C_0\int_0^R \left(\|\halfDT v\|_{L^2(\T)} \int_\T \DT v \cdot v_r \der t\right) \der r
		\\ &\leq 3 C_0 \left( \int_{[0, R] \times \T} v_r^2 \der[r] (r, t) \cdot \int_0^R \|\halfDT v\|_{L^2(\T)}^2 \|\DT v\|_{L^2(\T)}^2 \der[r] r \right)^{\nicefrac12}
    \end{align*}
    where the factor $r$ can be introduced since $v$ is supported on $[\frac12 R,R]\times \T$. Using \cref{lem:derivative_estimates}, \cref{rem:normN} and $\inf_\T\kappa>0$, $\inf_{[0,R]} \Gamma>0$ we continue the estimate
	\begin{align*}	
		\norm{\tr u}_{H^{\nicefrac12}(\T)}^3
        &\leq C_1 \norm{v_r}_{\Lrad{[0, R] \times \T}} \normNav{v_t}^2
		\\ &\leq C_2 \left( \norm{u_t}_{\Lrad{[0, R] \times \T}} + \norm{u_r}_{\Lrad{[0, R] \times \T}} \right) \normNav{u_t}^2
		\\ &\leq C_3 \left( \normNav{u_t} + \norm{u_r}_{\Lrad{[0, R] \times \T}} \right) \normNav{u_t}^2
		\\ &\leq C_3 \norm{u}_{\espace}^3.
	\end{align*}
	By approximation, the inequality
	\begin{align*}
		\norm{\tr u}_{H^{\nicefrac12}(\T)}^3 \leq C_3 \left( \normNav{u_t} + \norm{u_r}_{\Lrad{[0, R] \times \T}} \right) \normNav{u_t}^2
	\end{align*}
	can be shown to hold for all $u \in \espace$, so that $\tr$ is a well-defined and bounded operator on $\espace$. The inequality \eqref{eq:loc:trace:bound} now follows immediately using Young's inequality for products. It remains to show compactness of the trace operator. To do this, we consider the operator 
	\begin{align*}
		\tr^K \coloneqq \tr \circ S^K
	\end{align*}
	for $K \in \Nodd$, cf. \cref{lem:projection} for a definition of the projection operators $S^K$. Then $\tr^K$ is a compact operator since it is bounded and has finite-dimensional range. Since $\F_k[u - S^K u] = 0$ for $\abs{k} < K + 2$, using the improved estimate from \cref{lem:derivative_estimates} in our calculation above, we find
	\begin{align*}
		\norm{\tr (u - S^K u)}_{H^{\nicefrac12}(\T)}^3
		\leq \frac{C_3}{\sqrt{K+2}} \norm{u - S^K u}_{\espace}^3,
	\end{align*}
	so that in particular
	\begin{align*}
		\norm{\tr u - \tr^K u}_{H^{\nicefrac12}(\T)}^3
		= \norm{\tr (u - S^K u)}_{H^{\nicefrac12}(\T)}^3
		\leq \frac{C_2 \left(1 + \norm{S^K}\right)^3}{\sqrt{K+2}} \norm{u}_{\espace}^3
	\end{align*}
	holds. Using \cref{lem:projection} it follows that $\tr^K \to \tr$ in $\calB(\espace; H^{\nicefrac12}(\T))$, which shows that $\tr$ is compact.
\end{proof}

Next we show in \cref{lem:regularized_trace} the ``regularized'' trace inequality, which is the main tool used to obtain improved regularity in \cref{sec:approximation}.

\begin{lemma}\label{lem:regularized_trace}
	For all $\eps > 0$ there exists a constant $C(\eps) > 0$ such that 
	\begin{align}\label{eq:loc:regularized_trace}
		\norm{\tr u}_{H^1(\T)}^2
		\leq \eps \norm{\halfDT u_r}_{\Lrad{[0, R] \times \T}}^2
		+ C(\eps) \qN{u_t}^2
	\end{align}
	holds for all $u \in \espaceF$ and $K \in \Nodd$ and where $C(\eps)$ does not depend on $K$.
\end{lemma}
\begin{proof}
	\textit{Part 1:} Let $N = \Nins$. Fix $\psi \in C^\infty([0, R])$ with $\psi = 0$ on $[0, \tfrac12 R]$, $\psi(R) = 1$ and set $v(r, t) \coloneqq \psi(r) u(r, t)$. Further let $H$ denote the Hilbert transform in time, which is given by $\F_k H = \ii \sign(k) \F_k$. Using $\partial_t = H \DT$ we calculate
	\begin{align*}
		\norm{\tr u}_{H^1(\T)}^3 
		&= \norm{\tr v}_{H^1(\T)}^3 
		\\ &\lesssim \norm{\tr v_t}_{L^2(\T)}^3
		\\ &\lesssim \int_\T \abs{v_t(R, \impvar)}^3 \der t
		\\ &= 3 \int_{[0, R] \times \T} v_{tr} v_t \abs{v_t} \der (r, t)
		\\ &= 3 \int_{[0, R] \times \T} H \halfDT v_{r} \cdot \halfDT (v_t \abs{v_t}) \der (r, t)
		\\ &\leq 3 \left(\int_{[0, R] \times \T} \left(H \halfDT v_{r}\right)^2 \der (r, t) \cdot \int_{[0, R] \times \T} \left(\halfDT (v_t \abs{v_t})\right)^2 \der (r, t)\right)^{\nicefrac12}
		\\ &= 3 \left(\int_{[0, R] \times \T} \left(\psi \halfDT u_r + \psi' \halfDT u\right)^2 \der (r, t) \cdot \int_{[0, R] \times \T} \psi^4 \left(\halfDT (u_t \abs{u_t})\right)^2 \der (r, t) \right)^{\nicefrac12}
		\\ & \lesssim \left(\norm{\halfDT u_r}_{\Lrad{}} + \norm{\halfDT u}_{\Lrad{}}\right) \qNins{u_t}^2
	\end{align*}
	where in the last inequality we have estimated $\psi^2, {\psi'}^2, \psi^4 \lesssim r$. From \cref{lem:derivative_estimates} we further have
	\begin{align*}
		\norm{\halfDT u}_{\Lrad{}}
		\lesssim \norm{u_t}_{\Lrad{}}
		\lesssim \norm{u_t}_{\Lrad[4]{}}
		= \norm{u_t \abs{u_t}}_{\Lrad[2]{}}^{\nicefrac12}
		\lesssim \norm{\halfDT (u_t \abs{u_t})}_{\Lrad[2]{}}^{\nicefrac12}
		= \qNins{u_t}.
	\end{align*}
	Combining both inequalities with Young's inequality for products, the estimate \eqref{eq:loc:regularized_trace} follows.

	\textit{Part 2:} Here we consider $N = \Nav$. We define $v$ as above, but now we estimate
	\begin{align*}
		\norm{\tr u}_{H^1(\T)}^3
		&\lesssim \norm{\tr u_t}_{L^2(\T)}^3
		\\ &= \norm{\tr v_t}_{L^2(\T)}^3
		\\ &= 3 \int_0^R \left( \norm{v_t}_{L^2(\T)} \int_\T v_t v_{t r} \der t \right) \der r
		\\ &= 3 \int_0^R \left( \norm{v_t}_{L^2(\T)} \int_\T \halfDT v_t \cdot H \halfDT v_{r} \der t \right) \der r
		\\ &\leq 3 \left( \int_{[0, R] \times \T} \left(H \halfDT v_r\right)^2 \der (r, t) \cdot \int_0^R  \norm{v_t}_{L^2(\T)}^2 \norm{\halfDT v_t}_{L^2(\T)}^2 \der r \right)^{\nicefrac12}
		\\ &\lesssim \left( \norm{\halfDT u_r}_{\Lrad{}} + \norm{\halfDT u}_{\Lrad{}} \right) \qNav{u_t}^2
	\end{align*}
	where again $\supp \psi \subseteq [\frac{1}{2}R,R]$ has been used.
	Using \cref{lem:derivative_estimates} we further obtain
	\begin{align*}
		\norm{\halfDT u}_{\Lrad{}} 
		&\lesssim \norm{u_t}_{\Lrad{}} 
		\lesssim \norm{u_t}_{\Lrad[4]{[0, R]; L^2(\T)}} 
		= \left(\int_0^R \norm{u_t(r,\cdot)}_{L^2(\T)}^4 \der[r] r\right)^{\nicefrac14}
		\\ & \leq \left(\int_0^R \norm{u_t(r,\cdot)}_{L^2(\T)}^2 \norm{\halfDT u_t(r,\cdot)}_{L^2(\T)}^2 \der[r] r\right)^{\nicefrac14}
		= \qNav{u_t}.
	\end{align*}
	Combining both estimates above with Young's inequality for products, the estimate \eqref{eq:loc:regularized_trace} follows.
\end{proof}

\section{Examples}\label{sec:calc_examples}

In this section we prove \cref{thm:main_example} by verifying the assumptions of \cref{thm:radial:main} or \cref{thm:slab:main}. We prepare the proof with a lemma on convergence of infinite matrix products.

\begin{lemma}\label{lem:convergence_of_product}
	Let 
	\begin{align*}
		A_n = \begin{pmatrix}
			1 + \alpha_n & \beta_n \\
			\gamma_n & \lambda + \delta_n
		\end{pmatrix} \in \R^{2 \times 2}
	\end{align*}
	where $\abs{\lambda} < 1$, and $\abs{\alpha_n} \leq \frac{C}{n^2}$ as well as $\abs{\beta_n}, \abs{\gamma_n}, \abs{\delta_n} \leq \frac{C}{n}$ hold for all $n \in \N$. Then the product
	\begin{align*}
		\prod_{n=1}^{\infty} A_n \coloneqq \lim_{m \to \infty} \left(A_1 \cdot A_2 \cdot \ldots \cdot A_m\right)
	\end{align*}
	converges against a matrix of the form $\begin{pmatrix}
		* & 0 \\
		* & 0
	\end{pmatrix}$. If all $A_n$ are invertible, then $\prod_{n=1}^{\infty} A_n \neq 0$. 
	Further, there exists a function $f \colon (0, \infty) \to (0, \infty)$ with $f(0+) = 0$ such that
	\begin{align*}
		\norm{
			\prod_{n=1}^{\infty} A_n 
			- \begin{pmatrix}
				1 & 0 \\ 0 & 0
			\end{pmatrix}
		} \leq f(C).
	\end{align*}
\end{lemma}
\begin{proof}
	First we consider the product
	\begin{align*}
		\begin{pmatrix}
			a_m & b_m \\
			c_m & d_m
		\end{pmatrix} \coloneqq \prod_{n = N}^{m - 1} A_n
	\end{align*}
	where we choose $N \in \N$ so large that denominators appearing in the following four constants $C_a, C_b, C_c, C_d$ with 
	\begin{align*}
		C_a &\coloneqq 2 C C_b,
		\\ C_b &\coloneqq \max\set{
			\frac{1}{\tfrac{2 N}{N + 1} - 1 - \tfrac{2 C}{N}},
			\frac{C}{\tfrac{N}{N+1} - \abs{\lambda} - \tfrac{C + 2 C^2}{N}}
		}, 
		\\ C_c &\coloneqq 2 C C_d, 
		\\ C_d &\coloneqq \max\set{
			\frac{\tfrac{C}{N (1 - \abs{\lambda})}}{\tfrac{2 N}{N + 1} - 1 - \tfrac{2 C}{N}},
			\frac{C + \tfrac{C^2}{N (1 - \abs{\lambda})}}{\tfrac{N}{N + 1} - \abs{\lambda} - \frac{C + 2 C^2}{N}}
		},
	\end{align*}
	are positive and $\frac{C_a}{N} < 1$ holds. We show by induction that the following estimates hold for $m \geq N$: 
	\begin{align}\label{eq:loc:coefficient_estimates}
		\begin{aligned}
			\abs{a_m - 1} &\leq C_a (\tfrac1N - \tfrac1m),
			& \abs{b_m} &\leq \tfrac{C_b}{m},
			\\ \abs{c_m} &\leq C_c \left( \tfrac{1}{N} - \tfrac{1}{m}\right) + \tfrac{C}{N} \tfrac{1 - \abs{\lambda}^{m - N}}{1 - \abs{\lambda}}, \qquad
			& \abs{d_m - \lambda^{m - N}} &\leq \tfrac{C_d}{m}.
		\end{aligned}
	\end{align}
	We will moreover show for the differences that
	\begin{align}\label{eq:loc:coefficient_estimates:diff}
		\abs{a_{m+1} - a_{m}} &\leq \tfrac{C_a}{m(m+1)}
		& \abs{c_{m+1} - c_{m}} &\leq \tfrac{C_c}{m(m+1)} + \tfrac{C}{N} \abs{\lambda}^{m - N}
	\end{align}
	holds.
    First, for $m = N$ the estimates \eqref{eq:loc:coefficient_estimates} hold since $a_N = 1, b_N = 0, c_N = 0, d_N = 1$. For the induction step, let us assume that \eqref{eq:loc:coefficient_estimates} holds for fixed $m \geq N$. Using $a_{m + 1} = (1 + \alpha_{m}) a_{m} + \gamma_{m} b_{m}$ as well as $\abs{a_{m}} \leq C_a^+ \coloneqq 1 + \frac{C_a}{N}$ we find
	\begin{align*} %
		&\abs{a_{m+1} - a_{m}} 
		\leq \frac{C}{m^2} C_a^+ + \frac{C}{m}\frac{C_b}{m} 
		\leq \frac{(N+1) C}{N} \left(C_a^+ + C_b\right) \frac{1}{m (m+1)} \leq \frac{C_a}{m(m+1)}.
	\end{align*}
	This in turn implies that $\abs{a_{m+1} - 1} \leq \abs{a_{m+1} - a_{m}} + \abs{a_{m} - 1} \leq C_a \left(\tfrac{1}{N} - \tfrac{1}{m + 1}\right)$. Next, from $b_{m+1} = \beta_{m} a_{m} + (\lambda + \delta_{m}) b_{m}$ we obtain
	\begin{align*}
		\abs{b_{m+1}} 
		\leq \frac{C}{m} C_a^+ + \abs{\lambda} \frac{C_b}{m} + \frac{C}{m} \frac{C_b}{m}
		\leq \frac{N+1}{N} \left( C C_a^+ + \abs{\lambda} C_b + \frac{C C_b}{N} \right) \frac{1}{m + 1} \leq \frac{C_b}{m + 1}.
	\end{align*}
	Then we use $c_{m+1} = (1 + \alpha_{m}) c_{m} + \gamma_{m} d_{m}$ as well as $\abs{c_m} \leq C_c^+ \coloneqq \frac{C_c}{N} + \frac{C}{N (1 - \abs{\lambda})}$ to obtain
	\begin{align*} %
		\begin{aligned}
			\abs{c_{m+1} - c_{m}} 
			&\leq \frac{C}{m^2} C_c^+ + \frac{C}{m} \left( \abs{\lambda}^{m - N} + \tfrac{C_d}{m} \right)
			\\ &\leq \frac{(N+1) C}{N} \left( C_c^+ + C_d\right) \frac{1}{m (m + 1)} + \frac{C}{N} \abs{\lambda}^{m - N},
			\\ &\leq \frac{C_c}{m(m+1)} + \frac{C}{N} \abs{\lambda}^{m - N},
		\end{aligned}
	\end{align*}
	from which the desired estimate on $\abs{c_{m+1}}$ follows as before. From $d_{m+1} = \beta_m c_{m} + (\lambda + \delta_m) d_{m}$ we obtain
	\begin{align*}
		\abs{d_{m+1} - \lambda^{m + 1 - N}} 
		&\leq \frac{C}{m} C_c^+ +  \abs{\lambda} \frac{C_d}{m} + \frac{C}{m} \left(\abs{\lambda}^{m - N} + \frac{C_d}{m}\right)
		\\ &\leq \frac{N+1}{N} \left( C C_c^+ + \abs{\lambda} C_d + C + \frac{C C_d}{N}\right) \frac{1}{m + 1}
		\leq \frac{C_d}{m+1}.
	\end{align*}
	This shows the estimates \eqref{eq:loc:coefficient_estimates}, \eqref{eq:loc:coefficient_estimates:diff}. It follows that $b_m, d_m \to 0$ and that $a_m$, $c_m$ converge as $m \to \infty$. Thus we have shown that the product $\prod_{n=N}^{\infty} A_n$ converges against a matrix of the form $\begin{pmatrix}
		* & 0 \\
		* & 0
	\end{pmatrix}$. This implies convergence of the product $\prod_{n=1}^{\infty} A_n$ with the limit being given by 
	\begin{align*}
		\prod_{n=1}^{\infty} A_n = \left[A_1 \cdot \ldots \cdot A_{N-1}\right] \cdot \begin{pmatrix}
			\lim_{m \to \infty} a_m & 0 \\
			\lim_{m \to \infty} c_m & 0
		\end{pmatrix},
	\end{align*}
	which has the specified form of vanishing second column. From \eqref{eq:loc:coefficient_estimates} we get $|a_m-1|< \tfrac{C_a}{N} < 1$ so that $\lim_{m \to \infty} a_m \neq 0$. Thus $\prod_{n=1}^{\infty} A_n \neq 0$ if we assume that $A_1, \dots, A_{N-1}$ are invertible.
	
	\medskip
	
	It remains to show the estimate
	\begin{align*}
		\norm{
			\prod_{n=1}^{\infty} A_n 
			- \begin{pmatrix}
				1 & 0 \\
				0 & 0
			\end{pmatrix}
		} \leq f(C).
	\end{align*}
	We choose $\norm{}$ to be the column sum norm.
	To emphasize the dependence of $C_a, \dots, C_d$ on the constant $C$, in the following we write $C_a(C), \dots, C_d(C)$. Let $C^\star > 0$ and observe that there exists a sufficiently large $N \in \N$ such that the denominators appearing in $C_a(C), \dots, C_d(C)$ are positive for all $C \in (0, C^\star]$. Then \eqref{eq:loc:coefficient_estimates} shows that
	\begin{align*}
		\norm{
			\prod_{n=N}^{\infty} A_n 
			- \begin{pmatrix}
				1 & 0 \\
				0 & 0
			\end{pmatrix}
		} \leq 
			\frac{C_a(C) + C_c(C)}{N} + \frac{C}{N(1 - \abs{\lambda})}
		\eqqcolon f_N(C)
	\end{align*}
	holds, where $f_N \colon (0, C^\star] \to (0, \infty)$ with $f_N(0+) = 0$. Then
	\begin{align*}
		\norm{
			\prod_{n=N-1}^{\infty} A_n 
			- \begin{pmatrix}
				1 & 0 \\
				0 & 0
			\end{pmatrix}
		}
		&\quad\leq \norm{A_{N-1}} \norm{
			\prod_{n=N}^{\infty} A_n 
			- \begin{pmatrix}
				1 & 0 \\
				0 & 0
			\end{pmatrix}
		} + \norm{A_{N-1} \begin{pmatrix}
			1 & 0 \\
			0 & 0
		\end{pmatrix} - \begin{pmatrix}
			1 & 0 \\
			0 & 0
		\end{pmatrix}}
		\\ &\quad\leq \max\set{1 + \tfrac{N C}{(N-1)^2}, \tfrac{2 C}{N-1} + \abs{\lambda}} f_N(C) + \tfrac{N C}{(N-1)^2}
		\eqqcolon f_{N-1}(C)
	\end{align*}
	where $f_{N-1}(0+) = 0$. Repeating this $N-2$ times, we find a function $f_1$ such that $f_1(0+) = 0$ and
	\begin{align*}
		\norm{
			\prod_{n=1}^{\infty} A_n 
			- \begin{pmatrix}
				1 & 0 \\
				0 & 0
			\end{pmatrix}
		} \leq f_1(C). 
		&\qedhere
	\end{align*}
\end{proof}

The proof of \cref{thm:main_example} is split into four parts, because the fundamental solutions $\phi_k$ strongly depend on both the chosen geometry (radial problem or slab problem) and on the linear potential $\tilde \chi_1^*$ (step potential $\tilde \chi_1^\mathrm{step}$ or periodic step potential $\tilde\chi_1^\mathrm{per}$).

For the proof, we introduce the following (non-negative) variables to denote the values of the piecewise constant potential $V = -(\tilde \chi_1 + 1 - c^{-2})$:
\begin{enumerate}[(i)]
	\item If $\tilde \chi_1^\ast = \tilde \chi_1^{\mathrm{per}}$, we set
	\begin{align*}
		\alpha \coloneqq a + 1 - c^{-2}, 
		\quad
		\beta \coloneqq b + 1 - c^{-2}, 
		\quad
		\delta \coloneqq -(d + 1 - c^{-2}), 
	\end{align*}
	where by the assumptions of \cref{thm:main_example} we have $\alpha, \beta, \delta > 0$ and $\delta < \alpha$.

	\item If $\tilde \chi_1^\ast = \tilde \chi_1^{\mathrm{step}}$, let
	\begin{align*}
		\alpha \coloneqq a + 1 - c^{-2}, 
		\quad
		\beta \coloneqq -(b + 1 - c^{-2}), 
		\quad
		\delta \coloneqq -(d + 1 - c^{-2}), 
	\end{align*}
	where again $\alpha, \beta, \delta > 0$ by assumption.
\end{enumerate}

\begin{proof}[Proof of \cref{thm:main_example}, Part 1]
	First, we consider the periodic step potential $\tilde\chi_1^\mathrm{per}$ with cylindrical geometry, i.e. \eqref{eq:radial:problem}. We verify the assumptions \ref{ass:radial:first}--\ref{ass:radial:fundamental_solution:estimates} and \ref{ass:radial:nontrivial_sol:alt} in order to apply \cref{thm:radial:multiplicity}. Firstly, assumptions \ref{ass:radial:VandGamma}, \ref{ass:radial:N}, and \ref{ass:radial:elliptic} hold by definition.
	
	\textit{Step 1.} Here we construct the fundamental solutions $\phi_k$ based on the following idea: we define propagation matrices $M_{L_k}(r,r')$ with the property: $\binom{\phi_k(r)}{\phi_k'(r)} \coloneqq M_{L_k}(r,r') \binom{a}{b}$ provides the solution of $L_k \phi_k=0$ at $r$ with initial values $\binom{a}{b}$ at $r'$. On subintervals where the potential $V$ takes constant values, $M_{L_k}$ can be explicitly computed. 
	Iterating the propagation from $r_n \coloneqq R+nP+\tfrac{1}{2}\theta P$ back to $r$ with prescribed decay $r_n^{-\nicefrac12} \tau^n$ at $r_n$, $\tau=\min\{\sqrt{\tfrac{\alpha}{\beta}}, \sqrt{\tfrac{\beta}{\alpha}}\}<1$ and sending $n\to \infty$ will provide the fundamental solution.
	
	Now we start with the propagation matrices on intervals where $V$ is constant. The general solution of 
	\begin{align}\label{eq:loc:radial_constant_linear}
		\left(-\partial_r^2 - \tfrac{1}{r} \partial_r + \tfrac{1}{r^2} - k^2 \omega^2 \alpha \right) f = 0
	\end{align}
	is given by 
	\begin{align*}
		f(r) = A J_1(\alpha_k r) + B Y_1(\alpha_k r)
	\end{align*}
	where $J_\nu, Y_\nu$ are the Bessel functions of first (resp. second) kind and $\alpha_k \coloneqq k \omega \sqrt{\alpha}$. Thus the propagation matrix for \eqref{eq:loc:radial_constant_linear} is given by
	\begin{align*}
		\begin{pmatrix}
			f(r) \\ f'(r) 
		\end{pmatrix}
		= M_{\alpha, k}(r, r') \cdot \begin{pmatrix}
			f(r') \\ f'(r') 
		\end{pmatrix}
	\end{align*}
	with
	\begin{align*}
		&M_{\alpha, k}(r, r') 
		= \begin{pmatrix}
			J_1(\alpha_k r) & Y_1(\alpha_k r) \\
			\alpha_k J_1'(\alpha_k r) & \alpha_k Y_1'(\alpha_k r)
		\end{pmatrix} \cdot \begin{pmatrix}
			J_1(\alpha_k r') & Y_1(\alpha_k r') \\
			\alpha_k J_1'(\alpha_k r') & \alpha_k Y_1'(\alpha_k r')
		\end{pmatrix}^{-1}.
	\end{align*}
	Next we calculate the asymptotic expansion of $M_{\alpha, k}(r, r')$. With the asymptotics (cf. \cite{gradshteyn})
	\begin{align*}
		&J_1(z) = \sqrt{\frac{2}{\pi z}} \left( \sin(z - \tfrac{\pi}{4}) + \tfrac{3}{8 z} \cos(z - \tfrac{\pi}{4}) + \landauO\left(\tfrac{1}{z^2}\right) \right), 
		\\ &Y_1(z) = \sqrt{\frac{2}{\pi z}} \left( - \cos(z - \tfrac{\pi}{4}) + \tfrac{3}{8 z} \sin(z - \tfrac{\pi}{4}) + \landauO\left(\tfrac{1}{z^2}\right) \right), 
		\\ &J_1'(z) = \sqrt{\frac{2}{\pi z}} \left( \cos(z - \tfrac{\pi}{4}) - \tfrac{7}{8 z} \sin(z - \tfrac{\pi}{4}) + \landauO\left(\tfrac{1}{z^2}\right) \right), 
		\\ &Y_1'(z) = \sqrt{\frac{2}{\pi z}} \left( \sin(z - \tfrac{\pi}{4}) + \tfrac{7}{8 z} \cos(z - \tfrac{\pi}{4}) + \landauO\left(\tfrac{1}{z^2}\right) \right)
	\end{align*}
	as $z \to \infty$, we find
	\begin{align}\label{eq:loc:constant_propatagion_asymptotics}
		\begin{aligned}
			M_{\alpha, k}(r, r') 
			&= \sqrt{\frac{r'}{r}} 
			\begin{pmatrix}
				1 & 0 \\
				0 & \alpha_k
			\end{pmatrix} 
			\Bigl[ \begin{pmatrix}
				\frac{3}{8 z} & -1 \\
				1  & \frac{7}{8 z}
			\end{pmatrix} + \landauO\left(\tfrac{1}{z^2}\right) \Bigr]
			\begin{pmatrix}
				\cos(z - z') & \sin(z - z') \\
				- \sin(z - z') & \cos(z - z') \\
			\end{pmatrix} 
			\\ &\qquad \cdot \Bigl[ \begin{pmatrix}
				\frac{7}{8 z'} & 1 \\
				-1 & \frac{3}{8 z'}
			\end{pmatrix} + \landauO\left(\tfrac{1}{{z'}^2}\right) \Bigr]
			\begin{pmatrix}
				1 & 0 \\
				0 & \frac{1}{\alpha_k}
			\end{pmatrix}
		\end{aligned}
	\end{align}
	as $z, z' \to \infty$, where $z = \alpha_k r, z' = \alpha_k r'$. If, in particular, $r' - r = \theta P$, then since
	\begin{align*}
		z - z' = - k \omega \theta P \sqrt{\alpha} \in \tfrac{\pi}{2} \Zodd, 
	\end{align*}
	we have $\cos(z - z') = 0$ and $\sin(z - z') = \pm 1$, so we can further simplify
	\begin{align*}
		M_{\alpha, k}(r, r') 
		&= \pm \sqrt{\frac{r'}{r}} 
		\begin{pmatrix}
			1 & 0 \\
			0 & \alpha_k
		\end{pmatrix}
		\Bigl[ \begin{pmatrix}
			\frac{7}{8 z'} - \frac{3}{8 z} & 1 \\
			-1 & \frac{3}{8 z'} - \frac{7}{8 z}
		\end{pmatrix} + \landauO\left(\tfrac{1}{z^2} + \tfrac{1}{{z'}^2}\right) \Bigr]
		\begin{pmatrix}
		   1 & 0 \\
		   0 & \frac{1}{\alpha_k}
	   	\end{pmatrix}
		\\ &= \pm \sqrt{\frac{r'}{r}} 
		\begin{pmatrix}
			1 & 0 \\
			0 & k \omega
		\end{pmatrix}
		\Bigl[ \begin{pmatrix}
			\frac{1}{2 k \omega \sqrt{\alpha} r} & \frac{1}{\sqrt{\alpha}} \\
			-\sqrt{\alpha} & - \frac{1}{2 k \omega \sqrt{\alpha} r}
		\end{pmatrix} + \landauO\left(\tfrac{1}{k r^2}\right) \Bigr]
		\begin{pmatrix}
		   1 & 0 \\
		   0 & \frac{1}{k \omega}
	   	\end{pmatrix}
	\end{align*}
	as $k r \to \infty$. If we denote by $r_n \coloneqq R + n P + \tfrac12 \theta P$, $r_n' \coloneqq R + n P + \left(1 - \tfrac12 \theta\right) P$ for $n \in \N_0$ the points where $V$ changes from one constant value to the other, then we have
	\begin{align*}
		M_{L_k}(r_n, r_{n+1}) &= M_{\beta, k}(r_n, r_n') \cdot M_{\alpha, k}(r_n', r_{n+1})
		\\ &= \sigma \sqrt{\frac{r_{n+1}}{r_n}} 
		\begin{pmatrix}
			1 & 0 \\
			0 & k \omega
		\end{pmatrix}
		\left[ \begin{pmatrix}
			\sqrt{\frac{\alpha}{\beta}} & 0 \\
			\bigl( \sqrt{\frac{\beta}{\alpha}} - \sqrt{\frac{\alpha}{\beta}} \bigr) \frac{1}{2 k \omega r_n} & \sqrt{\frac{\beta}{\alpha}}
		\end{pmatrix} + \landauO\left(\tfrac{1}{k n^2}\right) \right]
		\begin{pmatrix}
		   1 & 0 \\
		   0 & \frac{1}{k \omega}
	   	\end{pmatrix}
	\end{align*}
	as $k n \to \infty$ where $\sigma = - \sin(k \omega \theta P \sqrt{\alpha}) \sin(k \omega (1 - \theta) P \sqrt{\beta}) \in \set{\pm 1}$ does not depend on $k$. 
	To simplify the asymptotics, we introduce the rescaling
	\begin{align*}
		S_k(r) \coloneqq \sqrt{r} \begin{pmatrix}
			1 & 0 \\
			0 & \tfrac{1}{\omega k}
		\end{pmatrix}
	\end{align*}
	and define the rescaled propagation matrices
	\begin{align*}
		&M_{L_k}^S(r, r') \coloneqq S_k(r) M_{L_k}(r, r') S_k(r')^{-1},
		\\ &M_{\alpha, k}^S(r, r') \coloneqq S_k(r) M_{\alpha, k}(r, r') S_k(r')^{-1},
		\\ &M_{\beta, k}^S(r, r') \coloneqq S_k(r) M_{\beta, k}(r, r') S_k(r')^{-1}.
	\end{align*}
	In the following we assume $\alpha > \beta$, the case $\beta > \alpha$ can be treated similarly. We then define
	\begin{align}\label{eq:loc:fundamental_sol_by_limit}
        \Psi_k(r) \coloneqq \lim_{m \to \infty} \left(\sigma^{-1} \sqrt{\tfrac{\beta}{\alpha}}\right)^{m} M_{L_k}^S(r, r_m).
	\end{align}
	This definition is according to the idea introduced at the beginning of the proof: $\Psi_k(r)$ contains decaying fundamental solutions for a rescaled version of the operator $L_k$, where the geometric decay factor $\sqrt{\tfrac{\beta}{\alpha}}$ has been built into the solution. As we shall see, the second column of $\Psi_k$ vanishes, which reflects the fact that there can only be one solution which decays at infinity. Next we note that 
	\begin{align*}
		\Psi_k(r) = \lim_{m \to \infty} M_{L_k}^S(r, r_0) \prod_{n=0}^{m-1} \left( \sigma^{-1} \sqrt{\tfrac{\beta}{\alpha}} M_{L_k}^S(r_n, r_{n+1}) \right)
		= M_{L_k}^S(r, r_0) \prod_{n=0}^{\infty} \left( \sigma^{-1} \sqrt{\tfrac{\beta}{\alpha}} M_{L_k}^S(r_n, r_{n+1}) \right).
	\end{align*}
	Using the asymptotics
	\begin{align*}
		\sigma^{-1} \sqrt{\frac{\beta}{\alpha}}
		M_{L_k}^S(r_n, r_{n+1}) = \begin{pmatrix}
			1 & 0 \\
			\landauO\left(\tfrac{1}{kn}\right) & \frac{\beta}{\alpha}
		\end{pmatrix} + \landauO\left(\tfrac{1}{k n^2}\right)
	\end{align*}
	as $k n \to \infty$ and \cref{lem:convergence_of_product}, the limit in \eqref{eq:loc:fundamental_sol_by_limit} exists and is nonzero. In particular, the limit matrix $\Psi_k$ has a vanishing second column so that we can define 
	\begin{align*}
		\begin{pmatrix}
			\psi_k^{(1)}(r) & 0 \\
			\psi_k^{(2)}(r) & 0
		\end{pmatrix} \coloneqq \Psi_k(r).
	\end{align*}
	If we also undo the rescaling and define
	\begin{align*}
		\begin{pmatrix}
			\phi_k^{(1)}(r) \\ \phi_k^{(2)}(r) 
		\end{pmatrix} 
		\coloneqq S_k(r)^{-1} 
		\begin{pmatrix}
			\psi_k^{(1)}(r) \\ \psi_k^{(2)}(r) 
		\end{pmatrix}
	\end{align*}
    then
	\begin{align*}
		\begin{pmatrix}
			\phi_k^{(1)}(r) \\ \phi_k^{(2)}(r) 
		\end{pmatrix} 
		= M_{L_k}(r, r') 
		\begin{pmatrix}
			\phi_k^{(1)}(r') \\ \phi_k^{(2)}(r') 
		\end{pmatrix}.
	\end{align*}
	since $\psi_k^{(1)}, \psi_k^{(2)}$ satisfy the identity
	\begin{align*}
		\begin{pmatrix}
			\psi_k^{(1)}(r) \\ \psi_k^{(2)}(r) 
		\end{pmatrix} 
		= M_{L_k}^S(r, r') 
		\begin{pmatrix}
			\psi_k^{(1)}(r') \\ \psi_k^{(2)}(r') 
		\end{pmatrix}.
	\end{align*}
	This shows that $\phi_k(r) \coloneqq \phi_k^{(1)}(r)$ satisfies $L_k \phi_k = 0$ and $\phi_k' = \phi_k^{(2)}$ so that $\phi_k$ is the sought fundamental solution. Its properties will be verified in the next step.
	
	\medskip
	\textit{Step 2.} We are left with verifying assumptions \ref{ass:radial:fundamental_solutions}, \ref{ass:radial:fundamental_solution:estimates}, and \ref{ass:radial:nontrivial_sol:alt} for the functions $\phi_k$ obtained in Step 1. We have already seen as a result of \cref{lem:convergence_of_product} that $\phi_k \neq 0$. Next we show that $\phi_k \in \Lrad{[R, \infty) \times \T}$.
	From the asymptotics \eqref{eq:loc:constant_propatagion_asymptotics} we see that there exists a constant $C_1 > 0$ such that
	\begin{align*}
		\norm{ M_{\alpha, k}^S(r, r')}, \norm{ M_{\beta, k}^S(r, r')} \leq C_1
	\end{align*}
	holds for all $k\in \Nodd$ and all $r, r' \geq R$. 
	By \cref{lem:convergence_of_product} there further exists a constant $C_2 > 0$ such that
	\begin{align*}
		\norm{\prod_{n=n_0}^{\infty} \left( \sigma^{-1} \sqrt{\tfrac{\beta}{\alpha}} M_{L_k}^S(r_n, r_{n+1}) \right)}
		\leq C_2
	\end{align*}
	holds for all $n_0 \in \N_0$ and all $k\in\Nodd$.

	For every $r \in [R, \infty)$ there exists a unique $n_0 \in \N_0$ such that $r \in (r_{n_0-1}', r_{n_0}']$. From
	\begin{align} \label{eq:psi_k_1}
		\begin{aligned}
			\sqrt{r} \phi_k(r)
			&= \psi_k^{(1)}(r)
			= \left( \lim_{m \to \infty} \left(\sigma^{-1} \sqrt{\tfrac{\beta}{\alpha}}\right)^{m} M_{L_k}^S(r, r_{n_0}) M_{L_k}^S(r_{n_0}, r_m)  \right)_{1,1}
			\\ &= \left(\sigma^{-1} \sqrt{\tfrac{\beta}{\alpha}}\right)^{n_0} \left( M_{L_k}^S(r, r_{n_0}) \prod_{n = n_0}^{\infty} \left( \sigma^{-1} \sqrt{\tfrac{\beta}{\alpha}} M_{L_k}^S(r_n, r_{n+1}) \right) \right)_{1,1},
		\end{aligned}
	\end{align}
	and
	\begin{align*}
		M_{L_k}^S(r, r_{n_0}) = \begin{cases}
			M_{\alpha, k}^S(r, r_{n_0}), & r \in (r_{n_0 - 1}', r_{n_0}] \\
			M_{\beta, k}^S(r, r_{n_0}), & r \in (r_{n_0}, r_{n_0}']
		\end{cases} 
	\end{align*}
	we obtain
	\begin{align*}
		r \abs{\phi_k(r)}^2 \leq \left(\tfrac{\beta}{\alpha}\right)^{n_0} C_1^2 C_2^2
	\end{align*}
	and therefore also
	\begin{align*}
		\norm{\phi_k}_{\Lrad{[R, \infty)}}^2
		= \int_{R}^\infty \abs{\phi_k(r)}^2 \der[r] r
		\leq \sum_{n_0 = 0}^{\infty} \left(\tfrac{\beta}{\alpha}\right)^{n_0} C_1^2 C_2^2 P 
		< \infty,
	\end{align*}
	where we used $r_{n_0}' - r_{n_0-1}' = P$. While we obtained an $L^2$-bound on $\phi_k$ which is uniform in $k$, with the help of the equation $L_k \phi_k = 0$ one can easily show that $\phi_k \in \Hrad[2]{[R, \infty)}$, but the $H^2$-bound will be $k$-dependent. Thus assumption \ref{ass:radial:fundamental_solutions} holds.

	Next we discuss the asymptotics of $\phi_k$. We use
	\begin{align*}
		M_{\alpha, k}^S(r, r') = \begin{pmatrix}
			\cos(k \omega \sqrt{\alpha} (r - r')) & \tfrac{1}{\sqrt{\alpha}} \sin(k \omega \sqrt{\alpha} (r - r')) \\
			- \sqrt{\alpha} \sin(k \omega \sqrt{\alpha} (r - r')) & \cos(k \omega \sqrt{\alpha} (r - r'))
		\end{pmatrix} + \landauO \left( \tfrac{1}{k} \right)
	\end{align*}
	and likewise for $M_{\beta, k}^S(r, r')$ as well as
	\begin{align*}
		\prod_{n = n_0}^{\infty} \left( \sigma^{-1} \sqrt{\tfrac{\beta}{\alpha}} M_{L_k}^S(r_n, r_{n+1}) \right)
		\to \begin{pmatrix}
			1 & 0 \\
			0 & 0
		\end{pmatrix}
	\end{align*}
	as $k \to \infty$, cf. \cref{lem:convergence_of_product}. Thus, together with \eqref{eq:psi_k_1} we obtain
	\begin{eqnarray*}
		\lefteqn{\int_{r_{n_0-1}'}^{r_{n_0}'} \abs{\phi_k(r)}^2 \der[r] r} \\
		&=& \left( \tfrac{\beta}{\alpha} \right)^{n_0} 
		\left( \int_{r_{n_0-1}'}^{r_{n_0}} \cos(k \omega \sqrt{\alpha} (r - r_{n_0}))^2 \der r 
		+ \int_{r_{n_0}}^{r_{n_0}'} \cos(k \omega \sqrt{\beta} (r - r_{n_0}))^2 \der r 
		+ \landauo(1) \right)
		\\ &=& \left( \tfrac{\beta}{\alpha} \right)^{n_0} \left(\tfrac{P}{2} + \landauo(1) \right)
	\end{eqnarray*}
	as $k \to \infty$. In particular, we have
	\begin{align*}
		\liminf_{k \to \infty} \norm{\phi_k}_{\Lrad{[R, \infty)}}^2
		\geq \sum_{n_0 = 1}^{\infty} \left( \tfrac{\beta}{\alpha} \right)^{n_0} \tfrac{P}{2} > 0.
	\end{align*}
	In order to verify assumptions \ref{ass:radial:fundamental_solution:estimates} and \ref{ass:radial:nontrivial_sol:alt} we calculate the asymptotics of $\phi_k(R)$ and $\phi_k'(R)$. Setting
	\begin{align*}
		m_\alpha \coloneqq 4 \omega \sqrt{\alpha} \theta \tfrac{P}{2 \pi} \in \Nodd,
	\end{align*}
	we have 
	\begin{align*}
		\sqrt{R} \phi_k(R) 
		&= \psi_k^{(1)}(R)
		= \left( M_{\alpha, k}^S(R, r_0) \prod_{n = 0}^{\infty} \left( \sigma^{-1} \sqrt{\tfrac{\beta}{\alpha}} M_{L_k}^S(r_n, r_{n+1}) \right) \right)_{1, 1} 
		\\ &= \cos(k \omega \sqrt{\alpha} (R - r_0)) + \landauo(1)
		= \cos(k m_\alpha \tfrac{\pi}{4}) + \landauo(1)
		= \pm \tfrac{1}{\sqrt{2}} + \landauo(1)
	\end{align*}
	as $k \to \infty$. Combined with the estimates on $\norm{\phi_k}_{\Lrad{}}$ this shows the first part of assumption \ref{ass:radial:fundamental_solution:estimates}. Next we have
	\begin{align*}
		\frac{\sqrt{R}}{k \omega} \phi_k'(R) 
		&= \psi_k^{(2)}(R) 
		= \left( M_{\alpha, k}^S(R, r_0) \prod_{n = 0}^{\infty} \left( \sigma^{-1} \sqrt{\tfrac{\beta}{\alpha}} M_{L_k}^S(r_n, r_{n+1}) \right) \right)_{2, 1} 
		\\ &= -\sqrt{\alpha} \sin(k \omega \sqrt{\alpha} (R - r_0)) + \landauo(1)
		= \sqrt{\alpha}\sin(k m_\alpha \tfrac{\pi}{4}) + \landauo(1)
		= \pm \tfrac{\sqrt{\alpha}}{\sqrt{2}} + \landauo(1),
	\end{align*}
	which shows the second part of assumption \ref{ass:radial:fundamental_solution:estimates}. Finally, we have
	\begin{align*}
		\frac{\phi_k'(R)}{k \phi_k(R)} 
		= \omega \sqrt{\alpha} \tan(k m_\alpha \tfrac{\pi}{4}) + \landauo(1)
		= -\omega \sqrt{\alpha} (-1)^{\tfrac{m_\alpha + k}{2}} + \landauo(1)
	\end{align*}
	as $k \to \infty$, so
	\begin{align*}
		\limsup_{k \to \infty} \frac{\phi_k'(R)}{k \phi_k(R)} 
		= \omega \sqrt{\alpha} > \omega \sqrt{\delta} = \omega \norm{V}_{L^\infty([0, R])}^{\nicefrac12}
	\end{align*}
	since $\alpha > \delta$. Thus assumption \ref{ass:radial:nontrivial_sol:alt} holds, and \cref{thm:radial:multiplicity} yields the result.
\end{proof}

\begin{proof}[Proof of \cref{thm:main_example}, Part 2]
	Now we consider the periodic step potential with the slab geometry, i.e. \eqref{eq:slab:problem}. We verify the assumptions of \cref{thm:slab:multiplicity} in order to apply it.

    By the set-up we have that \ref{ass:slab:VandGamma}, \ref{ass:slab:N}, and \ref{ass:slab:elliptic} hold. 
    The determination of the fundamental solutions $\tildePhi_k$ follows the Floquet-Bloch theory for second-order periodic differential operators. Details can be found in \cite[Appendix 6.2]{kohler_reichel}. The main outcome is the following: there are two Floquet-multipliers 
    $$\rho_k \in \left\{ -\sqrt{\frac{\alpha}{\beta}} \sin(km' l\pi)\sin(km'\pi), -\sqrt{\frac{\beta}{\alpha}} \sin(km' l\pi)\sin(km'\pi)\right\}
    $$
    where 
    $$
    l = \sqrt{\frac{\beta}{\alpha}}\frac{1-\theta}{\theta} 
	\quad \mbox{ and } \quad
	\left\{2m' = 4\sqrt{\alpha}\theta \omega \frac{P}{2\pi}, 2m'l = 4\sqrt{\beta}(1-\theta) \omega \frac{P}{2\pi}\right\} \subseteq \Nodd
    $$
    and in our setting $m=2m'$, $n=2m'l$. This implies that $|\rho_k|\in \bigl\{\sqrt{\frac{\alpha}{\beta}}, \sqrt{\frac{\beta}{\alpha}}\bigr\}$, so that in modulus one of them is smaller than $1$ and one of them larger than $1$. For each Floquet exponent there is a solution of $L_k \tildePhi_k=0$ with $\tildePhi_k(x+P) = \rho_k \tildePhi_k(x)$ for all $x\in\R$. If we choose the fundamental solution corresponding to the Floquet multiplier with modulus less than $1$ then this leads to 
    $$
    \norm{\tildePhi_k}_{L^2([R, \infty))}^2 = \frac{1}{1-\rho_k}\norm{\tildePhi_k}_{L^2([R,R+P))}^2.
    $$ 
    Using the normalization $\tildePhi_k(R)=1$ we have 
	\begin{align*}
		\tildePhi_k'(R)= -k\omega \sqrt{\alpha}\tan(k\omega \sqrt{\alpha}\theta \tfrac{P}{2})= - k \omega \sqrt{\alpha} \tan(m k \tfrac{\pi}{4}) = k \omega\sqrt{\alpha} (-1)^\frac{k + m}{2}	
	\end{align*}
	and 
	\begin{align*}
		0 < \inf_{k\in\Nodd} \norm{\tildePhi_k}_{L^2([R,R+P))}^2 \leq \sup_{k\in\Nodd} \norm{\tildePhi_k}_{L^2([R,R+P))}^2<\infty.	
	\end{align*}
    From these estimates it follows that also \ref{ass:slab:fundamental_solutions}--\ref{ass:slab:fundamental_solutions:estimates} hold. Moreover, since $\frac{\tildePhi_k'(R)}{k \tildePhi_k(R)} = (-1)^\frac{k + m}{2} \omega\sqrt{\alpha}$ and $\alpha>\delta=\norm{V}_{L^\infty([-R, R])}$ the final condition \ref{ass:slab:nontrivial_sol:alt} is true and so \cref{thm:slab:multiplicity} yields the result.
\end{proof}

\begin{proof}[Proof of \cref{thm:main_example}, Part 3]
	Next we consider the step potential $\tilde\chi_1^\mathrm{step}$ with cylindrical geometry, i.e. problem \eqref{eq:radial:problem}. We verify the assumptions \ref{ass:radial:first}--\ref{ass:radial:secondlast}, \ref{ass:radial:nontrivial_sol:alt} in order to apply \cref{thm:radial:multiplicity}. First, \ref{ass:radial:VandGamma}, \ref{ass:radial:N}, and \ref{ass:radial:elliptic} hold by definition. 
	Let $J_\nu, Y_\nu$ denote the Bessel functions of first (resp. second) kind and $K_\nu$ denote the modified Bessel function of second kind.
	For $k \in \Nodd$ and with $\alpha_k \coloneqq k \omega \sqrt{\alpha}$, $\beta_k \coloneqq k \omega \sqrt{\beta}$, the fundamental solution $\phi_k$ is (up to a constant) then given by
	\begin{align*}
		\phi_k(r) = \begin{cases}
			A_k J_1(\alpha_k r) + B_k Y_1(\alpha_k r), & R < r < R+\rho, \\
			K_1(\beta_k r), & r > R+\rho.
		\end{cases}
	\end{align*}
	with
	\begin{align*}
		\begin{pmatrix}
			A_k \\
			B_k
		\end{pmatrix}
		= \begin{pmatrix}
			J_1(\alpha_k (R+\rho)) & Y_1(\alpha_k (R+\rho)) \\
			\alpha_k J_1'(\alpha_k (R+\rho)) & \alpha_k Y_1'(\alpha_k (R+\rho))
		\end{pmatrix}^{-1} \begin{pmatrix}
			K_1(\beta_k (R+\rho)) \\
			\beta_k K_1'(\beta_k (R+\rho))
		\end{pmatrix}.
	\end{align*}
	We begin by estimating the functions $\phi_k$. Using the asymptotics (cf. \cite{grafakos_classical})
	\begin{align*}
		&J_1(z) = \sqrt{\tfrac{2}{\pi z}} \left( \sin(z - \tfrac{\pi}{4}) + \landauO\left(\tfrac{1}{z}\right) \right),
		&&J_1'(z) = \sqrt{\tfrac{2}{\pi z}} \left( \cos(z - \tfrac{\pi}{4}) + \landauO\left(\tfrac{1}{z}\right) \right),
		\\ &Y_1(z) = \sqrt{\tfrac{2}{\pi z}} \left( - \cos(z - \tfrac{\pi}{4}) + \landauO\left(\tfrac{1}{z}\right) \right),
		&&Y_1'(z) = \sqrt{\tfrac{2}{\pi z}} \left( \sin(z - \tfrac{\pi}{4}) + \landauO\left(\tfrac{1}{z}\right) \right),
		\\ &K_1(z) = \sqrt{\tfrac{\pi}{2 z}} \ee^{-z} \left( 1 + \landauO\left(\tfrac{1}{z}\right) \right),
		&&K_1'(z) = \sqrt{\tfrac{\pi}{2 z}} \ee^{-z} \left( -1 + \landauO\left(\tfrac{1}{z}\right) \right)
	\end{align*}
	as $z \to \infty$, we find
	\begin{align*}
		&A_k = - \tfrac{\pi}{2} \ee^{-\beta_k (R+\rho)} \left( \sqrt{\tfrac{\beta_k}{\alpha_k}} \cos(\alpha_k (R+\rho) - \tfrac{\pi}{4}) - \sqrt{\tfrac{\alpha_k}{\beta_k}} \sin(\alpha_k (R+\rho) - \tfrac{\pi}{4}) + \landauO\left(\tfrac{1}{k}\right) \right),
		\\ &B_k = - \tfrac{\pi}{2} \ee^{-\beta_k (R+\rho)} \left( \sqrt{\tfrac{\beta_k}{\alpha_k}} \sin(\alpha_k (R+\rho) - \tfrac{\pi}{4}) + \sqrt{\tfrac{\alpha_k}{\beta_k}} \cos(\alpha_k (R+\rho) - \tfrac{\pi}{4}) + \landauO\left(\tfrac{1}{k}\right) \right),
	\end{align*}
	and thus
	\begin{align*}
		&\norm{\phi_k}_{\Lrad{[R+\rho, \infty)}}^2 = \tfrac{\pi}{4 \beta_k^2} \ee^{-2 \beta_k (R+\rho)} \left( 1 + \landauO(\tfrac{1}{k}) \right),
		\\ &\norm{\phi_k}_{\Lrad{[R, R+\rho]}}^2 =\tfrac{\pi}{4 \alpha_k} \ee^{-2 \beta_k (R+\rho)} \left(\tfrac{\alpha_k}{\beta_k} + \tfrac{\beta_k}{\alpha_k}\right) \rho \left(1 + \landauO(\tfrac{1}{k}) \right)
	\end{align*}
	as $k \to \infty$. In particular we have
	\begin{align*}
		\norm{\phi_k}_{\Lrad{[R, \infty)}} = \frac{\ee^{- \beta_k (R+\rho)}}{\sqrt{k}} \left(C + \landauO\left(\tfrac{1}{k}\right)\right)
	\end{align*}
	for some $C > 0$. We further have 
	\begin{align*}
		&\phi_k(R) = \sqrt{\tfrac{\pi}{2 \alpha_k R}} \ee^{- \beta_k (R+\rho)} \left(\sqrt{\tfrac{\alpha_k}{\beta_k}} \cos(\alpha_k \rho) + \sqrt{\tfrac{\beta_k}{\alpha_k}} \sin(\alpha_k\rho) + \landauO\left(\tfrac{1}{k}\right) \right)
		\\ & \phi_k'(R) = \sqrt{\tfrac{\alpha_k \pi}{2 R}} \ee^{- \beta_k R+\rho} \left( \sqrt{\tfrac{\alpha_k}{\beta_k}} \sin(\alpha_k\rho) - \sqrt{\tfrac{\beta_k}{\alpha_k}} \cos(\alpha_k\rho) + \landauO\left(\tfrac{1}{k}\right) \right).
	\end{align*}
	Note that $\tfrac{\alpha_k}{\beta_k} = \sqrt{\tfrac{\alpha}{\beta}}$ is constant. 

	As $\phi_k$ above is the fundamental solution, assumption~\ref{ass:radial:fundamental_solutions} holds, and the second part of assumption~\ref{ass:radial:fundamental_solution:estimates} follows directly from the asymptotics. Let $\vartheta = \arctan\left(\tfrac{\alpha_k}{\beta_k}\right)$. Then 
	\begin{align*}
		&\phi_k(R) = \sqrt{\tfrac{\pi}{2 \alpha_k R} \left(\tfrac{\alpha_k}{\beta_k} + \tfrac{\beta_k}{\alpha_k}\right)} \ee^{- \beta_k (R+\rho)} \left( \sin\left( \alpha_k\rho + \vartheta \right) + \landauO\left(\tfrac{1}{k}\right) \right).
		\\ &\phi_k'(R) = - \sqrt{\tfrac{\alpha_k \pi}{2 R} \left(\tfrac{\alpha_k}{\beta_k} + \tfrac{\beta_k}{\alpha_k}\right)} \ee^{- \beta_k (R+\rho)} \left( \cos\left( \alpha_k\rho + \vartheta \right) + \landauO\left(\tfrac{1}{k}\right) \right).
	\end{align*}
	By assumption of the theorem on the values $T$ and $\vartheta$ we can write
	\begin{align*}
		\alpha_k \rho + \vartheta = 
		\frac{k m \pi}{2 n} + \frac{m \pi}{2 n} + \frac{l \pi}{n} - \xi
	\end{align*} 
	for some $l \in \Z$. Since $m, n$ are co-prime, the expression $\Zodd \ni k\mapsto \alpha_k \rho + \vartheta$ mod $\pi$ is $2n$-periodic and attains the $n$ values
	\begin{align*}
            \frac{\pi}{n} - \xi, \frac{2 \pi}{n} - \xi, \dots, \pi - \xi
	\end{align*}
	and no others. Further, none of these values are zeros of sine. This shows that also the first part of assumption~\ref{ass:radial:fundamental_solution:estimates} holds. In addition, we have 
	\begin{align*}
		\frac{\phi_k'(R)}{\phi_k(R)} = - \alpha_k \left( \cot\left( \tfrac{((k+1)m + 2l)\pi}{2n}-\xi \right) + \landauO\left(\tfrac{1}{k}\right)\right).
	\end{align*}
	Therefore for $\eps > 0$ sufficiently small we find infinitely many $k \in \Nodd$ such that
	\begin{align*}
	\alpha_k \rho + \vartheta = \pi - \xi \mod \pi \quad \mbox{ and } \quad 
		\frac{\phi_k'(R)}{k \phi_k(R)}
		= \omega \sqrt{\alpha} \cot(\xi) + \landauO\left(\tfrac{1}{k}\right)
		\geq \omega \sqrt{\delta} + \eps
	\end{align*}
	hold. This verifies \ref{ass:radial:nontrivial_sol:alt}. Finally, we have checked all assumptions of \cref{thm:radial:multiplicity} which provides existence of $T$-periodic solutions.
\end{proof}

\begin{proof}[Proof of \cref{thm:main_example}, Part 4]
	Lastly we discuss the step potential with slab geometry, i.e. \eqref{eq:slab:problem}.
	Like in part 3, we set $\alpha_k \coloneqq \omega k \sqrt{\alpha}$, $\beta_k \coloneqq \omega k \sqrt{\beta}$. Then the fundamental solutions $\tildePhi_k$ are (up to a constant) given by
	\begin{align*}
		\tildePhi_k(x) = \begin{cases}
			\widetilde A_k \sin(\alpha_k x) + \widetilde B_k \cos(\alpha_k x), & R < x < R+\rho, \\
			\ee^{- \beta_k x}, & x > R+\rho
		\end{cases}
	\end{align*}
	with 
	\begin{align*}
		\begin{pmatrix}
			\widetilde A_k \\ 
			\widetilde B_k
		\end{pmatrix}
		& = \begin{pmatrix}
			\sin(\alpha_k (R+\rho)) & \cos(\alpha_k (R+\rho)) \\
			\alpha_k \cos(\alpha_k (R+\rho)) & - \alpha_k \sin(\alpha_k (R+\rho))
		\end{pmatrix}^{-1} \begin{pmatrix}
			\ee^{- \beta_k (R+\rho)} \\
			- \beta_k \ee^{- \beta_k (R+\rho)} 
		\end{pmatrix} \\
		& = \ee^{- \beta_k (R+\rho)} \begin{pmatrix}
			\sin(\alpha_k (R+\rho)) - \tfrac{\beta_k}{\alpha_k} \cos(\alpha_k (R+\rho)) \\
			\cos(\alpha_k (R+\rho)) + \tfrac{\beta_k}{\alpha_k} \sin(\alpha_k (R+\rho))
		\end{pmatrix}.
	\end{align*}
	Therefore \ref{ass:slab:fundamental_solutions} holds and 
	\begin{align*}
		&\norm{\tildePhi_k}_{L^2([R+\rho, \infty))}^2 = \tfrac{1}{2\beta_k} \ee^{-2 \beta_k (R+\rho)},
		\\ &\norm{\tildePhi_k}_{L^2([R, R+\rho])}^2 =\tfrac{\rho}{2} \ee^{-2 \beta_k (R+\rho)} \left(1+\tfrac{\beta_k^2}{\alpha_k^2}\right) \left(1 + \landauO(\tfrac{1}{k}) \right)
	\intertext{so that}
		&\norm{\tildePhi_k}_{L^2([R, \infty))} = \ee^{- \beta_k(R+\rho)}\left(C + \landauO\left(\tfrac{1}{k}\right)\right)
	\end{align*}
	holds for some $C > 0$ as $k \to \infty$. 
	In particular, 
	\begin{align*}
		\tildePhi_k(R) 
		= \ee^{-\beta_k (R+\rho)} \left( \cos(\alpha_k \rho) + \tfrac{\beta_k}{\alpha_k} \sin(\alpha_k\rho) \right)
		= \ee^{-\beta_k (R+\rho)} \sqrt{1 + \tfrac{\beta}{\alpha}} \sin(\alpha_k \rho + \vartheta)
	\end{align*}
	and 
	\begin{align*}
		\tildePhi_k'(R) = - \alpha_k \ee^{-\beta_k (R+\rho)} \sqrt{1 + \tfrac{\beta}{\alpha}} \cos(\alpha_k \rho + \vartheta)
	\end{align*}
	with $\vartheta = \arctan\left(\sqrt{\tfrac{\alpha}{\beta}}\right)$. From here on we can argue almost identically as in the proof of part~3 for the verification of the conditions \ref{ass:slab:fundamental_solutions:estimates} and \ref{ass:slab:nontrivial_sol:alt}.
\end{proof}

\section{Numerical method}\label{sec:numerics}

In this section we provide details on the generation of \cref{fig:periodic:breathers,fig:step:breathers}. For simplicity, we only consider the radial geometry setting.

As discussed in \cref{sec:domain_restriction}, solutions $w$ to \eqref{eq:radial:problem} can be obtained from critical points $u$ of the functional $\efct$, see \eqref{eq:radial:energy}, and in particular from the minimizer of $\efct$. 
We numerically minimize $\efct \vert_Z$ over a finite dimensional space $Z$: $\efct(u) \approx \min \efct\vert_Z$. Then from $u$ we reconstruct an approximate breather $w$ using the formula \eqref{eq:solution_reconstruction}. 

Motivated by \cref{sec:approximation} we choose the ansatz space
\begin{align*}
	Z = \Bigl\{u \colon u(x, t) = \sum_{\substack{k \in \Zodd \\ \abs{k} \leq K}} f_k(x) e_k(t) \Bigm\vert f_k \in F, f_{-k} = \overline{f_k} \Bigr\},
\end{align*}
where $F$ is a (complex-valued) $1$d finite element space, which we have chosen to be the space of piecewise linear elements with equidistant nodes $0, \frac{R}{N}, \dots, \frac{(N-1)R}{N}, R$. 

The illustrations from \cref{fig:periodic:breathers,fig:step:breathers} are then obtained by choosing $K = 64, N = 128$ and using a MATLAB built-in function to solve the minimization problem. The code to generate them can be found in \cite{gitlab}.

	\section*{Acknowledgment}
	Funded by the Deutsche Forschungsgemeinschaft (DFG, German Research Foundation) – Project-ID 258734477 – SFB 1173. We thank Willy Dörfler (KIT) for providing his discontinuous Galerkin MATLAB-code on which our numerical simulations are based.

	\printbibliography
\end{document}